\documentclass[11pt]{article}
\usepackage[margin=1in]{geometry}
\usepackage{setspace} 
\usepackage{amsmath, amsthm, amssymb, mathtools, bm}
\usepackage{natbib}
\usepackage{bbm}
\usepackage{enumitem}
\usepackage{hyperref}
\usepackage{comment}
\usepackage{xcolor}
\usepackage{algorithm}
\usepackage{booktabs}
\usepackage[noend]{algpseudocode}
\title{Sharp Debiasing for Smooth Functional Estimation in Banach Spaces}
\author{Woonyoung Chang and Arun Kumar Kuchibhotla\\
Carnegie Mellon University}
\date{\today}

\newtheorem{theorem}{Theorem}[section]
\newtheorem{assumption}{Assumption}

\newtheorem{lemma}[theorem]{Lemma}
\newtheorem{proposition}[theorem]{Proposition}
\newtheorem{corollary}[theorem]{Corollary}
\theoremstyle{definition}

\newtheorem{remark}[theorem]{Remark}
\newtheorem{example}{Example}

\DeclarePairedDelimiterX{\mnorm}[1]
{\vvvert}
{\vvvert}
{\ifblank{#1}{\:\cdot\:}{#1}}

\DeclareMathOperator*{\argmin}{arg\,min}

\newcommand{\ip}[1]{\langle#1\rangle}

\newcommand{\Norm}[1]{\left\Vert#1\right\Vert}
\newcommand{\norm}[1]{\Vert#1\Vert}
\newcommand{\Abs}[1]{\left|#1\right|}
\newcommand{\abs}[1]{|#1|}
\newcommand{\Set}[1]{\left\{#1\right\}}
\newcommand{\set}[1]{\{#1\}}
\newcommand{\Real}{\mathbb R}

\newcommand{\A}{\mathfrak{A}}

\def\Var{\mbox{Var}}

\def\tr{{\rm tr}}
\def\op{{\rm op}}

\def\diag{\mbox{diag}}

\def\diag{\mbox{Diag}}

\def\cv{\mathbf c}

\def\ev{\mathbf e}

\def\gv{\mathbf g}
\def\hv{\mathbf h}

\def\gv{\mathbf g}

\def\lv{\mathbf l}

\def\tv{\mathbf t}

\def\xv{\mathbf x}

\def\Av{\mathbf A}
\def\Bv{\mathbf B}
\def\Cv{\mathbf C}
\def\Dv{\mathbf D}

\def\Fv{\mathbf F}

\def\Iv{\mathbf I}

\def\Nv{\mathbf N}

\def\Rv{\mathbf R}

\newcommand{\Ac}{\mathcal{A}}
\newcommand{\Bc}{\mathcal{B}}
\newcommand{\Cc}{\mathcal{C}}
\newcommand{\Dc}{\mathcal{D}}
\newcommand{\Ec}{\mathcal{E}}
\newcommand{\Fc}{\mathcal{F}}
\newcommand{\Gc}{\mathcal{G}}

\newcommand{\Ic}{\mathcal{I}}
\newcommand{\Jc}{\mathcal{J}}

\newcommand{\Lc}{\mathcal{L}}

\newcommand{\Nc}{\mathcal{N}}

\newcommand{\Pc}{\mathcal{P}}

\newcommand{\Rc}{\mathcal{R}}
\newcommand{\Sc}{\mathcal{S}}
\newcommand{\Tc}{\mathcal{T}}

\newcommand{\Xc}{\mathcal{X}}

\newcommand{\Bb}{\mathbb{B}}

\newcommand{\Eb}{\mathbb{E}}

\newcommand{\Nb}{\mathbb{N}}

\newcommand{\Pb}{\mathbb{P}}

\newcommand{\Sb}{\mathbb{S}}

\newcommand{\Ck}{\mathfrak{C}}

\newcommand{\Sk}{\mathfrak{S}}

\newcommand{\ak}{\mathfrak{a}}

\newcommand{\dk}{\mathfrak{d}}

\begin{document}

\maketitle
\begin{abstract}
    This paper studies the estimation of smooth functionals $f(\theta)$ of a mean parameter $\theta = \mathbb{E}_P[W]$ for a distribution $P$ on a general Banach space. We propose a cross-fitted estimator based on a single sample splitting and establish non-asymptotic moment bounds and Berry--Esséen bounds for both $m$-smooth and infinitely smooth functionals under the finite moment assumptions. Our framework is applied to precision matrix estimation and the inference of projection parameters in high-dimensional regression. In these Euclidean settings, the proposed estimators achieve asymptotic normality under the dimension regime $d \log^2(en) = o(n)$ without requiring any structural assumptions (e.g., sparsity). We discuss computational relaxations that enables polynomial-time implementation for a range of matrix functionals.
\end{abstract}
\section{Introduction}\label{sec:intro}
For a separable Banach space $(\Bb,\norm{\cdot})$ and a distribution $P$ on $\Bb$, let $P$ belong to a collection of distributions $\Pc$ such that the mean $\theta=\theta(P) = \mathbb{E}_{W\sim P}[W]$ exists and lies in an open set $\Theta \subseteq \Bb$. In this paper, we study the estimation of $f(\theta)$ for a smooth functional $f$, where $f: \Theta \to \mathbb{R}$ of smoothness $m=s+\rho$, ($s=\lceil m\rceil-1$,  $\rho\in(0,1]$) admits the local expansion:
\begin{equation}\label{eq:Frechet}
f(x + h) = f(x) + \sum_{k=1}^s \frac{\Dc ^k f(x)[h,\ldots,h]}{k!} + o(\|h\|^{s+\rho}).
\end{equation}

The scope of this problem is substantially broad. Several familiar Euclidean examples include the estimation of a linear contrast $c^\top\beta=c^\top\Sigma^{-1}\Gamma$ in linear regression where $\Sigma=\Eb[XX^\top]\in\Real^{d\times d}$, $\Gamma = \Eb[XY]\in\Real^d$, and $c\in\Real^d$ \citep{mammen1993bootstrap,ChangKuchibhotlaRinaldo2023ProjectionInference,AgrawallaBalasubramanianGhosal2026OnlineLSSGDInference}, and the estimation of the precision matrix functional $\eta_1^\top\Sigma^{-1}\eta_2$ for  $\eta_1,\eta_2\in\Real^d$ \citep{BodnarParolya2024RevivingPseudoInverses}. In a separable Hilbert space, this includes the estimation of functionals of a covariance operator. Here, for $\Sigma = \Eb[X\otimes X]$, the target $f(\Sigma)$ may be the trace functional \(\tr(\Sigma)\), the squared Hilbert--Schmidt norm \(\|\Sigma\|_{\rm HS}^{2}\), or a regularized inverse-type functional such as \(\langle u,(\Sigma+\lambda I)^{-1}v\rangle\) \citep{Koltchinskii2021,Koltchinskii2025Estimation,KoltchinskiiLi2026FunctionalEstimation}. A classic nonparametric example is the Gaussian sequence model $Y_j=\theta_j+n^{-1/2}\xi_j$ for $j\ge 1,$ where \(\xi_j\stackrel{\mathrm{i.i.d.}}{\sim}N(0,1)\). In this setting, the target of estimation is a smooth functional $f(\theta)$ for $\theta=(\theta_j)_{j\ge 1}$, often over Sobolev ellipsoids, with a canonical example being the quadratic functional $f(\theta)=\sum_{j=1}^\infty\theta_j^2$ for $\theta\in\ell_2$ \citep{BickelRitov1988IntegratedSquaredDerivatives,DonohoNussbaum1990MinimaxQuadratic,BalakrishnanKennedyWasserman2023StructureAgnosticLimits}.

A natural starting point to this problem is a plug-in  approach. A local $m(>1)$-smoothness yields the approximation,
\begin{equation}\label{eq:plugin-bias-intro}
f(\hat\theta)-f(\theta)
=
\Dc^1f(\theta)[\hat\theta-\theta]+o_P(\norm{\hat\theta-\theta}).
\end{equation} In classical regular parametric models, the linear term $\Dc^1f(\theta)[\hat\theta-\theta]$ is typically dominant and plug-in estimate $f(\hat\theta)$, built from efficient estimators $\hat\theta$ (e.g., MLEs), inherit asymptotic normality and efficiency for smooth functionals, often of parametric rates. This result is standard in the classical asymptotic and semi-parametric efficiency theory.

In high- and infinite-dimensional setting, however, the problem becomes considerably delicate. This stems from the fact that the linear term in \eqref{eq:plugin-bias-intro} is most often $\sqrt{n}$-consistent and asymptotically normal, but the remainder term need not be negligible at the target scale even if $\hat\theta$ is unbiased or consistent, especially in high-/infinite-dimensional spaces. Moreover, the well-known ``elbow phenomenon" in nonparametric minimax literature indicates that the optimal rates are not necessarily parametric but depend sharply on both functional smoothness and model/space complexity \citep{IbragimovNemirovskiiKhasminskii1987,CaiLow2005}, implying that estimating $f(\theta)$ can be a fundamentally different problem from estimating $\theta$. In such cases, a naive plug-in estimator is suboptimal.

\subsection{Related Works}
Recent works has been focusing on the `design' of estimators of $f(\theta)$ whose bias is reduced to scale in the target asymptotics. One common approach to bias reduction is through (iterative) bootstrap bias-correction \citep{HallMartin1988}. Significant contributions to this area have been made in a series of papers by Koltchinskii and coauthors regarding functionals of covariance operators \citep{Koltchinskii2018,Koltchinskii2021}, functionals in high-dimensional models \citep{Koltchinskii2022}, and in Gaussian shift/parametric location models \citep{KoltchinskiiZhilova2021aEfficient,KoltchinskiiZhilova2021bEstimationNormal,KoltchinskiiZhilova2021cEstimationLocation,KoltchinskiiWahl2023FunctionalLog-Concave}. For an estimator $\hat\theta$ of $\theta$ based on i.i.d.\ observations from $\Nc(\theta,\Sigma)$ in $\Bb$, they define an estimator-induced operator $\Tc f:=\Eb_\theta[f(\hat\theta)]$ and a bias operator $\Bc = \Tc - I$, i.e., $(\Bc f)(\theta):=\Eb_\theta[f(\hat\theta)]-f(\theta)$. A bias-reduced target is obtained by truncating the Neumann series for $(I+\Bc)^{-1}$; for $k\geq 0$, $ f_k := \sum_{j=0}^k(-1)^j\Bc^j f$ and estimate $f(\theta)$ by $f_k(\hat\theta)$ with the exact bias $\Eb_\theta[ f_k(\hat\theta)] -f(\theta)= (-1)^k\Bc^{k+1}f(\theta).$

To implement $\Bc$ from data, these works use a bootstrap chain. For $\vartheta\in\Theta$, write $\mathcal L_\vartheta(\hat\theta)$ for the sampling distribution of $\hat\theta$ under $P_\vartheta$ (assuming the identifiability of $\vartheta\mapsto P_\vartheta$). Starting from $\hat\theta^{(0)}:=\hat\theta$, generate $\hat\theta^{(j+1)}\mid \hat\theta^{(j)}\sim \mathcal L_{\hat\theta^{(j)}}(\hat\theta)$ from Monte Carlo sample $X^{(j+1)}_{1:b}\sim P_{\hat\theta^{(j)}}^{\otimes b}$ and setting $\hat\theta^{(j+1)}=\hat\theta(X_{1:b}^{(j+1)})$, so that conditional Monte Carlo averages along the chain approximate $\mathbb E_{\hat\theta^{(j)}}[\cdot]$ and yield computable surrogates for the iterates $\Bc^j f$. In structurally rich models, it is often possible to control the bias operator in an operator norm, $\|\Bc\|_{\Bb\to\Bb}\asymp \sqrt{\dk/n}$, for some effective dimension $\dk$, leading for $m$-smooth functionals to the composite rate $n^{-1/2}+(\sqrt{\dk/n})^m$, often minimax-optimal (and in some regimes asymptotically efficient) \citep{Koltchinskii2018,Koltchinskii2021,Koltchinskii2022,KoltchinskiiWahl2023FunctionalLog-Concave}. This rate exhibits a phase transition: the parametric rate is attainable only when $d=o(n^\gamma)$ for $\gamma\leq \frac{m-1}{m}$ \citep{KoltchinskiiZhilova2021bEstimationNormal,Koltchinskii2022,KoltchinskiiLi2026FunctionalEstimation}. Establishing such guarantees typically requires nontrivial concentration and, in sharper forms, Gaussian approximation for $f(\hat\theta)-\Eb_\theta[\hat\theta]$ and its iterated corrections, in a way that is often highly model-specific. Recently, \cite{Schaefer2024MobiusIteratedBootstrap} has clarified the operator-theoretic structure of iterative bootstrap bias correction as a combinatorial inversion problem on partition lattices.

A second route to bias reduction is to combine plug-in estimators computed at different scales so that the leading terms in their bias expansions cancel. Suppose that for a family of estimators $\{\hat\theta_{\tau} : \tau \in \mathcal T\}$ indexed by a scale parameter $\tau$, the corresponding plug-in estimators admit an expansion:
\[
\mathbb E_\theta[f(\hat\theta_{\tau})]=f(\theta)+\sum_{\ell=1}^{L} a_\ell(\theta)\tau^\ell + o(\tau^{L}).
\]
If one chooses scales $\tau_1,\dots,\tau_J$ and coefficients $c_1,\dots,c_J$ satisfying $\sum_{j=1}^J c_j = 1$ and $\sum_{j=1}^J c_j \tau_j^\ell = 0$ for $\ell=1,\ldots,r$, then the aggregated estimator
\[
f_{\mathrm{agg}}:=\sum_{j=1}^J c_j f(\hat\theta_{\tau_j})
\]
cancels the first $r$ bias terms and satisfies $\mathbb E_\theta[f_{\mathrm{agg}}]-f(\theta)=o(\tau^r).$

From this viewpoint, jackknife \citep{Quenouille1956NotesOnBias}, extrapolation \citep{Richardson1911DeferredApproach}, and related ensemble methods are instances of this bias-removal principle. There is no need to estimate the bias terms $a_\ell(\theta)$ directly, but instead one chooses weights so that the low-order monomials in the bias vanish which avoids explicit derivative calculations. In the binomial model, \citet{JiaoHan2020} show that the behaviors of delete-$k$ jackknife and $k$:th iterated bootstrap corrections can be compared through a common approximation framework. Recent high-dimensional results further show that such corrections can enlarge the regime in which plug-in inference is valid \citep{cattaneo2019two, lin2024worthwhilejackknifebreakingquadratic}. \citep{Koltchinskii2025Estimation} studied the estimation of functionals of covariance operator of Gaussian in a separable Banach space and showed the minimax optimal error rates are attained by such estimators.

A closely related formulation is extrapolation methods. Here one introduces an auxiliary perturbation level
$\lambda \ge 0$ and studies a family of estimators $\hat f(\lambda)$ satisfying $\mathbb E_\theta[\hat f(\lambda)]=f(\theta)+\beta_1(\theta)\lambda+\cdots+\beta_q(\theta)\lambda^q+o(\lambda^q).$ One then extrapolates this curve back to the target by either fitting a polynomial in $\lambda$ or, equivalently, by forming weights $c_j$ over a grid $\lambda_1,\dots,\lambda_J$ so that
\[
\hat f_{\mathrm{ext}}
=
\sum_{j=1}^J c_j \hat f(\lambda_j),
\qquad
\sum_{j=1}^J c_j \lambda_j^\ell = 0,\ \ell=1,\dots,r.
\]
In measurement-error problems, simulation-extrapolation (SIMEX) is the canonical example: one adds artificial noise, tracks the induced bias as a function of $\lambda$, and extrapolates to the zero-error target. \cite{CookStefanski1994SIMEX} and \cite{StefanskiCook1995SIMEXJackknife} made this connection explicit by showing that SIMEX can be interpreted as a measurement-error jackknife. 

A third approach, and the one closest in spirit to our estimator, is to debias by matching the higher-order terms in a functional expansion with stochastic corrections. At a formal level, the estimators are constructed so as to reproduce the polynomial leading terms in the expansion of $f(\theta)$ in \eqref{eq:Frechet} while keeping the higher-order terms sufficiently degenerate. This is closely connected to von Mises expansions \citep{vonMises1947} and, in semiparametric and nonparametric inference, to the higher-order influence function expansion \citep{RobinsLiTchetgenTchetgenVanderVaart2008}; see \citet{Kennedy2024SDRTDMLReview} for a recent survey. A central message of this literature is that minimax estimation of nonlinear functionals can often be achieved by combining higher-order pathwise expansions with suitable stochastic corrections and orthogonality/projection ideas, thereby preserving valid inference even in the presence of many nuisance parameters.

Among recent works, \citet{ZhouLiZhang2021HighOrderStatisticalFunctionalExpansion} is perhaps closest to our construction in the form of the correction terms: they develop a high-order statistical functional expansion in separable Hilbert spaces in which the corrections are realized through degenerate $U$-statistics, yielding an estimator with a small variance. Their exact order-$s$ estimator requires computing correction terms up to order $s$, and hence involves averaging over all distinct $k$-tuples at level $k$; as a result, the naive computation of their estimator may be prohibitive once $s$ is not very small. The later work \citet{ZhouLiZhang2023} obtains a substantially more explicit and computationally efficient estimator, but does so by exploiting the special additive structure $f(\theta)=\sum_{j=1}^d f_j(\theta_j)$ together with a Fourier-analytic construction, making it considerably more specialized than the present Banach-space setting. On the other hand, \citet{KoltchinskiiLi2026FunctionalEstimation} studies smooth functional estimation in general Banach spaces via Taylor expansion and sample splitting, and proves minimax-optimality and normal approximation results in a few parametric models. Their order-$s$ estimators are built from independent blockwise `base estimators' at each level $k=0,1,\dots,s$, requiring $1+\frac{s(s+1)}{2}$ base estimators in total. This reduces the effective sample size available to each block and may deteriorate the finite sample performance and efficiency. 

Finally, it is worth distinguishing our use of \emph{debiasing} from the de-biased inference in high-dimensional regression and graphical models under structural assumptions (e.g., sparsity). Works in that literature typically begins with a regularized estimator and corrects its leading regularization bias to recover valid asymptotics for a low-dimensional parameter. The bias we address here rather arises from the nonlinearity of the target functional $f$ itself in high-/infinite-dimensional regimes.


\subsection{Contributions}
Our first contribution is a general high-order debiasing framework for smooth functional estimation in Banach spaces based on a single sample split. The key structural point is that sample splitting is used not merely for convenience but to preserve the conditional degeneracy of the higher-order correction terms around an independent pilot. At the same time, the symmetrized cross-fitted construction does not incur a first-order efficiency loss. 

Our second contribution is a non-asymptotic statistical theory for our estimator under finite moment assumptions. We establish moment bounds and Berry--Esséen bounds for $m$-smooth functionals, and extend the theory to classes of infinitely differentiable functionals. For finite-order smoothness, our results imply asymptotic normality and efficiency under
\[
\dk=o(n), \qquad r_n=o(n^{-1/(2m)}),
\] up to constants in $m$, where $m$ denotes the order of smoothness of functional $f$, $\dk$ denotes the effective dimension of $\theta(P)$, and $r_n$ denotes the convergence rate of pilot estimators. For infinitely differentiable functionals, properly choosing the truncation order $s_n\asymp\log(n)$ as a function of $n$ yields parametric local behavior and asymptotic normality and efficiency under
\[
\dk=o\bigg(\frac{n}{\log^{2\alpha}(en)}\bigg), \qquad r_n=o\bigg(\frac{1}{\log^{\alpha-1}(en)}\bigg),
\] where $\alpha\geq 1$ denotes the order of Gevrey regularity (see Section~\ref{sec:infsmooth} for the definition of Gevrey class). This extrapolates the results for functionals of finite smoothness, and to the best of our knowledge, is a new type of result in the functional estimation literature.

Our third contribution is computational. In full generality,  exact evaluation of our estimator, which particularly make use of $U$-statistic of order $s\geq 1$ can be expensive, and in some applications with $s\asymp \log n$ it is super-polynomial. We show that for many classes of matrix functionals satisfying the product structure, the higher-order correction terms can be reorganized recursively. For such functionals, we propose a permutation-randomized estimator which can be computed within polynomial algebraic operations without sacrificing theoretical guarantees.

Finally, we turn our general theory into concrete problems in high-dimensional inference. For precision matrix functionals and projection parameters in linear regression, our estimators are asymptotically normal under the dimension regime,
\[
d\log^2(en)=o(n),
\]
without sparsity or other structural assumptions and only under fourth moment conditions. To the best of our knowledge, this is the most permissive currently available dimension regime for these problems under such weak moment assumptions.

\subsection{Notation and Preliminaries}Throughout this paper, $\Bb$ denotes a Banach space equipped with the norm $\|\cdot\|$. Let $(\Bb^*,\norm{\cdot}_*)$ be its dual and we write $\langle h, h^*\rangle:=h^*(h)$ for $h\in\Bb$ and $h^*\in\Bb^*$. We denote by $\mathcal{L}^k(\Bb)$ the space of bounded $k$-linear forms $T_k: \Bb \times \cdots \times \Bb \to \Real$, and by $\mathcal{L}_{\rm sym}^k(\Bb)$ the subspace of symmetric $k$-linear forms. For $T_k \in \mathcal{L}^k(\Bb)$, we denote its evaluation on $(h_1, \dots, h_k) \in \Bb^k$ by $T_k[h_1, \dots, h_k]$. We adopt the shorthand $h^{\otimes k}$ to represent the $k$-tuple $(h, \dots, h)$, such that $T_k[h^{\otimes k}] := T_k[h, \dots, h]$. More generally, for any $m \in \{0, \dots, k\}$, the notation $T_k[h^{\otimes (k-m)}, v_1, \dots, v_m]$ represents the evaluation where the first $k-m$ arguments are fixed as $h$. We denote the operator norm of a multilinear map $T \in \mathcal{L}^k(\Bb)$ as $$\|T_k\|_{\mathrm{S}} = \sup \{ |T_k[h_1, \dots, h_k]| : \|h_1\| = \dots = \|h_k\| = 1 \},$$ which defines a Banach space $(\Lc^k(\Bb),\norm{\cdot}_{\rm S})$. Given Banach space $(E,\norm{\cdot}_E)$, a subset $U\subseteq \Bb$, and a function $g:\Bb\to E$, we denote
\begin{equation*}
    \norm{g}_{{\rm Lip}_\rho(U)}:=\sup_{x,y\in U}\frac{\norm{g(x)-g(y)}_{E}}{\norm{x-y}^\rho},\quad \rho\in(0,1].
\end{equation*} We denote by $\widehat{\otimes}_{\pi}^k \Bb$ the completed $k$-fold projective tensor product of $\Bb$. It is well known that the dual space $(\widehat{\otimes}_{\pi}^k \Bb)^*$ is isometrically isomorphic to $\mathcal{L}^k(\Bb)$. Under this duality, for $T_k \in \mathcal{L}^k(\Bb)$ and $u \in \widehat{\otimes}_{\pi}^k \Bb$, we denote the dual pairing by $T_k[u]$. For a functional $f: \Bb \to \Real$, we denote its $k$-th order Fréchet derivative at $x \in \Bb$ by $\Dc^k f(x) \in \mathcal{L}_{\rm sym}^k(\Bb)$.

\subsection{Problem Setup and Methods}

Given i.i.d.\ observations $W_1, \dots, W_n \sim W\in\Bb$ and any $T_k\in\Lc_{\rm sym}^k(\Bb)$ ($k\geq 1$), we denote $T_k[\bar{U}^{(k)}(\cdot)]$ as the $k$-subsample average action on the shifted samples: for $b\in\Bb$,
\begin{equation*}T_k[\bar{U}^{(k)}(b)] := \binom{n}{k}^{-1} \sum_{1 \leq j_1 < \cdots < j_k \leq n} T_k\big[W_{j_1} - b, \dots, W_{j_k} - b\big].\end{equation*}
Here, the notation $\bar{U}^{(k)}(b)$ serves as a symbolic placeholder taking values in $\widehat{\otimes}_{\pi}^k \Bb$. Formally, we may view $T_k[\bar{U}^{(k)}(b)]$ as the evaluation of $T_k$ on the symmetric $k$-tensor $U$-statistic,
\begin{equation}\label{eq:U-function}
\bar{U}^{(k)}(b) = \binom{n}{k}^{-1} \sum_{1 \leq j_1 < \cdots < j_k \leq n} (W_{j_1} - b) \otimes_{\rm s} \cdots \otimes_{\rm s} (W_{j_k} - b),
\end{equation}
where $\otimes_{\rm s}$ denotes the symmetric tensor product, that is, $h_1\otimes_{\rm s}h_2 = \frac{1}{2}(h_1\otimes h_2+h_2\otimes h_1)$.

Our construction is motivated by the high order degenerate stochastic expansion of functionals \citep[Proposition~2.1]{ZhouLiZhang2021HighOrderStatisticalFunctionalExpansion}, which in turn can be generalized as follows.
\begin{proposition}\label{prop:1.1}For any $\tilde \theta\in \Theta$, the following deterministic identity holds:
\begin{equation}
    f(\tilde \theta) + \sum_{k=1}^s\frac{\Dc^kf(\tilde \theta)[\bar U^{(k)}(\tilde\theta)]}{k!} = f(\theta) + \sum_{k=1}^s\frac{\Dc^kf(\theta)[\bar U^{(k)}(\theta)]}{k!} - \mathrm{Rem}_s,\label{eq:prop:1.1}
\end{equation} where
\begin{equation*}
\mathrm{Rem}_s = \sum_{k=0}^s \frac{1}{k!} \mathcal{J}^{s-k} \Delta^{(s)} [(\theta - \tilde\theta)^{\otimes (s-k)}, \bar U^{(k)}(\theta)].
\end{equation*} Here, $\mathcal{J}$ is the Riemann-Liouville operator where $\Jc^0g=g(1)$ and $\Jc^m g = \int_0^1 g(t)(1-t)^{m-1}\,dt/\Gamma(m)$ for $m>0$, and $\Delta^{(s)}(t) = \Dc^s f(\tilde\theta + t(\theta - \tilde\theta)) - \Dc^s f(\tilde\theta)$, $t\in[0,1]$.\end{proposition}

Given $\tilde\theta$, the quantity in the left-hand side is fully attainable from the data. \cite{ZhouLiZhang2021HighOrderStatisticalFunctionalExpansion} considered taking the sample average, $\tilde\theta =n^{-1}\sum_{i=1}^nW_i$.

For any non-stochastic $T_k \in \mathcal{L}_{\rm sym}^k(\Bb)$ $(k\geq 1)$ and $\theta=\Eb[W]$, the scalar $T_k[\bar{U}^{(k)}(\theta)]$ can be viewed as completely degenerate $U$-statistic, yet involving unknown $\theta$, in a sense that for all $j\in[k]$,
\begin{equation*}
    \mathbb{E}\big( T_k[W_1-\theta,\ldots,W_k-\theta]\mid \{W_i:1\leq i\leq k, i\neq j\}\big) = 0 \quad\mbox{almost surely}.
\end{equation*} 
This follows from the fact that the multilinearity and boundedness of $T_k$ permit the interchange of the linear form operator and the Bochner expectation (see, e.g., \cite[Chapter 3.5]{delaPenaGine1999Decoupling}). Indeed, it follows from standard variance formula of $U$-statistics that $${\rm Var}\big(T_k[\bar U_n^{(k)}(\theta)]\big)=\binom{n}{k}^{-1}{\rm Var}\big( T_k[W_1-\theta,\ldots,W_k-\theta]\big).$$ While this appears to be an $O(n^{-k})$ decay, the variance also scales in the product of centered moments, which nonetheless remain controllable under regularity. Consequently, for a functional proxy $f^\dagger:\Bb\to\Real$,
\begin{equation*}
    f^\dagger(x) := f(x) + \sum_{k=1}^s\frac{\Dc^kf(x)[\bar U^{(k)}(x)]}{k!},
\end{equation*} it is clear that the data-dependent functional $f^\dagger(\theta)$ is unbiased, i.e., $\Eb[ f^\dagger(\theta)]=f(\theta)$, and has a small variance. Proposition~\ref{prop:1.1} can thus be interpreted as a first-order expansion of the proxy $f^\dagger(\theta)$ around the pilot estimator $\tilde\theta$. Notably, for a fixed $\tilde\theta$, the quantity $\mathcal{J}^{s-k} \Delta^{(s)} [(\theta - \tilde\theta)^{\otimes (s-k)}, \bar U^{(k)}(\theta)]$, which constitutes the remainder in \eqref{eq:prop:1.1}, is itself also a degenerate $U$-statistic of order $k$. This structure further motivates the use of sample splitting, allowing us to employ an independent estimator to maintain the (conditional) degeneracy.

To formally define our estimator, we consider a sample of $N = 2n$ independent and identically distributed observations $W_1, \dots, W_{2n}$ drawn from a distribution $P$. We employ a sample-splitting strategy by partitioning the data into two disjoint subsets of equal size: $\mathcal{S}_1 = \{W_1, \dots, W_n\}$ and $\mathcal{S}_2 = \{W_{n+1}, \dots, W_{2n}\}$. For a fixed $b \in \mathcal{B}$, let $\bar{U}^{(k)}(b)$ denote the $U$-statistic defined in \eqref{eq:U-function} computed using the observations in $\mathcal{S}_1$. Let $\hat\theta_{\mathcal{S}_2}$ be an estimator of $\theta$ constructed solely from the complementary subsample $\mathcal{S}_2$. We define the one-sided cross-fitted estimator as:\begin{equation*}\widehat{f}_s(\mathcal{S}_1, \mathcal{S}_2) = f(\hat\theta_{\mathcal{S}_2}) + \sum_{k=1}^s \frac{1}{k!} \mathcal{D}^k f(\hat\theta_{\mathcal{S}_2}) \left[ \bar{U}^{(k)}(\hat\theta_{\mathcal{S}_2}) \right].\end{equation*}The symmetric counterpart, $\widehat{f}_s(\mathcal{S}_2, \mathcal{S}_1)$, is obtained by interchanging the roles of $\mathcal{S}_1$ and $\mathcal{S}_2$. The proposed estimator is the cross-fitted average:\begin{equation} \label{eq:estimator_def}\widehat{f}_s = \frac{1}{2} \left( \widehat{f}_s(\mathcal{S}_1, \mathcal{S}_2) + \widehat{f}_s(\mathcal{S}_2, \mathcal{S}_1) \right).\end{equation}

\section{Results for \texorpdfstring{$m$}{m}-Smooth Functionals}\label{sec:msmooth}
Our results will be presented through the moments of the underlying distribution $P$ and the stochastic error of the preliminary estimators $\hat\theta_{\Sc_1}$ and $\hat\theta_{\Sc_2}$. We define the weak and strong variance as
\begin{equation*}
    \nu :=\nu(P)=\sup_{\norm{u}_*\leq 1}\Eb_{W\sim P}\big[|\langle W-\theta, u\rangle|^2\big]\quad\mbox{and}\quad V:=V(P)=\Eb_{W\sim P}\big[\norm{W-\theta}^2\big].
\end{equation*}Throughout, it will be assumed that $\hat\theta_{\Sc_1}$ and $\hat\theta_{\Sc_2}$ are identically distributed. We denote $\hat\theta_n:=\hat\theta_{n,P}$ be a generic estimator sharing this common distribution, such that $\hat\theta_{\Sc_1}\overset{d}{=}\hat\theta_{\Sc_2}\overset{d}{=}\hat\theta_n$. We will sometimes suppress the dependence on $P$.
\subsection{Moment Bounds}
The following theorem provides a non-asymptotic characterization of the $L_2$-error of our estimator. For a functional $f$ of smoothness $m=s+\rho$, ($s=\lceil m\rceil-1$,  $\rho\in(0,1]$), we recall
\begin{align*}
    &\norm{\Dc^k f (\theta)}_{\rm S} = \sup\Set{|\Dc^k f (\theta)[h_1,\ldots,h_k]|:\norm{h_1}=\cdots=\norm{h_k}=1},\qquad 1\leq k\leq s\\
    &\norm{\Dc^s f}_{{\rm Lip}_\rho(U)} = \sup_{x,y\in U}\frac{\norm{\Dc^s f(x)-\Dc^s f(y)}_{\rm S}}{\norm{x-y}^\rho}.
\end{align*}

\begin{theorem}\label{thm:moment-msmooth}
    For $s=\lceil m\rceil-1\geq 0$, we have
    \begin{align}\label{eq:thm:moment-msmooth:1}
    &\Eb^{1/2}\left[|\widehat{f}_s-f(\theta)|^2\right]\leq \sqrt{2}\norm{\Dc^s f}_{{\rm Lip}_\rho(\Theta)}\Eb^{1/2}\norm{\hat\theta_n-\theta}^{2m}+\sqrt{\frac{2\nu}{n}}\sum_{k=1}^s\norm{\Dc^k f(\theta)}_{\rm S}\bigg(\sqrt{\frac {V}{n}}\bigg)^{k-1}\nonumber\\
    &\quad+\sqrt{\frac{2e\nu}{n}}\norm{\Dc^s f}_{{\rm Lip}_\rho(\Theta)}\Eb^{1/2}\bigg[\bigg(\norm{\hat\theta_n - \theta}^2\bigvee\,\frac{V}{n}\bigg)^{m-1}\bigg]\mathbbm{1}(s\geq1).
    \end{align} Moreover, if $s\geq 1$ then 
    \begin{align}\label{eq:thm:moment-msmooth:2}
        &\Eb^{1/2}\left[N\big(\widehat{f}_s-f(\theta)-\Dc^1 f(\theta)[\bar W_N-\theta]\big)^2\right]\leq \sqrt{4n}\norm{\Dc^s f}_{{\rm Lip}_\rho(\Theta)}\Eb^{1/2}\norm{\hat\theta_n-\theta}^{2m}\nonumber\\
        &\qquad+\sqrt{4\nu}\sum_{k=1}^s\norm{\Dc^k f(\theta)}_{\rm S}\bigg(\sqrt{\frac {V}{n}}\bigg)^{k-1} +\sqrt{4e\nu}\norm{\Dc^s f}_{{\rm Lip}_\rho(\Theta)}\Eb^{1/2}\bigg[\bigg(\norm{\hat\theta_n - \theta}^2\bigvee\,\frac{V}{n}\bigg)^{m-1}\bigg],
    \end{align} where $\bar W_N = N^{-1}\sum_{i=1}^N W_i$ is  the full sample average.
\end{theorem}
The results in Thereom~\ref{thm:moment-msmooth} should be compared to those in \cite{ZhouLiZhang2021HighOrderStatisticalFunctionalExpansion} and \cite{KoltchinskiiLi2026FunctionalEstimation}. Following the notation in the latter, we let
\begin{equation}\label{eq:notation-kolt}
    \ak_p := \sup_{n\geq 1}\sup_{\norm{u}_*\leq 1}n\Eb^{2/p}[|\ip{\hat\theta_n-\theta,u}|^p]\quad\mbox{and}\quad \dk_p:=\sup_{n\geq 1}n\Eb^{2/p}[\norm{\hat\theta_n-\theta}^p],\quad p\geq 2.
\end{equation} Under this notation, our bound in \eqref{eq:thm:moment-msmooth:1} simplifies, up to constants depending on $s$, to
\begin{align}\label{eq:ourbound}
    \max_{1\leq k\leq s}\norm{\Dc^kf(\theta)}_{\rm S} \sqrt{\frac{\nu}{n}}\bigg(\sqrt{\frac{V}{n}}\bigg)^{k-1}+\norm{\Dc^s f}_{{\rm Lip}_\rho(\Theta)}\bigg(\sqrt{\frac{V\vee \dk_{2m}}{n}}\bigg)^{m}.
\end{align} Over the distributions class where $\nu\lesssim1$ and $V\vee\dk_{2m}\lesssim d$, the error scales as $1/{\sqrt{n}} + (\sqrt{d/n})^m$. This rate matches the minimax lower bound in $d$-dimensional Gaussian shift model \citep[Proposition~2.2 and 3.1]{KoltchinskiiLi2026FunctionalEstimation}. 

The $L_2$-moment bound derived in Theorem~2.1 of \cite{KoltchinskiiLi2026FunctionalEstimation} is given by:
\begin{equation}\label{eq:bound_KL}
    \max_{1\leq k\leq s}\sup_{x\in\Bb}\norm{\Dc^kf(x)}_{\rm S}\sqrt{\frac{\ak_2}{n}}\bigg(\sqrt{\frac{\dk_{2m}}{n}}\bigg)^{k-1} + \norm{\Dc^s f}_{{\rm Lip}_\rho(\Theta)}\bigg(\sqrt{\frac{\dk_{2m}}{n}}\bigg)^{m},
\end{equation} again up to constant in $s$. Our bound \eqref{eq:ourbound} offers two primary refinements. First, we relax the requirement of uniform control over functional derivatives across the entire space $\Bb$ to pointwise control at $\theta$. Second, \eqref{eq:ourbound} is strictly tighter in the worst case. As shown in Proposition~\ref{prop:lowerbound}, for any $n\geq 1$ and any measurable estimator $\hat\theta_n$, there exists a distribution $P$ such that $\nu \leq \ak_2$ and $V \leq \dk_2 \leq \dk_{2m}$. While in a favorable Gaussian setting, the sample mean satisfies $\ak_p\asymp \nu$ and $\dk_p\asymp V$ for all $p\geq 1$, we do not know the existence of estimator that simultaneously attains such directional and norm moment bounds in general.

\cite{ZhouLiZhang2021HighOrderStatisticalFunctionalExpansion} studied smooth functionals on Hilbert spaces without sample splitting. Theorem~3.2 of \cite{ZhouLiZhang2021HighOrderStatisticalFunctionalExpansion} yields a bound that scales as:
\begin{equation}\label{eq:bound_ZLZ}
        \max_{1\leq k\leq s}\norm{\Dc^kf(\theta)}_{\rm S} \sqrt{\frac{\nu}{n}}\bigg(\sqrt{\frac{V}{n}}\bigg)^{k-1} +\norm{\Dc^s f}_{{\rm Lip}_\rho(\Theta)}\Set{\bigg(\sqrt{\frac{V}{n}}\bigg)^{m}+\frac{\Eb\|W-\theta\|^{2m}}{n^{m-1/2}}}.
\end{equation} Under a finite $(2m)$:th moment of $\norm{W-\theta}$, we may take the sample means as the pilot estimators where the Burkholder-Rosenthal inequality \citep{Pinelis2012Rosenthal} implies that a remainder term, composing our bound \eqref{eq:ourbound}, satisfies:
\begin{equation*}
    \bigg(\sqrt{\frac{\dk_{2m}}{n}}\bigg)^{m}\lesssim_m\bigg(\sqrt{\frac{V}{n}}\bigg)^{m} + \frac{\Eb^{1/2}\norm{W-\theta}^{2m}}{n^{m-1/2}},
\end{equation*} where $\lesssim_m$ denotes the inequality up to constant in $m$. Thus, our result \eqref{eq:ourbound} is no worse than \eqref{eq:bound_ZLZ} and holds under more general Banach settings without requiring stronger moment conditions.
\begin{proposition}\label{prop:lowerbound}
    Let $(\Bb,\norm{\cdot})$ be a Banach space. Consider i.i.d. $W_1,\ldots,W_n\sim P$ with mean $\theta(P) = \Eb_{P}[W]$. Let $\Pc_1(\nu_0):=\{P:\sup_{\norm{u}_*\leq 1}\Eb_P|\ip{W-\theta(P),u}|^2\leq \nu_0\}$ and $\Pc_2(V_0):=\{P:\Eb_P\norm{W-\theta(P)}^2\leq V_0\}$ for some $\nu_0,V_0>0$. Then
    \begin{align}
        \inf_{n\geq 1}\inf_{\hat\theta_n}\sup_{P\in\Pc_1(\nu_0)}&\sup_{\norm{u}_*\leq 1}n\Eb_P|\ip{\hat\theta_n-\theta(P),u}|^2\geq\nu_0,\label{prop:lowerbound:1}\\
        \inf_{n\geq 1}\inf_{\hat\theta_n}\sup_{P\in\Pc_2(V_0)}&n\Eb_P\norm{\hat\theta_n-\theta(P)}^2\geq V_0,\label{prop:lowerbound:2}
    \end{align} where $\inf_{\hat\theta_n}$ denotes the infimum taken over all measurable estimators $\hat\theta_n=\hat\theta_n(W_1,\ldots,W_n)$.
\end{proposition}

\subsection{Asymptotic Normality}

In this section, we study the asymptotically normality and efficiency of our estimator. Towards this, we assume the estimator $\hat\theta_n$ of $\theta$ satisfies certain stochastic bounds on their norms:
\begin{assumption}\label{asmp:theta_tilde_general}
    There exist a non-increasing function $r_{n,P}:(0,1)\to\Real_+$ and a sequence $\delta_{n,P}\to0$ such that
    \begin{equation*}
        \Pb_P\bigg(\norm{\hat\theta_n-\theta}\leq r_{n,P}(\delta_{n,P})\bigg)\geq 1-\delta_{n,P}.
    \end{equation*}
\end{assumption}
\noindent Assumption~\ref{asmp:theta_tilde_general} allows the usage of robust estimators that require the knowledge of the target coverage.

\begin{example}[Hilbert Space]\label{ex:hilbert} Let $\Bb$ be a Hilbert space and let $\Cv = \Eb[(W-\theta)\otimes (W-\theta)]$ be a covariance operator. Then $\nu = \norm{\Cv}_{\rm op}$ and $V = {\rm tr}(\Cv)$. The ratio $V/\nu = {\rm tr}(\Cv)/\norm{\Cv}_{\rm op}=:\mathsf{r}(\Cv)$ is commonly referred to as the effective rank of $\Cv$, see also \cite{KoltchinskiiLounici2017Concentration} for extended notion to the covariance operator for Gaussian on a separable Banach space. The sample mean $\bar W_n = n^{-1}\sum_{i=1}^nW_i$ satisfies $\Eb\norm{\bar W_n-\theta}^2 = n^{-1}\Eb\norm{W-\theta}^2=\norm{\Cv}_{\rm op}\mathsf{r}(C)/n$, and thus it follows from simple Chebyshev inequality that for any $\delta\in(0,1)$,
\begin{equation*}
    \Pb\bigg(\norm{\bar W_n-\theta}\leq \norm{\Cv}_{\rm op}^{1/2}\sqrt{\frac{\mathsf{r}(\Cv)}{n\delta}}\bigg)\geq1- \delta.
\end{equation*} It is well-known that the sample mean may not be the best choice in heavy-tail setting, and there has been growing literature on robust mean estimators that attain sub-Gaussian tail bounds only under the existence of the second moment or weaker. For instance, the median-of-mean tournament estimate $\hat\theta_{n,\delta}$ (for given $\delta\in(0,1)$) attains
\begin{equation*}
    \Pb\bigg(\norm{\hat\theta_{n,\delta}-\theta}\lesssim \norm{\Cv}_{\rm op}^{1/2}\sqrt{\frac{\mathsf{r}(\Cv)+\log(e/\delta)}{n}}\bigg)\geq1- \delta,
\end{equation*} where $\lesssim$ hides universal constants \citep{LugosiMendelson2019SubGaussianMean}. One can choose a vanishing sequence $\delta_n$ for Assumption~\ref{asmp:theta_tilde_general}.
\end{example}

We shall refer the following moment ratio as the \emph{effective dimension} of the distribution, which, as written, is defined independently of functional $f$:
\begin{equation*}
    d_{\rm eff}=d_{{\rm eff}}(P) :=\frac{V(P)}{\nu(P)}.
\end{equation*} 
\begin{example}[Smooth Banach Space]\label{ex:smBanach} Let $\Bb$ be a separable and $(2,\beta)$-smooth\footnote{$(\Bb,\norm{\cdot})$ is $(2,\beta)$-smooth if $\norm{x+y}^2+\norm{x-y}^2\leq 2(\norm{x}^2+\beta\norm{y}^2)$ for any $x,y\in\Bb$.} Banach space. The sample mean $\bar W_n$ satisfies $\Eb\norm{\bar W_n-\theta}^2 \leq (4\beta^2/n)\Eb\norm{W-\theta}^2$ \citep{Pisier1975MartingalesUniformlyConvex}. This is an example of Nemirovski's inequality \citep{DumbgenVanDeGeerVeraarWellner2010Nemirovski}. The Chebyshev inequality then yields
\begin{equation*}
    \Pb\bigg(\norm{\bar W_n-\theta}\leq 4\beta\nu^{1/2}\sqrt{\frac{d_{\rm eff}}{n\delta}}\bigg)\geq1-\delta, \quad \delta\in(0,1).
\end{equation*} For a given $\delta\in(0,1)$, robust mean estimator $\hat\theta_{n,\delta}$ can be designed to have a sharper tail inequality, such as, trimmed mean estimators \citep{WhitehouseChuggMartinezTaboadaRamdas2025MeanBanachInfiniteVariance} or geometric median-of-means estimators \citep{Minsker2015GeometricMedianBanach}, as
\begin{equation*}
    \Pb\bigg(\norm{\hat\theta_{n,\delta}-\theta}\lesssim \beta\nu^{1/2}\sqrt{\frac{d_{\rm eff}\log(e/\delta)}{n}}\bigg)\geq1- \delta.
\end{equation*}

\end{example}

More examples of the robust covariance-operator estimators in Hilbert and Euclidean settings, including the particular constructions used in our applications to the precision-matrix and regression functionals, are collected in Section~\ref{sec:appl:precision}.

For a number $\bar\nu$ and sequences $\set{\dk>0},\set{r_n>0}$, $\{\delta_n\in(0,1)\}$, denote the class of distribution which satisfies Assumption~\ref{asmp:theta_tilde_general} as
\begin{align*}
    \Pc(\bar\nu,\dk,r_n,\delta_n)=\bigg\{P\in\Pc:~&\nu(P)\leq \bar \nu,~d_{{\rm eff}}(P)\leq \dk,~\delta_{n,P}\leq \delta_n,~r_{n,P}(\delta_{n,P})\leq r_n\bigg\}.
\end{align*}

\begin{theorem}\label{thm:normality-msmooth} For $s=\lceil m\rceil-1\geq 1$ and $P\in \Pc(\bar\nu,\dk,r_n,\delta_n)$, suppose that for some $U\supseteq B(\theta(P),r_n)$ and a constant $C_0>0$, $\norm{\Dc^s f}_{{\rm Lip}_\rho(U)}\leq C_0,$ and $\max_{2\leq k\leq s}\norm{\Dc^kf(\theta(P))}_{\rm S}\leq C_0.$ For any such $f$, let $\sigma_f^2 = \mathrm{Var}_P\big(\Dc^1f(\theta(P))[W-\theta(P)]\big)$ and $\mu_{3,f} = \Eb_P\big[|\Dc^1f(\theta(P))[W-\theta(P)]|^3\big]$. Then
    \begin{align*}
        &\sup_{t\in\Real}\Abs{\Pb_P\left(\frac{\sqrt{N}\,(\widehat{f}_s - f(\theta(P)))}{\sigma_f}\leq t\right)-\Phi(t)}\\
        &\quad\lesssim_{\bar\nu,s} \frac{\mu_{3,f}}{\sigma_f^3\sqrt{n}}+\frac{\sqrt{n}r_n^{m}}{\sigma_f} + \frac{1}{\sigma_f^{2/3}}\bigg[\bigg(\frac{\dk}{n}\bigg)^{1\wedge(m-1)} + r_n^{2m-2}\bigg]^{1/3} + \delta_{n}.
    \end{align*}
\end{theorem}
The third-moment condition $\mu_{3,f}<\infty$ is only introduced for a parametric leading term; asymptotic normality
still holds under finite second moments, and one can obtain $n^{-\delta/2}$ leading term under $(2+\delta)$:th moments \citep{Bikelis1966RemainderCLT,Shevtsova2020LowerBoundsNonuniformCLT}.

If $\dk=o(n)$, $r_n=o(n^{-1/(2m)})$, and $\delta_n=o(1)$, then $\widehat{f}_s$ is asymptotically normal with efficient variance $\sigma_f^2$. Under the moment bound on the pilot \eqref{eq:notation-kolt}, one may take $r_n(\delta)=\delta^{-\frac{1}{2m}}\sqrt{\dk_{2m}/n}$ (a faster rate is viable; see Example~\ref{ex:smBanach}), so that with $\delta=\delta_n\to0$ arbitrarily slowly, $\dk_p=O(n^\gamma)$ and $\gamma<\frac{m-1}{m}$ suffice for asymptotic normality. \cite[Corollary~2.3]{KoltchinskiiLi2026FunctionalEstimation} further proved the $2$-Wasserstein convergence under additional assumptions that $\ak_{2m}=O(1)$ and that the projections of the pilot estimator $\langle \hat\theta_n-\theta,u\rangle$ converges to normal distribution in $2$-Wasserstein distance, uniformly over $\|u\|_*\le1$. Their results are stated for functionals satisfying $\|\Dc^1 f\|_{\Cc^{s-1}(U)}\le C_0$, where for $r>0$ the H\"older $\Cc^r$-norm on an open set $U\subseteq\Theta$ is
\[
\|g\|_{\Cc^{r}(U)}:=\max\Big\{\max_{0\le k\le \lceil r\rceil-1}\sup_{x\in U}\|\Dc^k g(x)\|_{\rm S},\ \|\Dc^s g\|_{{\rm Lip}_{r+1-\lceil r\rceil}(U)}\Big\},
\]which is stronger than the requirement imposed in Theorem~\ref{thm:normality-msmooth}. In a Hilbert-space setting, $d_{\rm eff}$ corresponds to the effective rank of the covariance operator (Example~\ref{ex:hilbert}). \cite[Theorem~4.1]{ZhouLiZhang2021HighOrderStatisticalFunctionalExpansion} obtain asymptotic normality and efficiency under $d_{\rm eff}=o(n^{1-1/m})$, without sample splitting.

\section{Results for Infinitely Differentiable Functionals}\label{sec:infsmooth}
The results presented in the preceding section, as well as those in the existing literature on high order debiasing of smooth functionals \citep{Koltchinskii2021, Koltchinskii2022, ZhouLiZhang2021HighOrderStatisticalFunctionalExpansion, KoltchinskiiLi2026FunctionalEstimation}, indicate that asymptotic normality for a functional $f(\theta)$ requires regularity beyond the mere consistent estimation of $\theta$. Specifically, the order of smoothness $m$ must scale with the relative growth of effective dimension to the sample size or the rate of the preliminary estimator. In this section, we demonstrate that these can be naturally extrapolated to certain classes of infinitely differentiable functionals, allowing for efficient estimation for even broader distributional classes.

For $\alpha\ge 1$, $R>0$ and $U\subseteq \Theta$, a functional $f\in \Cc^\infty(U)$, define a semi-norm
\begin{equation}\label{eq:gevrey}
    \Norm{f}_{\Gc^\alpha(U),R}:=\sup_{x\in U}\sup_{k\geq 0}\frac{\norm{\Dc^k f(x)}_{\rm S}}{R^k (k!)^\alpha}.
\end{equation} A functional $f$ is said to be \emph{Gevrey of order $\alpha$ on $U$} if there exists $R$ such that $\Norm{f}_{\Gc^\alpha(U),R}<\infty$, that is, we write
\begin{equation*}
    f\in \Gc^\alpha(U):=\bigcup_{R>0}\set{f\in \Cc^\infty(U): \Norm{f}_{\Gc^\alpha(U),R}<\infty}.
\end{equation*}
This is the natural extension of the usual finite-dimensional Gevrey definition to Banach spaces, with Fréchet derivatives in place of partial derivatives. For the purposes of this paper, we work with a global Roumieu-type bound on $U$, which keeps notation simple and is sufficient for our arguments. Related definitions in the ultra-differentiable literature are often stated locally on relatively compact sets \citep{Komatsu1979,KrieglMichorRainer2015,TeofanovTomicZigic2024ExtendedGevrey}.

It is immediate that $\Gc^\alpha(U)\subset\Gc^{\alpha'}(U)$ whenever $1\leq \alpha\leq \alpha'$. For $\alpha=1$, if $f\in \Gc^1(U)$ with $\norm{f}_{\Gc^1(U),R}<\infty$, then for every $x\in U$, the Taylor expansion is convergent, and 
\begin{equation*}
    f(x+ h) = f(x) + \sum_{k=1}^\infty\frac{\Dc^kf(x)[h^{\otimes k}]}{k!},
\end{equation*} for all $h\in\Bb$ such that $[x,x+h]\subset U$ and $\norm{h}<1/R$. In the finite-dimensional case, the Gevrey hierarchy provides a refinement from locally analytic functions toward the class of infinitely differentiable functions, though this interpolation is not exhaustive, i.e., $\cup_{\alpha\geq 1}\Gc^\alpha(U)\subsetneq \mathcal{C}^\infty(U)$.

\subsection{Moment Bounds}
For an infinitely differentiable functional $f$, the truncation level $s$ in our expansion can be chosen to balance the stochastic and remainder terms. We obtain sharp local moment bounds for a suitable choice of $s$ with respect to the Gevrey quantification. For $f\in\Gc^\alpha(U)$ with $U$ convex, then \eqref{eq:gevrey} yields explicit bounds on both operator norms and Lipschitz moduli of the Fr\'echet derivatives. In particular, for every $s\ge 0$,
\begin{equation*}
    \norm{\Dc^s f}_{{\rm Lip}_1(U)}\leq \sup_{x\in U} \norm{\Dc^{s+1} f(x)}_{\rm S}\leq \Norm{f}_{\Gc^\alpha(U),R} R^{s+1}\{(s+1)!\}^\alpha.
\end{equation*}
Theorem~\ref{thm:moment-gevrey} below shows that with an appropriate truncation level, the local $L_2$ error of $\widehat{f}_s$ attains a parametric rate.

\begin{theorem}\label{thm:moment-gevrey} For some $r_n>0$, $P\in\Pc$ and $f\in\Gc^\alpha(U)$ with $B(\theta(P),r_n)\subseteq U$ and $\norm{f}_{\Gc^\alpha(U),R}<\infty$, suppose the following condition holds.
\begin{equation}\label{eq:cond:gevrey}
    R\,r_n\log^{\alpha-1}(en)\vee R\log^{\alpha}(en)\sqrt{\frac{V(P)}{n}}\leq e^{-1}.
\end{equation} Denote the event $\Ec = \set{\max\{\norm{\hat\theta_{\Sc_1}-\theta(P)},\norm{\hat\theta_{\Sc_2}-\theta(P)}\}\leq r_n}$. Then, our estimator $\widehat{f}_s$ with $s=\lfloor\log(en)\rfloor$ satisfies that for a universal constant $C>0$,
    \begin{align*}
        &\Eb_P^{1/2}\Big[|\widehat{f}_s-f(\theta(P))|^2\mathbbm{1}(\Ec)\Big]\leq C\,2^{\alpha}(R\vee R^2)\norm{f}_{\Gc^\alpha(U),R}\bigg(\sqrt{\frac{\nu(P)}{n}}+\frac{r_n}{n}\bigg).\\
       & \Eb_P^{1/2}\Big[N\big(\widehat{f}_s-f(\theta(P))-\Dc^1 f(\theta(P))[\bar W_N-\theta(P)]\big)^2\mathbbm{1}(\Ec)\Big]\\
       &\quad\quad\leq C\,2^{\alpha}(R\vee R^2)\norm{f}_{\Gc^\alpha(U),R}\bigg(\sqrt{\frac{\nu(P) V(P)}{n}}+\frac{r_n}{\sqrt{n}}+\frac{\sqrt{\nu(P)}}{n^{3/4}}\bigg).
    \end{align*}
\end{theorem}
The results in Theorem~\ref{thm:moment-gevrey} are local and suffice to derive asymptotic normality and efficiency in the next section. The particular choice $s=\lfloor\log(en)\rfloor$ is convenient but not essential. More generally, any diverging sequence $s_n\to\infty$ that satisfies the growth conditions below yields parametric control of local moments:
\begin{align*}
    &\frac{\log(en)}{\log\frac{1}{Rr_n}}\ll s_n \ll (Rr_n)^{-\frac{1}{\alpha-1}}\wedge\bigg(R\sqrt{\frac{V}{n}}\bigg)^{-\frac{1}{\alpha}},&&\alpha>1,\\
    &\frac{\log(en)}{\log\frac{1}{Rr_n}}\ll s_n \ll \bigg(R\sqrt{\frac{V}{n}}\bigg)^{-\frac{1}{\alpha}},&&\alpha=1.
\end{align*}

\subsection{Asymptotic Normality}
\begin{theorem}\label{thm:normality-gevrey}For $P\in\Pc$, suppose Assumption~\ref{asmp:theta_tilde_general} holds and let $r_n = r_{n,P}(\delta_{n,P})$. Let $f\in\Gc^\alpha(U)$ with $U\supseteq B(\theta(P),r_n)$ and let $R$ be such that $\norm{f}_{\Gc^\alpha(U),R}<\infty$. If $r_n$ satisfies the condition~\eqref{eq:cond:gevrey} eventually as $n\to\infty$, then our estimator $\widehat{f}_s$ with $s=\lfloor\log(en)\rfloor$ is asymptotically normal, i.e., $\sqrt{N}(\widehat{f}_s - f(\theta))\xrightarrow{D}\Nc(0,\sigma_f^2)$. Moreover, under \eqref{eq:cond:gevrey}, there exists a universal constant $C>0$,
\begin{align*}
    &\sup_{t\in\Real}\Abs{\Pb_P\left(\frac{\sqrt{N}\,(\widehat{f}_s - f(\theta(P)))}{\sigma_f}\leq t\right)-\Phi(t)}\leq C\bigg[ \frac{\mu_{3,f}}{\sigma_f^3\sqrt{n}}+\frac{\norm{f}_{\Gc^\alpha(U),R}Rr_n}{\sigma_f\sqrt{n}}\\
    &\quad +\bigg(\frac{2^\alpha\norm{f}_{\Gc^\alpha(U),R}R^{2}}{\sigma_f}\bigg)^{2/3}\bigg\{\bigg(\frac{\nu(P) V(P)}{n}\bigg)^{1/3} + \frac{\nu(P)^{1/3}}{\sqrt{n}}\bigg\}\bigg]+2\delta_{n,P}.
\end{align*}
\end{theorem} A finite third moment $\mu_{3,f}<\infty$ is not essential; for bounds under weaker moment assumptions, see the discussion around Theorem~\ref{thm:normality-msmooth}.

Under $P\in\Pc$ with $\nu(P)\lesssim1$ and $d_{\rm eff}(P)\lesssim \dk$, our estimator is asymptotically normal if
\begin{equation*}
    r_n = o(1/\log^{\alpha-1}(en))\quad\mbox{and}\quad \dk = o(n/\log^{2\alpha}(en)),
\end{equation*} with KS distance of order at most $1/\sqrt{n}+(\dk/n)^{1/3}$. Notably, for $\alpha=1$ (locally analytic functionals in Euclidean setting), mere consistency with an arbitrarily slow rate $r_n\to 0$ suffices for $\sqrt{n}$-consistency and asymptotic normality. We here remark that our analysis is non-asymptotic, and the bound in Theorem~\ref{thm:normality-gevrey} applies to functionals of finite smoothness of order $\lfloor\log(en)\rfloor+1$ through the truncated Gevrey semi-norm:
\begin{equation*}
    \Norm{f}^\dagger_{\Gc^\alpha(U),R}:=\sup_{x\in U}\sup_{0\leq k\leq \lfloor\log(en)\rfloor+1}\frac{\norm{\Dc^k f(x)}_{\rm S}}{R^k (k!)^\alpha}.
\end{equation*}

\section{Applications}\label{sec:appl}

\subsection{Precision Matrix Estimation}\label{sec:appl:precision}
This section studies the estimation of precision matrix functionals under various distributional assumptions. Consider a random vector $X$ in $\Real^d$ with a finite second moment and assume the population gram matrix $\Sigma =\Eb[XX^\top]\in\Real^{d\times d}$ be positive definite. We aim to estimate the functional \begin{equation}\label{eq:precisionmatrixfunctional} \omega := \omega(\Sigma) = \eta_1^\top \Sigma^{-1} \eta_2, \quad \text{for any }~\eta_1, \eta_2 \in \Real^d\setminus\{0_d\}. \end{equation} Suppose that we observe $N=2n$ i.i.d. random vectors $X_1,\ldots,X_{2n}$ with finite directional moments.

\begin{enumerate}[label=\textbf{(M)(X)}]
    \item\label{asmp:momentX} There exist constants $K_x\geq1$ and $q_x\geq 4$ such that $\sup_{u\in\Sb^{d-1}}\Eb[|u^\top\Sigma^{-1/2}X|^{q_x}]\leq K_x^{q_x}$.
    \end{enumerate} 
Let $\Bb=\Sb\Pb^{d}$, then the map $\Sigma\mapsto\omega(\Sigma)$ is infinitely continuously differentiable on $\Bb$. For $H,h_1,\ldots,h_k\in\Bb$ and $k\geq 1$, the $k$:th order Fréchet derivative $\Dc^k\omega(h):\Bb^{\otimes k}\to \Real$ is given by, 
\begin{equation}\label{eq:precisionmatrix_Fréchetderivatives}
    \Dc^k\omega(H)[h_1,\ldots, h_k] = (-1)^{k}\sum_{\pi\in\Sk_k}\eta_1^\top H^{-1}h_{\pi(1)}H^{-1}\cdots H^{-1}h_{\pi(k)}H^{-1}\eta_2,
\end{equation} where the summation is taken over all permutations in $\Sk_k=\{\pi:[k]\to[k]\}$. Since $\norm{\Dc^k\omega (H)}_{\rm S}\leq k!\lambda_{\rm min}^{-k-1}(H)$, $\omega$ is Gevrey of order $1$ on the subsets of positive definite matrices whose eigenvalues are bounded below. For any $d\times d$ matrix $M$, the tensor $U$-statistic component is of the form:
\begin{equation*}
    \bar U^{(k)}(M) = \binom{n}{k}^{-1} \sum_{1 \leq j_1 < \cdots < j_k \leq n}(X_{j_1}X_{j_1}^\top-M)\otimes_{\rm s}\cdots\otimes_{\rm s} (X_{j_k}X_{j_k}^\top-M),\quad k\geq 1.
\end{equation*} 

We assume that there is a matrix estimator for $\Sigma$ with a reasonable rate of convergence.
\begin{enumerate}[label=\textbf{(P1)},leftmargin=1.2cm]
    \item\label{asmp:Sigmapilot} There exists an estimator $\tilde\Sigma:=\tilde\Sigma(X_1,\ldots,X_n)\in\Real^{d\times d}$, a non-increasing function $r_n:(0,1)\to\Real^+$, and a sequence $\delta_n$ with $\delta_n\to0$ as $n\to\infty$ such that
    \begin{equation*}
        \Pb\bigg(\norm{\Sigma^{-1/2}\tilde\Sigma\Sigma^{-1/2}-I_d}_{\rm op}\geq r_n(\delta_n)\bigg)\leq \delta_n.
    \end{equation*}
\end{enumerate} 
Given an estimator $\tilde\Sigma_2=\tilde\Sigma(X_{n+1},\ldots,X_{2n})$, we define a matrix estimator $\tilde\Omega(\Sc_1,\Sc_2)$ be such that
\begin{equation*}
    \eta_1^\top \big[\tilde\Omega(\Sc_1,\Sc_2)\big] \eta_2 = \eta_1^\top \tilde\Sigma_2^{-1}\eta_2 + \sum_{k=1}^{s_n}\frac{\Dc^k\omega(\tilde\Sigma_2)[\bar U^{(k)}(\tilde\Sigma_2)]}{k!},\quad\forall\eta_1,\eta_2\in\Real^d
\end{equation*} for a potentially diverging sequence $\{s_n:n\geq 1\}.$ Similarly, we define $\tilde\Omega(\Sc_2,\Sc_1)$ by swapping the roles of $\Sc_1$ and $\Sc_2$, with $\tilde\Sigma_1=\tilde\Sigma(X_1,\ldots,X_n)$. Our final estimator is given as
\begin{equation}\label{eq:precisionmatrixest.}
    \tilde\Omega = \frac{1}{2}\big(\tilde\Omega(\Sc_1,\Sc_2) + \tilde\Omega(\Sc_2,\Sc_1)\big).
\end{equation}
We remark that matrices $\tilde\Sigma_1,\tilde\Sigma_2$ are invertible with a high probability under \ref{asmp:Sigmapilot}, and thus our estimator is well defined, again with a high-probability.

Let $\sigma_\omega^2 = {\rm Var}(\eta_1^\top\Sigma^{-1}XX^\top\Sigma^{-1}\eta_2)$ and the normalized variance by $\bar\sigma_\omega^2 = \sigma_\omega^2/(\norm{\eta_1}_{\Sigma^{-1}}^2\norm{\eta_2}_{\Sigma^{-1}}^2)$. Theorem below provides sufficient conditions for the asymptotic normality: for all $\eta_1,\eta_2\in\Real^d\setminus\{0_d\}$,
\begin{equation}\label{eq:asymnor_PM}
        \frac{\sqrt{N}(\eta_1^\top\tilde\Omega\,\eta_2-\eta_1^\top\Sigma^{-1}\eta_2)}{\sigma_\omega}\xrightarrow{D}\Nc(0,1),\quad\mbox{as}\quad N\to\infty,
\end{equation} and a precise Berry--Esséen bound for the Kolmogorov-Smirnov (KS) distance,
\begin{equation*}
    {\rm BE}_n(\eta_1^\top\tilde\Omega\,\eta_2):=\sup_{t\in\Real}\bigg|\Pb\bigg(\frac{\sqrt{N}(\eta_1^\top\tilde\Omega\,\eta_2-\eta_1^\top\Sigma^{-1}\eta_2)}{\sigma_\omega}\leq t\bigg)-\Phi(t)\bigg|.
\end{equation*}


\begin{theorem}\label{thm:PM1}
    Suppose that Assumption~\ref{asmp:momentX} and \ref{asmp:Sigmapilot} hold. Let $r_n^*=r_n(\delta_n)$ for $n\geq 1.$
    \begin{enumerate}
        \item For a fixed $s_n=s\geq 1$, if $d=o(n)$ and $\sqrt{n}(r_n^*)^{s+1}=o(1)$, then asymptotic normality \eqref{eq:asymnor_PM} holds. Moreover, let $\bar q_x = q_x\wedge 6$. If $(2r_n^*)^2\vee (64K_x^4s^2d/n)\leq e^{-2}$, then there exists a constant $\Ck=\Ck(s)>0$ such that 
        \begin{align}\label{eq:thm:PM1:1}
        {\rm BE}_n(\eta_1^\top\tilde\Omega\,\eta_2)\leq \frac{\Ck K_x^{\bar q_x}}{(\bar \sigma_\omega\wedge 1)^{\bar q_x/2}}\bigg[\frac{1}{n^{\bar q_x/4-1}}+\bigg(\frac{d}{n}\bigg)^{1/3}+\sqrt{n}(r_n^*)^{s+1} + (r_n^*)^{2s/3}\bigg]+2\delta_n.
        \end{align}
        \item Let $s_n=\lfloor \log(en)\rfloor$. If $d\log^2(en) = o(n)$ and $r_n^*\leq (2e)^{-1}$ eventually, then \eqref{eq:asymnor_PM} holds. Moreover, if $(2r_n^*)^2\vee (64K_x^4d\log^2(en)/n)\leq e^{-2}$, then there exists a universal constant $\Ck>0$ such that
    \begin{equation}\label{eq:thm:PM1:2}
    {\rm BE}_n(\eta_1^\top\tilde\Omega\,\eta_2)\leq \frac{\Ck K_x^{\bar q_x}}{(\bar \sigma_\omega\wedge 1)^{\bar q_x/2}}\bigg[\frac{1}{n^{\bar q_x/4-1}}+\bigg(\frac{d}{n}\bigg)^{1/3}\bigg]+2\delta_n.
    \end{equation}
    \end{enumerate}
\end{theorem}
The normalized variance $\bar\sigma_\omega^2$ admits the lower bound
\begin{equation}\label{eq:lowerbound_normalizedvariance}
    \bar\sigma_\omega^2
    \ge
    \inf_{u,v\in\Sb^{d-1}}
    \Var\!\big(u^\top \Sigma^{-1/2}XX^\top\Sigma^{-1/2}v\big).
\end{equation}
Therefore, the Berry--Esséen  bounds in Theorem~\ref{thm:PM1} can be made uniform over all directions $\eta_1,\eta_2\in\Real^d\setminus\{0_d\}$ with an additional condition. For $X\sim\Nc(\mu,\Sigma)$, the right-hand side of \eqref{eq:lowerbound_normalizedvariance} equals 1 for $d\geq 3$, and is no smaller than $1$ for $d=1,2$.

The dimension requirement for asymptotic normality and the Berry--Esséen  bound are driven by the convergence rate of a pilot covariance estimate $\tilde\Sigma_2$ in the operator norm. Under Assumption~\ref{asmp:momentX}, one may use the robust Gram/covariance estimator of \citet[Theorem~1.3]{OliveiraRico2024ImprovedCovariance} with a sub-Gaussian tail behavior; see also the related development in \citet{AbdallaZhivotovskiy2024covariance}. Specializing their guarantee to the uncontaminated case and choosing the confidence level $\delta = 1/n$, there exists a measurable estimator $\widehat{\mathsf{E}}_{1}=\widehat{\mathsf{E}}_1(X_1,\ldots,X_n)\in\mathbb{R}^{d\times d}$ such that, whenever $n \gtrsim c(r(\Sigma)+\log(en))$,
\begin{equation}\label{eq:pilot_OR}
\mathbb{P}\!\left(
\bigl\|\widehat{\mathsf{E}}_1-\Sigma\bigr\|_{\rm op}
\lesssim\kappa_4^2\,\|\Sigma\|_{\rm op}\sqrt{\frac{\mathsf{r}(\Sigma)+\log(en)}{n}}
\right)\ge 1-\frac{1}{n},
\end{equation}
where $\kappa_4$ is the marginal $L_4$--$L_2$ constant (in our notation, $\kappa_4 \leq K_x$ under \ref{asmp:momentX}). Using this estimator as a pilot, we yield the following corollary.

\begin{corollary}\label{cor:PM1} Suppose that Assumption~\ref{asmp:momentX} holds. Then there exists an estimator $\hat\Omega^\dagger = \hat\Omega^\dagger(X_1,\ldots,X_{2n})\in\Real^{d\times d}$ such that (1) if $d\log^2(en)=o(n)$ then \eqref{eq:asymnor_PM} holds. Also, there exists a universal constant $\Ck>0$ such that if $\Ck K_x^4d\log^2(en)\leq n$, then
\begin{equation*}
    {\rm BE}_n(\eta_1^\top\tilde\Omega^\dagger\,\eta_2)\leq \frac{\Ck K_x^{\bar q_x}}{(\bar \sigma_\omega\wedge 1)^{\bar q_x/2}}\bigg[\frac{1}{n^{\bar q_x/4-1}}+\bigg(\frac{d}{n}\bigg)^{1/3}\bigg]
\end{equation*}
\end{corollary}Corollary~\ref{cor:PM1} requires $d\log^2(en)=o(n)$ for asymptotic normality. To our knowledge, this is the most permissive dimension regime currently available under only a fourth-moment condition.

The estimator in Corollary~\ref{cor:PM1} is not, however, computationally tractable as stated. The difficulty has two sources. First, a direct evaluation of the $U$-statistics of order up to $\log(en)$ entails super-polynomial running time. Second, the pilot in \eqref{eq:pilot_OR} is information-theoretic and is not designed for efficient computation. Section~\ref{sec:comrelax} discusses a possible relaxation of the former. For the latter, one may replace $\widehat{\mathsf{E}}_1$ by a polynomial-time pilot at the cost of a logarithmic factor. For example, the spectrum-wise truncated estimator $\widehat{\mathsf{E}}_2$ of \cite{KeMinskerRenSunZhou2019UserFriendly} is computable in polynomial time and satisfies, for all $\delta\in(0,1)$,
\[
\mathbb{P}\!\left(
\|\widehat{\mathsf{E}}_2-\Sigma\|_{\rm op}
\lesssim
\kappa_4^2\,\|\Sigma\|_{\rm op}\sqrt{\frac{r(\Sigma)\{\log (ed)+\log(e/\delta)\}}{n}}
\right)\ge 1-\delta.
\] Related computationally efficient constructions based on matrix-valued $M$-estimation were studied in \citet{WeiMinsker2017CovarianceHeavyTailed}.

By contrast, using the sample Gram matrix typically yields only polynomial-tail control and correspondingly slower operator-norm rates. Under \ref{asmp:momentX}, Proposition~\ref{prop:samcov} shows that for 
$\hat\Sigma_2=n^{-1}\sum_{i=n+1}^{2n}X_iX_i^\top$, if $n\gtrsim_{K_x,q_x} r(\Sigma)$, then
\begin{equation*}
    \Pb\bigg(\norm{\hat\Sigma_2-\Sigma}_{\rm op}\lesssim_{K_x,q_x} \norm{\Sigma}_{\rm op}\bigg[\sqrt{\frac{r(\Sigma)+\log(e/\delta)}{n}}+\frac{r(\Sigma)}{n^{1-2/q_x}\delta^{2/q_x}}\bigg]\bigg)\geq1-\delta.
\end{equation*} The term $r(\Sigma)/n^{1-2/q_x}$ is unavoidable for the sample Gram matrix in general; see Proposition~\ref{prop:sampcov.opt} for a corresponding lower bound. As a benchmark, under a sub-Gaussian design, the sample covariance enjoys standard sub-Gaussian operator-norm concentration \citep{Vershynin2012SampleCovariance}.

\begin{enumerate}[label=\textbf{(SG)(X)}]
    \item\label{asmp:sg} There exists constants $K_{\rm sg}\geq1$ such that $\sup_{u\in\Sb^{d-1}}\Eb[\exp({|u^\top \Sigma^{-1/2}X|^2}/{K_{\rm sg}^2})]\leq 2.$
\end{enumerate}
Under~\ref{asmp:sg}, \ref{asmp:momentX} automatically holds for any $q_x\geq 4$ with $K_x = K_{\rm sg}\sqrt{2eq_x}$.
\begin{corollary}\label{cor:PM2} Let $\widehat\Omega$ be the matrix estimator in \eqref{eq:precisionmatrixest.} with $\tilde\Sigma_1 = \hat\Sigma_1$, $\tilde\Sigma_2 = \hat\Sigma_2$ and $s_n=\lfloor\log(en)\rfloor$.
\begin{enumerate}[left=0pt]
    \item Under Assumption~\ref{asmp:sg}, $\widehat\Omega$ satisfies the asymptotic normality in \eqref{eq:asymnor_PM} if $d\log^2(en)=o(n)$. Moreover, there exists a universal constant $\Ck>0$ such that if $\Ck K_{\rm sg}^4d\log^2(en)\leq n$, then
\begin{equation*}
    {\rm BE}_n(\eta_1^\top\widehat\Omega\,\eta_2)\leq \frac{\Ck K_{\rm sg}^6}{(\bar \sigma_\omega\wedge 1)^3}\bigg(\frac{d}{n}\bigg)^{1/3}.
\end{equation*}
\item Under Assumption~\ref{asmp:momentX} with $q_x>4$, the asymptotic normality in \eqref{eq:asymnor_PM} holds if $d=o(n^{1-2/q_x})$. Moreover, there exists a constant $\Ck=\Ck(K_x,q_x)>0$ such that 
\begin{equation*}
    {\rm BE}_n(\eta_1^\top\widehat\Omega\,\eta_2)\leq \frac{\Ck}{(\bar \sigma_\omega\wedge 1)^{\bar q_x/2}}\bigg[\frac{1}{n^{\bar q_x/4-1}}+\bigg(\frac{d}{n}\bigg)^{1/3}+\bigg(\frac{d}{n^{1-2/q_x}}\bigg)^{\frac{q_x}{q_x+2}}\bigg].
\end{equation*}
\end{enumerate}
\end{corollary}


\subsection{Estimation of Linear Regression Functional}
 Consider a regression tuple $(Y,X)\in\Real\times\Real^d$ where $Y$ is a scalar response and $X$ is a $d$-dimensional covariate vector. The vector of projection parameters \citep{bujaS1} is defined as the minimizer of the population $L_2$-loss:
\begin{equation}\label{eq:PP}
\beta := \argmin_{\gamma \in \Real^d} \Eb[(Y - X^\top \gamma)^2].
\end{equation}
Assume the joint distribution of $(Y, X)$ admits the second moments, and let $\Sigma = \Eb[XX^\top]$ and $\Gamma = \Eb[XY]$. If $\Sigma$ is invertible, the parameter $\beta$ is uniquely given by the functional
\begin{equation} \label{eq:regressionfunctional}
\beta := \beta(\Sigma, \Gamma) = \Sigma^{-1} \Gamma.
\end{equation}
In this section, we study on the estimation of the linear functional $\beta_\eta := \eta^\top \beta$ for an arbitrary direction $\eta \in \Real^d\setminus\{0_d\}$.

Suppose that $N=2n$ i.i.d.\ observations $\{(Y_i,X_i)\}_{i=1}^{2n}$ are available, and split the sample into two independent subsamples, $\Sc_1=\{(Y_i,X_i):i\in\{1,\ldots,n\}\}$ and $\Sc_2=\{(Y_i,X_i):i\in\{n+1,\ldots,2n\}\}$, each of size $n$. In addition to distribution assumptions \ref{asmp:momentX} and \ref{asmp:sg} on covariates $X$, we impose the following moment condition on the population residuals (termed by \cite{bujaS1}).

\begin{enumerate}[label=\textbf{(M)($\varepsilon$)},leftmargin=1.5cm]
    \item There exist constants $K_\varepsilon\geq1$ and $q_\varepsilon\geq 2$ such that $(\Eb[|Y-X^\top\beta|^{q_\varepsilon}])^{1/q_\varepsilon}\leq K_\varepsilon$.\label{asmp:momente}
\end{enumerate}

The map $(\Sigma,\Gamma)\mapsto \beta_\eta(\Sigma,\Gamma)$ is infinitely Fr\'echet differentiable and it belongs to the Gevrey class of order $1$. Precisely, let $\Bb = \Sb\Pb^d \oplus \Real^d$ be the product space of positive definite matrices and vectors. For $H = (A, B) \in \Bb$ and increments $h_j = (a_j, b_j) \in \Bb$ for $j=1, \ldots, k$, the symmetric $k$-linear form $\Dc^k \beta_\eta(H): \Bb^{\otimes k} \to \Real$ is given by
\begin{equation} \label{eq:regressionfunctional_derivative}
\Dc^k \beta_\eta(H)[h_1, \dots, h_k] = (-1)^k \sum_{\pi\in\Sk_k} \eta^\top A^{-1} a_{\pi(1)} A^{-1} \cdots A^{-1} (a_{\pi(k)} A^{-1}B - b_{\pi(k)}).
\end{equation}

We assume that there exist estimators $(\tilde\Sigma,\tilde\Gamma)$ with the following concentration guarantees.

\begin{enumerate}[label=\textbf{(P2)}]
    \item\label{asmp:SigmapilotGammapilot} There exist estimators $\tilde\Sigma:=\tilde\Sigma(X_1,\ldots,X_n)\in\Real^{d\times d}$ and $\tilde\Gamma:=\tilde\Gamma(X_1,\ldots,X_n)\in\Real^d$, non-increasing functions $r_{1,n},r_{2,n}:(0,1)\to\Real^+$, and sequences $\delta_{1,n},\delta_{2,n}\in(0,1)$ with $\delta_{1,n}\vee\delta_{2,n}\to0$ as $n\to\infty$ such that
    \begin{equation*}
        \Pb\bigg(\norm{\Sigma^{-1/2}\tilde\Sigma\Sigma^{-1/2}-I_d}_{\rm op}\geq r_{1,n}(\delta_{1,n})\bigg)\leq \delta_{1,n},\quad \Pb\bigg(\norm{\tilde\Gamma-\tilde\Sigma\beta}_{\Sigma^{-1}}\geq r_{2,n}(\delta_{2,n})\bigg)\leq \delta_{2,n}.
    \end{equation*}
\end{enumerate}
For $\Sc_2$-measurable estimators $\tilde\Sigma_2 = \tilde\Sigma(X_{n+1},\ldots,X_{2n})$, $\tilde\Gamma_2=\tilde\Gamma(X_{n+1},\ldots,X_{2n})$ and an integer sequence $\{s_n\ge 1\}$, we define an estimator $\tilde\beta(\Sc_1,\Sc_2)$ as
\begin{equation*}
\eta^\top \tilde\beta(\Sc_1,\Sc_2) = \eta^\top \tilde\Sigma_2^{-1} \tilde\Gamma_2 + \sum_{k=1}^{s_n} \frac{\Dc^k \beta_\eta(\tilde\Sigma_2, \tilde\Gamma_2)[\bar U^{(k)}(\tilde\Sigma_2, \tilde\Gamma_2)]}{k!},
\end{equation*} where, for $k\ge 1$, the tensor $U$-statistic component is 
\begin{align*}
    \bar U^{(k)}(\tilde\Sigma_2, \tilde\Gamma_2)= \binom{n}{k}^{-1}\sum_{1\leq j_1<\ldots<j_k\leq n}(X_{j_1}X_{j_1}^\top-\tilde\Sigma_2,X_{j_1}Y_{j_1}-\tilde\Gamma_2)\otimes_{\rm s}\cdots\\
    \cdots\otimes_{\rm s} (X_{j_k}X_{j_k}^\top-\tilde\Sigma_2,X_{j_k}Y_{j_k}-\tilde\Gamma_2).
\end{align*} We denote the cross fitted estimator as
\begin{equation}\label{eq:PPestimate}
    \tilde\beta=\frac{1}{2}\big(\tilde\beta(\Sc_1,\Sc_2)+\tilde\beta(\Sc_2,\Sc_1)\big).
\end{equation}
Let $\sigma_\beta^2=\eta^\top\Sigma^{-1}H\Sigma^{-1}\eta$ for $H={\rm Var}\big(X(Y-X^\top\beta)\big)$ and denote the normalized variance  as $\bar\sigma_\beta^2:=\sigma_\beta^2/\|\eta\|_{\Sigma^{-1}}^2$. The matrix $\Sigma^{-1}H\Sigma^{-1}$ is the usual sandwich variance, which is efficient in nonparametric linear regression \citep{White1980HC0,bujaS1}.
 Theorem below provides the mild sufficient conditions for the asymptotic normality of a linear contrast $\eta^\top\tilde\beta$: for any $\eta\in\Real^d\setminus\{0_d\}$,
\begin{equation}\label{eq:asympnorm.PP}
    \frac{\sqrt{N}(\eta^\top\tilde\beta-\eta^\top\beta)}{\sigma_\beta}\xrightarrow{D}\Nc(0,1),\quad \mbox{as}\quad N\to\infty,
\end{equation}together with explicit bounds on the KS-distance:
\begin{equation*}
    {\rm BE}_n(\eta^\top\tilde\beta):=\sup_{t\in\Real}\bigg|\Pb\bigg(\frac{\sqrt{N}(\eta^\top\tilde\beta-\eta^\top\beta)}{\sigma_\beta}\leq t\bigg)-\Phi(t)\bigg|.
\end{equation*}

\begin{theorem}\label{thm:PP1} Suppose that Assumption~\ref{asmp:momentX} and ~\ref{asmp:momente} hold with $q_{x\varepsilon}:=(1/q_x + 1/q_\varepsilon)^{-1}\geq 2$, and that \ref{asmp:SigmapilotGammapilot} holds. Let $r_{1,n}^* = r_{1,n}(\delta_{1,n})$ and $r_{2,n}^* = r_{2,n}(\delta_{2,n})$ for $n\geq 1$.
\begin{enumerate}
    \item Let $\tilde\beta$ be the estimator in \eqref{eq:PPestimate} with a fixed $s_n=s\geq 1$. If $d=o(n)$, $\sqrt{n}(r_{1,n}^*)^sr_{2,n}=o(1)$, and $(r_{1,n}^*)^{s-1}r_{2,n}=o(1)$, then $\tilde\beta$ satisfies \eqref{eq:asympnorm.PP}. Moreover, let $\bar q_{x\varepsilon}=q_{x\varepsilon}\wedge 3$. Then, there exists a constant $\Ck = \Ck(s)>0$ such that
    \begin{align*}
        {\rm BE}_n(\eta^\top\tilde\beta)&\leq \frac{\Ck\,(K_xK_\varepsilon)^{\bar q_{x\varepsilon}}}{(\bar\sigma_\beta\wedge 1)^{\bar q_{x\varepsilon}}}\bigg[\frac{1}{n^{\bar q_{x\varepsilon}/2-1}}+\bigg(\frac{d}{n}\bigg)^{1/3}\\
        &\qquad+\sqrt{n}(r_{1,n}^*)^sr_{2,n}+(r_{1,n}^*\vee r_{2,n}^*)^{2/3}(r_{1,n}^*)^{2s/3-2/3}\bigg]+\delta_{1,n} + \delta_{2,n}.
    \end{align*}
    \item Let $\tilde\beta$ be the estimator in \eqref{eq:PPestimate} with $s_n=\lfloor\log(en)\rfloor.$ If $d\log^2(en)=o(n)$ and $(2er_{1,n}^*)\vee r_{2,n}^*\leq 1$ eventually, then $\tilde\beta$ satisfies \eqref{eq:asympnorm.PP}. Moreover, if $(2r_{1,n}^*)^2\vee (8K_x^4d\log^2(en))/n\leq e^{-2}$ and $r_{2,n}^*\leq 1$, then there exists a universal constant $\Ck>0$ such that
    \begin{equation*}
        {\rm BE}_n(\eta^\top\tilde\beta)\leq \frac{\Ck\,(K_xK_\varepsilon)^{\bar q_{x\varepsilon}}}{(\bar\sigma_\beta\wedge 1)^{\bar q_{x\varepsilon}}}\bigg[\frac{1}{n^{\bar q_{x\varepsilon}/2-1}}+\bigg(\frac{d}{n}\bigg)^{1/3}\bigg]+\delta_{1,n} + \delta_{2,n}.
    \end{equation*}
\end{enumerate}
\end{theorem}
Since $\bar\sigma_\beta^2\ge \lambda_{\min}(\Sigma^{-1/2}H\Sigma^{-1/2})$, Theorem~\ref{thm:PP1} implies uniform Gaussian approximation over $\eta\in\Real^d\setminus\{0\}$ under additional spectral condition. The following corollary specializes the result to method-of-moments pilots. We let $\hat\Gamma_1=n^{-1}\sum_{i=1}^{n}X_iY_i$ and $\hat\Gamma_2=n^{-1}\sum_{i=n+1}^{2n}X_iY_i$.

\begin{corollary}\label{cor:PP1} Let $\tilde\beta$ be the estimator in \eqref{eq:PPestimate} with $(\tilde\Sigma_1,\tilde\Gamma_1,\tilde\Sigma_2,\tilde\Gamma_2)=(\hat\Sigma_1,\hat\Gamma_1,\hat\Sigma_2,\hat\Gamma_2)$ and $s_n=\lfloor\log(en)\rfloor.$ 
\begin{enumerate}
    \item Suppose that Assumption~\ref{asmp:momentX} and \ref{asmp:momente} hold with $q_x>4$ and $q_{x\varepsilon}>2.$ If $d=o(n^{1-2/q_x})$ then the asymptotic normality~\eqref{eq:asympnorm.PP} holds. Moreover, there exists a constant $\Ck = \Ck(K_x,K_\varepsilon,q_x,q_\varepsilon)>0$ such that 
    \begin{equation*}
        {\rm BE}_n(\eta^\top\tilde\beta)\leq \frac{\Ck}{(\bar\sigma_\beta\wedge 1)^{\bar q_{x\varepsilon}}}\bigg[\frac{1}{n^{\bar q_{x\varepsilon}/2-1}}+\bigg(\frac{d}{n}\bigg)^{1/3}+\bigg(\frac{d}{n^{1-2/q_x}}\bigg)^{\frac{q_x}{q_x+2}}\bigg].
    \end{equation*}
    \item Suppose that Assumption~\ref{asmp:sg} and \ref{asmp:momente} hold with $q_{\varepsilon}>2.$ If $d\log^2(en) = o(n)$, then the asymptotic normality~\eqref{eq:asympnorm.PP} holds. Moreover, for any $\Delta\in(0,q_{\varepsilon}-2)\cap(0,1]$, there exists a constant $\Ck = \Ck(q_\varepsilon,\Delta)>0$ such that if $\Ck K_{\rm sg}^4d\log^2(en)\leq n$, then
    \begin{equation*}
        {\rm BE}_n(\eta^\top\tilde\beta)\leq \frac{\Ck\,(K_{\rm sg}K_\varepsilon)^{2+\Delta}}{(\bar\sigma_\beta\wedge 1)^{2+\Delta}}\bigg[\frac{1}{n^{\Delta/2}}+\bigg(\frac{d}{n}\bigg)^{1/3}\bigg].
    \end{equation*}
\end{enumerate}
\end{corollary}
Under sub-Gaussian covariates, the method-of-moments pilots suffice for asymptotic normality of $\eta^\top\tilde\beta$ provided that $q_\varepsilon>2$ and $d\log^2(en)=o(n)$. If $q_\varepsilon>3$, the Berry--Esséen bound scales as $\frac{1}{\sqrt{n}}+(\frac{d}{n})^{1/3}$. Under Assumption~\ref{asmp:momentX} alone, this construction requires $d=o(n^{1-2/q_x})$, which matches the maximal regime in which the sample Gram matrix is consistent in operator norm (Proposition~\ref{prop:samcov} and \ref{prop:sampcov.opt}). It is nevertheless possible to obtain asymptotic normality over a near-full dimensional range by using an information-theoretic pilot for $\Sigma$. Recall the robust matrix estimator $\widehat{\mathsf{E}}_{1}$ in \eqref{eq:pilot_OR}. We set $\tilde\Sigma_2=\widehat{\mathsf{E}}_{1}(X_{n+1},\ldots,X_{2n})$ and $\tilde\Gamma_2=\tilde\Sigma_2\hat\beta_2$, where $\hat\beta_2=\hat\Sigma_2^{-1}\hat\Gamma_2$ is the ordinary least squares estimator computed from $\Sc_2$. The estimators $(\tilde{\Sigma}_1, \tilde{\Gamma}_1)$ are defined analogously using the observations in $\mathcal{S}_1$.

\begin{corollary}\label{cor:PP2}
    There exists an estimator $\tilde\beta^\dagger=\tilde\beta^\dagger(Y_1,X_1,\ldots,Y_{2n},X_{2n})$ such that the asymptotic normality~\eqref{eq:asympnorm.PP} holds if $d\log^2(en) = o(n)$. Moreover, if $(8K_x^4d\log^2(en))/n\leq e^{-2}$, then there exists a constant $\Ck=\Ck(q_{xy})>0$ such that
    \begin{equation*}
        {\rm BE}_n(\eta^\top\tilde\beta)\leq \frac{\Ck\,(K_xK_\varepsilon)^{\bar q_{x\varepsilon}}}{(\bar\sigma_\beta\wedge 1)^{\bar q_{x\varepsilon}}}\bigg[\frac{1}{n^{\bar q_{x\varepsilon}/2-1}}+\bigg(\frac{d}{n}\bigg)^{1/3}\bigg].
    \end{equation*}
\end{corollary}

\subsection{Computational Relaxation under Product Structure of Derivatives}\label{sec:comrelax}

Direct computation of our estimator can be prohibitively expensive because evaluating $U$-statistics of order $k$ naively requires averaging over $\binom{n}{k}$ subsets for each $k$, and explicit formulas for $\Dc^k f(\cdot)$ often introduce additional sums over permutations of $[k]$ through noncommutative products.

In this section, we identify a structural condition on the Frech\'et differentials, common to functionals of matrices and operators, enabling scalable computation. Suppose that $\Bb$ is a unital Banach algebra, which is closed under the left-multiplication, with submultiplicative norm: $\norm{ab}\leq \norm{a}\norm{b}$ for $a,b\in\Bb$. We consider $m$-smooth functionals $f$ whose Frechet differentials obeys the following structure.
\begin{assumption}\label{asmp:f_product_structure}
    For $s=\lceil m\rceil-1\geq1$ and each $k\in[s]$ there exists $T_k(\cdot):\Bb\to\Lc^k(\Bb)$ such that for all $x\in\Theta$ and $h_1,\dots,h_k\in\Bb$,
    \begin{equation}\label{eq:f_product_structure}
        \Dc^k f(x)[h_1,\ldots,h_k] = \sum_{\sigma\in\Sk_k} T_k(x)[h_{\sigma(1)},\ldots,h_{\sigma(k)}],
    \end{equation} where $T_k(x)[h_1,\ldots,h_k]=\Lambda_{k}(x)\!\big(
G_{k,0}(x)\,h_1\,G_{k,1}(x)\cdots h_k\,G_{k,k}(x)\big)$, for some $\Lambda_{k}:\Bb\to\Bb^*$, and maps
$G_{k,j}:\Theta\to \Bb$ for $j=0,1,\dots,k$.
\end{assumption}

In addition to the precision-matrix and regression functionals discussed in Section~\ref{sec:appl}, Assumption~\ref{asmp:f_product_structure} encompasses several canonical examples in matrix algebras. For instance, the matrix entropy $f(A) = \log\det(A)$ on the cone of positive definite matrices $\mathbb{S}\mathbb{P}^d$ satisfies the required structure; specifically, its $k$-th order differential is given by $$D^k f(A)[h_1,\dots,h_k] = \frac{(-1)^{k-1}}{k}\sum_{\sigma\in \Sk_k}{\rm tr}(A^{-1}h_{\sigma(1)}A^{-1}\cdots h_{\sigma(k)}).$$ Moreover, the Stieltjes transform (or penalized inverse) defined by $f(A) = \tr(B(A-zI_d)^{-1})$ on $\mathbb{C}^{d\times d}$ admits a product structure for $B \in \mathbb{C}^{d\times d}$, where $$D^k f(A)[h_1,\dots,h_k] = (-1)^k\sum_{\sigma\in \Sk_k}{\rm tr}(B(A-zI_d)^{-1}h_{\sigma(1)}(A-zI_d)^{-1}\cdots h_{\sigma(k)}(A-zI_d)^{-1}).$$

Under \eqref{eq:f_product_structure}, the $k$-th order term of the expansion can be expressed as an average over the symmetric group $\Sk_n$. Specifically, for $F_k(\pi):=\binom{n}{k}^{-1}\sum_{1\le i_1<\cdots<i_k\le n}T_k(\hat\theta_{\Sc_2})[H_{\pi(i_1)},\ldots,H_{\pi(i_k)}]$,
\begin{equation}\label{eq:Fk_pi_general}
    \frac{1}{k!}\,\Dc^k f(\hat\theta_{\Sc_2})\big[\bar U^{(k)}(\hat\theta_{\Sc_2})\big] =\frac{1}{n!}\sum_{\pi\in \Sk_n}F_k(\pi).
\end{equation}Naive computation of $F_k(\pi)$ still requires averaging over $\binom{n}{k}$ subsets. However, the associative product structure in Assumption~\ref{asmp:f_product_structure} allows us to use dynamic recursion, reducing the overall cost from combinatorial to polynomial. Algorithm~\ref{alg:ro_dp_general} computes our estimator in $O(bns^2)$ algebra multiplications/additions (see Proposition~\ref{prop:recursion} for the computational gurantee). This complexity excludes the cost of constructing the preliminary estimators and evaluating the derivative objects.

\begin{algorithm}[ht]
\caption{Permutation-randomized estimator via dynamic programming }
\label{alg:ro_dp_general}
\begin{algorithmic}[1]
\Require Data $W_1,\dots,W_n\in\Bb$, pilot $\hat\theta_{\Sc_2}\in\Theta$, expansion order $s\in\mathbb{N}$, number of permutations $B\in\mathbb{N}$.
\Require For each $k\in\{2,\dots,s\}$: linear functional $\Lambda_k(\hat\theta_{\Sc_2}):\Bb\to\mathbb{R}$ and factors $G_{k,j}(\hat\theta_{\Sc_2})\in\Bb$ for $j=0,\dots,k$.
\Ensure Estimator in \eqref{eq:PRE}.

\State Compute $H_i \gets W_i-\hat\theta_{\Sc_2}$ for $i=1,\dots,n$.
\State Initialize $\bar F_k^{(B)}\gets 0$ for $k=2,\dots,s$.
\For{$r=1,\dots,b$}
    \State Draw $\pi\sim{\rm Unif}(\Sk_n)$.
    \For{$k=2,\dots,s$}
        \State $Y_0 \gets G_{k,0}(\hat\theta_{\Sc_2})$; \quad $Y_j \gets 0_{\Bb}$ for $j=1,\dots,k$.
        \For{$t=1,\dots,n$}
            \For{$j=\min\{t,k\},\min\{t,k\}-1,\dots,1$}
                \State $Y_j \gets Y_j \;+\; Y_{j-1}\,H_{\pi(t)}\,G_{k,j}(\hat\theta_{\Sc_2})$.
            \EndFor
        \EndFor
        \State $F_k(\pi) \gets \binom{n}{k}^{-1}\,\Lambda_k(\hat\theta_{\Sc_2})\!\big(Y_k\big)$.
        \State $\bar F_k^{(b)} \gets \bar F_k^{(b)} + F_k(\pi)/b$.
    \EndFor
\EndFor
\State \Return $f(\hat\theta_{\Sc_2})+\Dc^1 f(\hat\theta_{\Sc_2})[\bar W_{\Sc_1}-\hat\theta_{\Sc_2}]+\sum_{k=2}^s \bar F_k^{(b)}$.
\end{algorithmic}
\end{algorithm}

Let $\{\pi_{kr}:1\leq r\leq b,2\leq k\leq s\}$\footnote{One may alternatively reuse a common set of permutations $\{\pi_r:1\le r\le b\}$ across all orders $2\le k\le s$. This simplifies implementation, yet introduces dependence and an additional factor (at most linear in $s$) in their statistical guarantees.} be i.i.d.\ $b(s-1)$ permutations of $[n]$, independent of everything else, and set $\bar F_k^{(b)} := \frac{1}{b}\sum_{r=1}^b F_k(\pi_{kr})$. We define the (one-sided) permutation-randomized estimator as
\begin{equation*}
    \widehat{f}_s^{(b)} (\Sc_1,\Sc_2) := f(\hat\theta_{\Sc_2})+\Dc^1f(\hat\theta_{\Sc_2})[\bar W_{\Sc_1}-\hat\theta_{\Sc_2}]+\sum_{k=2}^s \bar F_k^{(b)}.
\end{equation*} We similarly define $\widehat{f}_s^{(b)} (\Sc_2,\Sc_1)$ by interchanging the use of $\Sc_1$ and $\Sc_2$ with same permutations $\{\pi_{kr}:1\leq r\leq b,2\leq k\leq s\}$. Final cross-fitted estimator is given by
\begin{equation}\label{eq:PRE}
    \widehat{f}_s^{(b)} = \frac{1}{2}\bigg(\widehat{f}_s^{(b)} (\Sc_1,\Sc_2) + \widehat{f}_s^{(b)} (\Sc_2,\Sc_1)\bigg).
\end{equation}

\subsubsection{Statistical Guarantee}
In this section, we quantify the statistical cost of permutation randomization. The results herein do not require the associative product structure of Assumption~\ref{asmp:f_product_structure}. We rather rely on a multilinear form expression through $T_k$ in \eqref{eq:f_product_structure}, which is always guaranteed for symmetric $k$-linear forms $\Dc^k f$ by taking $T_k(x)[h_1,\ldots,h_k]=\sum_{\sigma\in\Sk_k} \Dc^k f(x)[h_{\sigma(1)},\ldots,h_{\sigma(k)}]$. We first consider $m$-smooth functionals.

\begin{theorem}\label{thm:normality-msmooth-MC} Under the conditions in Theorem~\ref{thm:normality-msmooth}, assume $s\geq 2$ and $f$ satisfies Assumption~\ref{asmp:f_product_structure} with $\max_{2\leq k\leq s}\sup_{x\in U}\norm{T_k}_{\rm S}\leq C_0.$ Let ${\rm BE}_n$ be the Berry--Esséen  bound in Theorem~\ref{thm:normality-msmooth}, then
\begin{align*}
    &\sup_{t\in\Real}\Abs{\Pb_P\left(\frac{\sqrt{N}\,(\widehat{f}_s^{(b)} - f(\theta(P)))}{\sigma_f}\leq t\right)-\Phi(t)}\lesssim_{\bar\nu,s} {\rm BE}_n + \frac{1}{b^{1/3}\sigma_f^{2/3}}\bigg(r_n^2\vee\frac{\dk}{n}\bigg)^{1/3} + \delta_{n}.
\end{align*}
\end{theorem}
Theorem~\ref{thm:normality-msmooth-MC} shows that permutation randomization contributes an additional Berry--Esséen  term of order $b^{-1/3}\big(r_n^2\vee \dk/n\big)^{1/3}$. In particular, even $b=1$ yields the asymptotic normality under $r_n =o(1)$ and $\dk = o(n)$.

We next state an analogue for infinitely differentiable functionals. For $R>0$ and $U\subseteq\Theta$, define a semi-norm:
\begin{equation*}
    \norm{T_{\geq 0}}_{L_\infty(U),R} = \sup_{k\geq 0}\frac{\sup_{x\in U}\norm{T_k(x)}_{\rm S}}{R^k}.
\end{equation*} Since $\norm{\Dc^k f}_{\rm S}\leq k!\,\norm{T_k}_{\rm S}$ (Assumption~\ref{asmp:f_product_structure}), the bound $\|T_{\ge 0}\|_{L_\infty(U),R}<\infty$ implies Gevrey-$1$ regularity of $f$ on $U$; in particular one has $\|f\|_{\Gc^1(U),R}\leq \|T_{\ge 0}\|_{L_\infty(U),R}$.
Moreover, if
\[
R_0:=\sup_{k\ge 0}\sup_{x\in U}\max\Big\{\|\Lambda_k(x)\|_*,\ \|G_{k,0}(x)\|,\ldots,\|G_{k,k}(x)\|\Big\}<\infty,
\]
then $\|T_{\ge 0}\|_{L_\infty(U),R_0}<\infty$.

\begin{theorem}\label{thm:normality-gevrey-MC}
    For $P\in\Pc$, suppose Assumption~\ref{asmp:theta_tilde_general} holds and let $r_n = r_{n,P}(\delta_{n,P})$. For $U\supseteq B(\theta(P),r_n)$ and $R>0$, let $\norm{T_{\geq 0}}_{L_\infty(U),R}<\infty$. For any $b\geq 1$, if $r_n$ satisfies \eqref{eq:cond:gevrey} with $\alpha=1$ eventually as $n\to\infty$, then $\widehat{f}_s^{(b)}$ with $s=\lfloor\log(en)\rfloor$ is asymptotically normal, i.e., $\sqrt{n}(\widehat{f}_s - f(\theta))\xrightarrow{D}\Nc(0,\sigma_f^2)$. Moreover, under \eqref{eq:cond:gevrey} with $\alpha=1$, there exists a universal constant $C>0$ such that
\begin{align*}
    &\sup_{t\in\Real}\Abs{\Pb_P\left(\frac{\sqrt{N}\,(\widehat{f}_s^{(b)} - f(\theta(P)))}{\sigma_f}\leq t\right)-\Phi(t)}\leq C\bigg[ \frac{\mu_{3,f}}{\sigma_f^3\sqrt{n}}+\frac{\norm{T_{\geq 0}}_{L_\infty(U),R}Rr_n}{\sigma_f\sqrt{n}}\\
    &\quad +\bigg(\frac{\norm{T_{\geq 0}}_{L_\infty(U),R}R^{2}}{\sigma_f}\bigg)^{2/3}\bigg\{\bigg(\frac{\nu(P) V(P)}{n}\bigg)^{1/3} + \frac{\nu(P)^{1/3}}{\sqrt{n}}+\bigg(\frac{\nu(P)r_n^2}{b}\bigg)^{1/3}\bigg\}\bigg]+\delta_{n,P}.
\end{align*}
\end{theorem}
We conclude this section with remarks on the inadequacy of incomplete $U$-statistics in our applications.
\begin{remark}
A classical alternative is the incomplete $U$-statistic $U^{\mathrm{inc}}_{n,k}$ via $b$ random $k$-subsampling. However, its variance scales as $\Var(U_{n,k}) + \zeta_k/b$ where $\zeta_k$ denotes the variance of kernel. In our setting, $\zeta_k$ scales as $\nu V^{k-1}$ (Lemma~\ref{lem:varbound}), growing exponentially in $k$ or polynomially in the effective dimension. To suppress this, the number of subsamples needs to be prohibitively large. In contrast, the permutation randomized estimator in \eqref{eq:PRE} exploits and tailored to the specific algebraic structure and is fundamentally different to the incomplete $U$-statistic.
\end{remark}
\section{Numerical Experiments}
\begin{table}[tp]
\centering
\small
\setlength{\tabcolsep}{4pt}
\begin{tabular}{ccrrrrrrr}
\toprule
$\gamma$ & $d$ & Plug-in & Jackknife & HODSE (o2) & K\&L (o2) & IB (o2) & C\&K Full (o2) & C\&K PRE (o2) \\
\midrule
0.15 & 2   & 1.000 & 1.002 & 1.002 & 2.526 & 1.049 & 1.003 & \textbf{0.948} \\
0.2  & 3   & 1.000 & \textbf{0.987} & 0.987 & 5.072 & 1.550 & 1.211 & 1.296 \\
0.25 & 5   & 1.000 & 1.007 & 1.007 & 5.066 & 1.590 & \textbf{0.954} & 1.086 \\
0.3  & 7   & 1.000 & 0.995 & 0.995 & 3.499 & 1.146 & 0.977 & \textbf{0.952} \\
0.35 & 11  & \textbf{1.000} & 1.039 & 1.038 & 3.990 & 1.766 & 1.071 & 1.134 \\
0.4  & 15  & 1.000 & 1.024 & 1.023 & 2.364 & 0.908 & 1.010 & \textbf{0.849} \\
0.45 & 22  & 1.000 & 0.966 & 0.968 & 3.488 & 0.929 & \textbf{0.918} & 0.971 \\
0.5  & 31  & 1.000 & \textbf{0.982} & 0.983 & 3.053 & 1.146 & 1.012 & 1.594 \\
0.55 & 44  & 1.000 & 0.997 & \textbf{0.995} & 4.479 & 1.000 & 1.309 & 1.820 \\
0.6  & 63  & \textbf{1.000} & 1.000 & 1.002 & 12.232 & 1.375 & 1.613 & 2.073 \\
0.65 & 89  & 1.000 & \textbf{0.993} & 0.995 & 49.682 & 1.078 & 2.883 & 4.524 \\
0.7  & 125 & 1.000 & 0.971 & 0.983 & 92.685 & \textbf{0.949} & 5.316 & 7.623 \\
0.75 & 177 & \textbf{1.000} & 1.023 & 1.028 & 4807.042 & 1.513 & 88.608 & 87.770 \\
\bottomrule
\end{tabular}
\caption{Median squared-error ratios relative to the plug-in estimator. Values below 1 indicate improvement over plug-in. The smallest value in each row is bolded.}
\label{tab:all_med_ratio}
\end{table}

\begin{table}[tp]
\centering
\small
\setlength{\tabcolsep}{4pt}
\begin{tabular}{ccrrrrr}
\toprule
$\gamma$ & $d$ & Plug-in & Jackknife & K\&L (best) & IB (best) & C\&K PRE (best) \\
\midrule
0.15 & 2   & 1.000 & 1.002 & 1.504 (o1) & 1.040 (o3) & \textbf{0.920 (o8)} \\
0.2  & 3   & 1.000 & \textbf{0.987} & 3.865 (o1) & 1.279 (o1) & 1.103 (o8) \\
0.25 & 5   & 1.000 & 1.007 & 2.202 (o1) & 0.949 (o1) & \textbf{0.882 (o5)} \\
0.3  & 7   & 1.000 & 0.995 & 2.263 (o1) & 1.043 (o1) & \textbf{0.911 (o7)} \\
0.35 & 11  & \textbf{1.000} & 1.039 & 2.193 (o1) & 1.108 (o1) & 1.111 (o4) \\
0.4  & 15  & 1.000 & 1.024 & 2.246 (o1) & 0.908 (o2) & \textbf{0.849 (o2)} \\
0.45 & 22  & 1.000 & 0.966 & 2.164 (o1) & \textbf{0.929 (o2)} & 0.971 (o2) \\
0.5  & 31  & 1.000 & 0.982 & 2.653 (o1) & \textbf{0.929 (o1)} & 1.163 (o3) \\
0.55 & 44  & 1.000 & 0.997 & 2.409 (o1) & 1.000 (o2) & \textbf{0.844 (o1)} \\
0.6  & 63  & \textbf{1.000} & 1.000 & 5.313 (o1) & 1.038 (o1) & 2.073 (o2) \\
0.65 & 89  & 1.000 & \textbf{0.993} & 3.078 (o1) & 1.078 (o2) & 3.024 (o1) \\
0.7  & 125 & 1.000 & 0.971 & 5.323 (o1) & \textbf{0.801 (o1)} & 1.244 (o1) \\
0.75 & 177 & \textbf{1.000} & 1.023 & 16.674 (o1) & 1.195 (o1) & 7.288 (o1) \\
\bottomrule
\end{tabular}
\caption{Median squared-error ratios for practical higher-order methods. For K\&L, IB, and C\&K PRE, the best order in $\{1,\dots,8\}$ is reported (parentheses). The smallest value in each row is bolded.}
\label{tab:high_med_ratio}
\end{table}
We consider a regression projection problem. For each \(i=1,\dots,N\), with \(N=1000\), we generate i.i.d. pairs \((X_i,Y_i)\), where
\[
X_i \sim N(0,\Sigma),
\qquad
\Sigma_{jk}=(0.6)^{|j-k|}, \quad 1\leq j,k\leq d,
\]
and the dimension is \(d=\lfloor n^\gamma\rfloor\) for \(\gamma\in\{0.15,0.20,\ldots,0.75\}\). The response is generated from the nonlinear heteroskedastic model
\[
Y_i
=
\sin(X_{i1})
+
\frac12\bigl(X_{i2}^2-0.6\bigr)
+
(1+0.3X_{i1}^2)\varepsilon_i,
\]
where \(\varepsilon_i\sim N(0,1)\) is independent of \(X_i\). We take \(\eta=\ev_1\), the first canonical basis vector, and consider the scalar target \(\theta=\eta^\top\beta\), where \(\beta\) is the population linear projection coefficient in \eqref{eq:PP}. In this setting, the projection target is available in closed form: $\beta=\exp(-1/2)\ev_1.$

We compare the standard plug-in and jackknife estimators against several higher-order alternatives. Our analysis is divided into two parts:
\begin{itemize}
    \item We include the second-order corrected estimator of the High-Order Degenerate Statistical Expansion (HODSE) \citep{ZhouLiZhang2021HighOrderStatisticalFunctionalExpansion}, the estimator of \citet{KoltchinskiiLi2026FunctionalEstimation} (K\&L), and the iterative Gaussian multiplier bootstrap (IB) \citep{HallMartin1988}, alongside our full estimator  \eqref{eq:estimator_def} (C\&K Full) and permutation randomized estimator \eqref{eq:PRE} (C\&K PRE). This setup provides a baseline while ensuring the computationally intensive HODSE and C\&K Full procedures remain feasible.
    \item We then focus on the more scalable methods (K\&L, IB, and C\&K PRE), reporting the best-performing order $k \in \{1,\dots,8\}$ for each. We exclude order 0 as K\&L and IB reduces to the plug-in estimator.
\end{itemize} Throughout these experiments, C\&K PRE uses a single randomized permutation ($b=1$), and IB is implemented with a Monte Carlo sample size of 40 for each iterative bias correction.

Table entries represent median squared-error ratios relative to the plug-in estimator across 100 Monte Carlo replications. We report medians because higher-order methods occasionally produce extreme values; corresponding mean squared-error ratios are provided in Appendix.

\bibliographystyle{plainnat}
\bibliography{bib}

\appendix
\numberwithin{equation}{section}
\numberwithin{theorem}{section}
\numberwithin{figure}{section}
\numberwithin{table}{section}
\renewcommand{\theequation}{\thesection.\arabic{equation}}
\renewcommand{\thelemma}{\thesection.\arabic{lemma}}
\renewcommand{\theproposition}{\thesection.\arabic{proposition}}
\renewcommand{\thetheorem}{\thesection.\arabic{theorem}}
\renewcommand{\thecorollary}{\thesection.\arabic{corollary}}
\renewcommand{\theremark}
{\thesection.\arabic{remark}}
\renewcommand{\thefigure}{\thesection.\arabic{figure}}
\renewcommand{\thetable}{\thesection.\arabic{table}}
\section{Additional Numerical Results}
For completeness, we report the mean squared-error ratios in Table~\ref{tab:all_mean_ratio} and \ref{tab:high_med_ratio} corresponding to the median-based comparisons in the main text. Each entry is computed from 100 Monte Carlo replications.

\begin{table}[tp]
\centering
\small
\setlength{\tabcolsep}{4pt}
\begin{tabular}{ccrrrrrrr}
\toprule
$\gamma$ & $d$ & Plug-in & Jackknife & HODSE (o2) & K\&L (o2) & IB (o2) & C\&K Full (o2) & C\&K PRE (o2) \\
\midrule
0.15 & 2   & \textbf{1.000} & 1.005 & 1.005 & 2.852 & 1.174 & 1.010 & 1.046 \\
0.2  & 3   & 1.000 & 1.006 & 1.006 & 3.187 & 1.204 & \textbf{0.988} & 1.054 \\
0.25 & 5   & 1.000 & 1.005 & 1.005 & 4.825 & 1.259 & \textbf{0.987} & 1.009 \\
0.3  & 7   & 1.000 & 1.003 & 1.003 & 3.161 & 1.154 & 0.998 & \textbf{0.994} \\
0.35 & 11  & 1.000 & 1.007 & 1.007 & 2.533 & 1.188 & 1.002 & \textbf{0.971} \\
0.4  & 15  & 1.000 & 1.005 & 1.004 & 3.125 & 1.079 & \textbf{0.989} & 1.084 \\
0.45 & 22  & 1.000 & 1.003 & 1.003 & 3.077 & 1.222 & \textbf{0.959} & 1.025 \\
0.5  & 31  & \textbf{1.000} & 1.006 & 1.006 & 3.102 & 1.146 & 1.042 & 1.187 \\
0.55 & 44  & \textbf{1.000} & 1.005 & 1.004 & 5.292 & 1.175 & 1.396 & 1.647 \\
0.6  & 63  & 1.000 & 1.003 & 1.003 & 11.868 & \textbf{0.989} & 1.781 & 2.456 \\
0.65 & 89  & \textbf{1.000} & 1.004 & 1.003 & 57.555 & 1.151 & 4.205 & 5.640 \\
0.7  & 125 & \textbf{1.000} & 1.001 & 1.001 & 168.938 & 1.093 & 9.555 & 11.164 \\
0.75 & 177 & \textbf{1.000} & 1.002 & 1.001 & 6718.941 & 1.055 & 63.399 & 71.483 \\
\bottomrule
\end{tabular}
\caption{Mean squared-error ratios relative to the plug-in estimator. Each entry is based on 100 Monte Carlo replications. Values below 1 indicate improvement over plug-in. The smallest value in each row is shown in boldface.}
\label{tab:all_mean_ratio}
\end{table}

\begin{table}[tp]
\centering
\small
\setlength{\tabcolsep}{4pt}
\begin{tabular}{ccrrrrr}
\toprule
$\gamma$ & $d$ & Plug-in & Jackknife & K\&L (best) & IB (best) & C\&K PRE (best) \\
\midrule
0.15 & 2   & 1.000 & 1.005 & 2.175 (o1) & 1.047 (o1) & \textbf{0.990 (o3)} \\
0.2  & 3   & 1.000 & 1.006 & 2.295 (o1) & 1.038 (o1) & \textbf{0.958 (o7)} \\
0.25 & 5   & 1.000 & 1.005 & 2.535 (o1) & 0.989 (o1) & \textbf{0.987 (o3)} \\
0.3  & 7   & 1.000 & 1.003 & 2.522 (o1) & 1.033 (o1) & \textbf{0.965 (o7)} \\
0.35 & 11  & 1.000 & 1.007 & 1.945 (o1) & 1.061 (o1) & \textbf{0.971 (o2)} \\
0.4  & 15  & 1.000 & 1.005 & 2.710 (o1) & \textbf{0.984 (o1)} & 0.990 (o8) \\
0.45 & 22  & \textbf{1.000} & 1.003 & 2.121 (o1) & 1.009 (o1) & 1.010 (o6) \\
0.5  & 31  & 1.000 & 1.006 & 2.434 (o1) & \textbf{0.984 (o1)} & 1.187 (o2) \\
0.55 & 44  & \textbf{1.000} & 1.005 & 3.273 (o1) & 1.045 (o1) & 1.218 (o1) \\
0.6  & 63  & 1.000 & 1.003 & 4.687 (o1) & \textbf{0.989 (o2)} & 2.107 (o1) \\
0.65 & 89  & \textbf{1.000} & 1.004 & 5.107 (o1) & 1.036 (o1) & 2.762 (o1) \\
0.7  & 125 & \textbf{1.000} & 1.001 & 6.589 (o1) & 1.004 (o1) & 2.293 (o1) \\
0.75 & 177 & \textbf{1.000} & 1.002 & 14.110 (o1) & 1.050 (o1) & 5.580 (o1) \\
\bottomrule
\end{tabular}
\caption{Mean squared-error ratios relative to the plug-in estimator. For K\&L, IB, and C\&K PRE, the best order in \(\{1,\dots,8\}\) is reported, with the selected order shown in parentheses. Each entry is based on 100 Monte Carlo replications. The smallest value in each row is shown in boldface.}
\label{tab:high_mean_ratio}
\end{table}
\section{Proofs of Main Results}
We first establish the conditional mean and variance bounds. Recall
\begin{equation*}
    \widehat{f}_s(\mathcal{S}_1, \mathcal{S}_2) = f(\hat\theta_{\mathcal{S}_2}) + \sum_{k=1}^s \frac{1}{k!} \mathcal{D}^k f(\hat\theta_{\mathcal{S}_2}) \left[ \bar{U}^{(k)}(\hat\theta_{\mathcal{S}_2}) \right].
\end{equation*}

\begin{theorem}\label{thm:condmoment}
    Let $\hat r_n = \norm{\hat\theta_{\Sc_2}-\theta}$ and $B_n=B(\theta,\hat r_n)$. Then
    \begin{align}
        &\Big|\Eb\left[\widehat{f}_s(\mathcal{S}_1, \mathcal{S}_2)-f(\theta)\mid\Sc_2\right]\Big|\leq \frac{\Gamma(\rho+1)\norm{\Dc^s f}_{{\rm Lip}_\rho(B_n)}\hat r_n^{s+\rho}}{\Gamma(s+\rho+1)},\label{thm:3.2:1}\\
        &{\rm Var}\left(\widehat{f}_s(\mathcal{S}_1, \mathcal{S}_2)\mid \Sc_2\right)\leq2\sum_{k=1}^s\bigg(\norm{\Dc^kf(\theta)}_{\rm S}^2 + \norm{\Dc^s f}^2_{{\rm Lip}_\rho(B_n)}\frac{\Gamma(\rho+1)^2\hat r_n^{2s-2k+2\rho}}{\Gamma(s-k+\rho+1)^2}\bigg)\frac{\nu V^{k-1}}{k!(n)_k}.\label{thm:3.2:2} 
    \end{align} Moreover, if $s\geq 1$ then
    \begin{align}\label{thm:3.2:3}
        &\Big|\Eb\left[\sqrt{n}\big(\widehat{f}_s(\mathcal{S}_1, \mathcal{S}_2)-f(\theta)-\Dc^1f(\theta)[\bar W_{\Sc_1}-\theta]\big)\mid\Sc_2\right]\Big|\leq \frac{\Gamma(\rho+1)\norm{\Dc^s f}_{{\rm Lip}_\rho(B_n)}\sqrt{n}\hat r_n^{s+\rho}}{\Gamma(s+\rho+1)},\nonumber\\
        &{\rm Var}\left(\sqrt{n}\big(\widehat{f}_s(\mathcal{S}_1, \mathcal{S}_2)-f(\theta)-\Dc^1f(\theta)[\bar W_{\Sc_1}-\theta]\big)\mid\Sc_2\right) \leq 2\sum_{k=2}^s\norm{\Dc^kf(\theta)}_{\rm S}^2\frac{\nu V^{k-1}}{k!(n-1)_{k-1}}\nonumber\\
        &\qquad+ 2\sum_{k=1}^s \norm{\Dc^s f}^2_{{\rm Lip}_\rho(B_n)}\frac{\Gamma(\rho+1)^2\hat r_n^{2s-2k+2\rho}}{\Gamma(s-k+\rho+1)^2}\frac{\nu V^{k-1}}{k!(n-1)_{k-1}}.
    \end{align}.
\end{theorem}
\begin{proof}[proof of Theorem~\ref{thm:condmoment}]
     Denote $\tilde\Delta^{(s)}_k$ as the symmetric $k$-linear form such that $$\tilde\Delta^{(s)}_k[v_1, \dots, v_k] := \Jc^{s-k} \Delta^{(s)} [ (\theta - \hat\theta_{\Sc_2})^{\otimes (s-k)}, v_1, \dots, v_k ].$$ From Proposition~\ref{prop:1.1}, we have 
     \begin{equation*}
         \widehat{f}_s(\mathcal{S}_1, \mathcal{S}_2)-f(\theta)=\sum_{k=0}^s Q_k/k!,
     \end{equation*} where $Q_k = (\Dc^kf(\theta)-\tilde\Delta^{(s)}_k)[\bar U^{(k)}(\theta)]$ for $1\leq k\leq s$ and $Q_0 = -\Jc^s\Delta^{(s)}[(\theta-\tilde\theta)^{\otimes s}]$. We note that $Q_0$ is independent of $\Sc_1 = \{W_1,\ldots,W_n\}$, and moreover,
     \begin{equation*}
         Q_1 = \frac{1}{n}\sum_{i=1}^n\Dc^1f(\theta)[W_i-\theta].
     \end{equation*}For $k\ge1$ define the scalar kernel,
\begin{equation*}
    t_k(w_1,\dots,w_k)
  := (\Dc^kf(\theta)-\tilde\Delta^{(s)}_k)\bigl[(w_1-\theta)\otimes\cdots\otimes(w_k-\theta)\bigr],\qquad w_1,\dots,w_k\in\Bb.
\end{equation*} Then $\Eb_{W_r} t_k(W_1,\dots,W_k) = 0$ for every $k\ge1$ and every $1\le r\le k$ where $\Eb_{W_r}$ denotes expectation with respect to $W_r$ only. Moreover, for $k\geq 1$,
\begin{equation*}
    Q_k
  = \frac{1}{(n)_k}
    \sum_{\substack{1\le j_1,\dots,j_k\le n\\\text{all distinct}}}
      t_k(W_{j_1},\dots,W_{j_k}),
\end{equation*} is U-statistic of order $k$ with the degenerate kernel $t_k$, and thus $\Eb[Q_k]=0$ for $k\geq 1$. Moreover, we note that
\begin{equation*}
    \Eb_{\Sc_1}[\widehat{f}_s(\mathcal{S}_1, \mathcal{S}_2)-f(\theta)]=Q_0 =-\frac{1}{\Gamma(s)}\int_0^1\Delta^{(s)}(t)[(\theta-\hat\theta_{\Sc_2})^{\otimes s}](1-t)^{s-1}\, dt.
\end{equation*} Along the segment $\hat\theta_{\Sc_2}+t(\theta-\hat\theta_{\Sc_2})$, $t\in[0,1]$, we note that $\norm{\Delta^{(s)}(t)}_{\rm S}\leq \norm{\Dc^s f}_{{\rm Lip}_\rho(B_n)}t^\rho\hat r_n^\rho.$ Therefore, 
\begin{align*}
    |Q_0|\leq \frac{\norm{\Dc^s f}_{{\rm Lip}_\rho(B_n)}\hat r_n^{s+\rho}}{\Gamma(s)}
       \int_0^1 t^{\rho}(1-t)^{s-1}\,dt=\frac{\Gamma(\rho+1)\norm{\Dc^s f}_{{\rm Lip}_\rho(B_n)}\hat r_n^{s+\rho}}{\Gamma(s+\rho+1)}.
\end{align*}This gives \eqref{thm:3.2:1}.

In order to control the conditional variance of $\widehat{f}_s(\mathcal{S}_1, \mathcal{S}_2)-f(\theta)$, we first note that
\begin{equation*}
    {\rm Var}_{\Sc_1}\left(\widehat{f}_s(\mathcal{S}_1, \mathcal{S}_2)-f(\theta)\right)=\sum_{k=1}^s\frac{1}{(k!)^2}{\rm Var}_{\Sc_1}(Q_k) + 2\sum_{1\leq k<l\leq s}\frac{1}{k!l!}{\rm Cov}_{\Sc_1}(Q_k,Q_l).
\end{equation*} Here, we used that $Q_0$ is independent of $\Sc_1$. Since $Q_k$ is a canonical $U$-statistics of order $k$ with kernel $h_k$, it follows from Hoeffding's variance formula (specialized to canonical kernel) that
\begin{equation*}
    {\rm Var}_{\Sc_1}(Q_k)=\binom{n}{k}^{-1}\tau_k^2,\quad \tau_k^2 = \Eb_{\Sc_1}[t_k^2(W_1,\ldots,W_k)].
\end{equation*} Let $\Ic_k$ be the collection of $k$-tuples $I=(i_1,\ldots,i_k)$ of distinct indices from $\{1,\ldots,n\}$ and let $W_I = (W_{i_1},\ldots,W_{i_k})$. Then, for any $1\leq k<l\leq s$,
\begin{equation*}
    {\rm Cov}_{\Sc_1}(Q_k,Q_l)=\Eb_{\Sc_1}(Q_kQ_l)=\frac{1}{(n)_k(n)_l}\sum_{I\in\Ic_k}\sum_{J\in\Ic_l}\Eb_{\Sc_1}[t_k(W_I)t_l(W_J)].
\end{equation*} Since $|J|=l>|I|=k$, there exists an index $\iota\in J\setminus I$, and thus,
\begin{equation*}
    \Eb_{\Sc_1}[t_k(W_I)t_l(W_J)]=\Eb_{\Sc_1}\big[t_k(W_I)\Eb_{W_\iota}[t_l(W_J)]\big]=0.
\end{equation*} This implies that ${\rm Cov}_{\Sc_1}(Q_k,Q_l)=0$ for any $1\leq k<l\leq s$. Hence, we have
\begin{equation}\label{eq:pf.thm3.2:1}
    {\rm Var}_{\Sc_1}\left(\widehat{f}_s(\mathcal{S}_1, \mathcal{S}_2)-f(\theta)\right)=\sum_{k=1}^s\frac{\tau_k^2}{k!(n)_k}.
\end{equation}
Finally, since the operator $\tilde\Delta^{(s)}_k$ is independent of $\Sc_1$, we have from Lemma~\ref{lem:varbound} that
\begin{equation}\label{eq:pf.thm3.2:2}
    \tau_k^2\leq \norm{\Dc^kf(\theta)-\tilde\Delta^{(s)}_k}_{\rm S}^2\nu V^{k-1}.
\end{equation} From triangle inequality we get $\norm{\Dc^kf(\theta)-\tilde\Delta^{(s)}_k}_{\rm S}\leq \norm{\Dc^kf(\theta)}_{\rm S}+\norm{\tilde\Delta^{(s)}_k}_{\rm S}$. Moreover, for fixed $0\leq k\leq s-1$,
\begin{equation}\label{eq:pf.thm3.2:3}
    \norm{\tilde\Delta^{(s)}_k}_{\rm S}\leq \norm{\Dc^s f}_{{\rm Lip}_\rho(B_n)}\frac{\hat r_n^{s-k+\alpha}}{\Gamma(s-k)}\int_0^1 t^\alpha(1-t)^{s-k-1}\,dt=\norm{\Dc^s f}_{{\rm Lip}_\rho(B_n)}\frac{\Gamma(\alpha+1)\hat r_n^{s-k+\alpha}}{\Gamma(s-k+\alpha+1)}.
\end{equation} Combining \eqref{eq:pf.thm3.2:1}---\eqref{eq:pf.thm3.2:3} gives \eqref{thm:3.2:2}. The inequality \eqref{thm:3.2:3} follows similarly from
\begin{equation*}
    \widehat{f}_s(\mathcal{S}_1, \mathcal{S}_2)-f(\theta)-\frac{1}{n}\sum_{i=1}^n\Dc^1f(\theta)[W_i-\theta]=Q_0 -\tilde\Delta^{(s)}_1[\bar U^{(1)}(\theta)]+ \sum_{k=2}^s Q_k/k!,
\end{equation*} and the (conditional) uncorrelated-ness of $\tilde\Delta^{(s)}_1[\bar U^{(1)}(\theta)]$ and $Q_2,\ldots,Q_k$.
\end{proof}
\subsection{Proof of Theorem~\ref{thm:moment-msmooth}}
We shall establish the moment bound on the one-sided cross-fitted estimators. Then, it follows from Jensen's inequality that
\begin{equation*}
    \Eb^{1/2}\big[(\widehat f_s -f(\theta))^2\big]\leq \frac{1}{2}\bigg(\Eb^{1/2}\big[(\widehat{f}_s(\mathcal{S}_1, \mathcal{S}_2) -f(\theta))^2\big]+\Eb^{1/2}\big[(\widehat{f}_s(\mathcal{S}_2, \mathcal{S}_1) -f(\theta))^2\big]\bigg).
\end{equation*} And also, since
\begin{align*}
    \widehat f_s - f(\theta) -\Dc^1 f(\theta)[\bar W_N-\theta] = \frac{1}{2}\bigg(\widehat f_s(\Sc_1,\Sc_2) - f(\theta) -\Dc^1 f(\theta)[\bar W_{\Sc_1}-\theta]\bigg)\\
    +\frac{1}{2}\bigg(\widehat f_s(\Sc_2,\Sc_1) - f(\theta) -\Dc^1 f(\theta)[\bar W_{\Sc_2}-\theta]\bigg),
\end{align*} we get
\begin{align*}
    \Eb^{1/2}\big[N(\widehat f_s -f(\theta)-\Dc^1 f(\theta)[\bar W_N-\theta])^2\big]\leq \Eb^{1/2}\big[n(\widehat f_s(\Sc_1,\Sc_2) - f(\theta) -\Dc^1 f(\theta)[\bar W_{\Sc_1}-\theta])^2\big]\\
    +\Eb^{1/2}\big[n(\widehat f_s(\Sc_2,\Sc_1) - f(\theta) -\Dc^1 f(\theta)[\bar W_{\Sc_2}-\theta])^2\big].
\end{align*}

From the standard bias and variance decomposition and the tower property, we have
\begin{align}\label{eq:pf.thm2.1:1}
    &\Eb[(\widehat{f}_s(\mathcal{S}_1, \mathcal{S}_2)-f(\theta))^2] = \big(\Eb[\widehat{f}_s(\mathcal{S}_1, \mathcal{S}_2)-f(\theta)]\big)^2 + {\rm Var}(\widehat{f}_s(\mathcal{S}_1, \mathcal{S}_2))\nonumber\\
    &\quad=\big(\Eb\big[\Eb[\widehat{f}_s(\mathcal{S}_1, \mathcal{S}_2)-f(\theta)\mid\Sc_2]\big]\big)^2+ {\rm Var}\big[\Eb[\widehat{f}_s(\mathcal{S}_1, \mathcal{S}_2)\mid\Sc_2]\big] + \Eb\big[{\rm Var}(\widehat{f}_s(\mathcal{S}_1, \mathcal{S}_2)\mid\Sc_2)\big]\nonumber\\
    &\qquad\leq \big(\Eb\big[\Eb[\widehat{f}_s(\mathcal{S}_1, \mathcal{S}_2)-f(\theta)\mid\Sc_2]\big]\big)^2+ \Eb\big[(\Eb[\widehat{f}_s(\mathcal{S}_1, \mathcal{S}_2)-f(\theta)\mid\Sc_2])^2\big] + \Eb\big[{\rm Var}(\widehat{f}_s(\mathcal{S}_1, \mathcal{S}_2)\mid\Sc_2)\big]\nonumber\\
    &\qquad\quad\leq 2\Eb\big[(\Eb[\widehat{f}_s(\mathcal{S}_1, \mathcal{S}_2)-f(\theta)\mid\Sc_2])^2\big] + \Eb\big[{\rm Var}(\widehat{f}_s(\mathcal{S}_1, \mathcal{S}_2)\mid\Sc_2)\big].
\end{align}
From Theorem~\ref{thm:condmoment}, we can bound the quantities in the last display. Let $\hat r_n = \norm{\hat\theta_{\Sc_2}-\theta}$ and $B_n=B(\theta,\hat r_n)$. First, we note that
\begin{equation*}
    \Eb\big[\big(\Eb[\widehat{f}_s(\mathcal{S}_1, \mathcal{S}_2)-f(\theta)\mid\Sc_2]\big)^2\big]\leq \frac{\Gamma(\rho+1)^2\big[\Eb\norm{\Dc^s f}_{{\rm Lip}_\rho(B_n)}^2\hat r_n^{2(s+\rho)}\big]}{\Gamma(s+\rho+1)^2}\leq \norm{\Dc^s f}_{{\rm Lip}_\rho(\Theta)}^2\Eb[\hat r_n^{2m}].
\end{equation*} Moreover, using \eqref{thm:3.2:2} with $k!(n)_k\geq n^k$, we get
\begin{align*}
    {\rm Var}(\widehat{f}_s(\mathcal{S}_1, \mathcal{S}_2)\mid \Sc_2)\leq\frac{2\nu}{n}\sum_{k=1}^s\norm{\Dc^kf(\theta)}_{\rm S}^2\bigg(\frac{V}{n}\bigg)^{k-1} + \frac{2\nu\norm{\Dc^s f}^2_{{\rm Lip}_\rho(B_n)}}{n}\sum_{k=1}^s\frac{\hat r_n^{2m-2k}(V/n)^{k-1}}{(s-k)!^2}\\
    \leq \frac{2\nu}{n}\sum_{k=1}^s\norm{\Dc^kf(\theta)}_{\rm S}^2\bigg(\frac{V}{n}\bigg)^{k-1} + \frac{2e\nu\norm{\Dc^s f}^2_{{\rm Lip}_\rho(B_n)}}{n}\bigg(\hat r_n^2\vee \frac{V}{n}\bigg)^{m-1}.
\end{align*}The last inequality holds since $\sum_{k=1}^s1/(s-k)!^2\leq \sum_{k=0}^\infty 1/k!=e$. Therefore, 
\begin{equation*}
    \Eb[{\rm Var}(\widehat{f}_s(\mathcal{S}_1, \mathcal{S}_2)\mid \Sc_2)]\leq \frac{2\nu}{n}\sum_{k=1}^s\norm{\Dc^kf(\theta)}_{\rm S}^2\bigg(\frac{V}{n}\bigg)^{k-1} + \frac{2e\nu\norm{\Dc^s f}^2_{{\rm Lip}_\rho(\Theta)}}{n}\Eb\bigg[\bigg(r_n^2\vee \frac{V}{n}\bigg)^{m-1}\bigg].
\end{equation*} Similar bound applies to $\widehat{f}_s(\mathcal{S}_2, \mathcal{S}_1)$, and combining these proves \eqref{eq:thm:moment-msmooth:1}. The moment bound in \eqref{eq:thm:moment-msmooth:2} follows similarly from \eqref{thm:3.2:3}. Precisely, we have
\begin{align}\label{eq:pf.thm2.1:2}
    &\Eb[\sqrt{n}(\widehat{f}_s(\mathcal{S}_1, \mathcal{S}_2)-f(\theta)-\Dc^1f(\theta)[\bar W_{\Sc_1}-\theta])^2]\nonumber\\
    &\quad\leq 2\Eb\bigg[\Eb\bigg(\sqrt{n}\big(\widehat{f}_s(\mathcal{S}_1, \mathcal{S}_2)-f(\theta)-\Dc^1f(\theta)[\bar W_{\Sc_1}-\theta]\big)\mid\Sc_2\bigg)^2\bigg]\nonumber\\
    &\qquad\qquad+\,\Eb\bigg[{\rm Var}\bigg(\sqrt{n}\big(\widehat{f}_s(\mathcal{S}_1, \mathcal{S}_2)-f(\theta)-\Dc^1f(\theta)[\bar W_{\Sc_1}-\theta]\big)\mid\Sc_2\bigg)\bigg].
\end{align} The first inequality in \eqref{thm:3.2:3} reads,
\begin{equation*}
    \Eb\bigg[\Eb\bigg(\sqrt{n}\big(\widehat{f}_s(\mathcal{S}_1, \mathcal{S}_2)-f(\theta)-\Dc^1f(\theta)[\bar W_{\Sc_1}-\theta]\big)\mid\Sc_2\bigg)^2\bigg]\leq n\norm{\Dc^s f}_{{\rm Lip}_\rho(\Theta)}^2\Eb[\norm{\theta-\hat\theta_{\Sc_2}}^{2m}],
\end{equation*} and the second with elementary calculation above implies that
\begin{align}\label{eq:pf.thm2.2:1}
    {\rm Var}\bigg(\sqrt{n}\big(\widehat{f}_s(\mathcal{S}_1, \mathcal{S}_2)-f(\theta)-\Dc^1f(\theta)[\bar W_{\Sc_1}-\theta]\big)\mid\Sc_2\bigg)\nonumber\\
    \leq 2\nu\sum_{k=2}^s\norm{\Dc^kf(\theta)}_{\rm S}^2\bigg(\frac{V}{n}\bigg)^{k-1} + 2e\nu\norm{\Dc^s f}^2_{{\rm Lip}_\rho(B_n)}\bigg(\hat r_n^2\vee \frac{V}{n}\bigg)^{m-1}\\
    \leq 2\nu\sum_{k=2}^s\norm{\Dc^kf(\theta)}_{\rm S}^2\bigg(\frac{V}{n}\bigg)^{k-1} + 2e\nu\norm{\Dc^s f}^2_{{\rm Lip}_\rho(\Theta)}\bigg(\hat r_n^2\vee \frac{V}{n}\bigg)^{m-1}\nonumber.
\end{align} Again, similar bound applies to $\widehat{f}_s(\mathcal{S}_2, \mathcal{S}_1)$ and this proves \eqref{eq:thm:moment-msmooth:2}.

\subsection{Proof of Theorem~\ref{thm:normality-msmooth}}\label{sec:pf.thm:normality-msmooth} Theorem~\ref{thm:normality-msmooth} follows from the conditional moment bound in Theorem~\ref{thm:condmoment} and Berry--Esséen  bound under conditional perturbation in Lemma~\ref{lem:BElemma}. Define
\begin{align*}
    \Av_n = \frac{\sqrt{N}\big(\widehat{f}_s-f(\theta(P))\big)}{\sigma_f},~\Bv_n = \frac{\sqrt{N}~\Dc^1f(\theta(P))[\bar W_N-\theta(P)]}{\sigma_f},\\
    \Cv_n = \frac{\sqrt{N/4}\big(\widehat f_s(\Sc_1,\Sc_2) - f(\theta) -\Dc^1 f(\theta)[\bar W_{\Sc_1}-\theta]\big)}{\sigma_f},\\
    \Dv_n = \frac{\sqrt{N/4}\big(\widehat f_s(\Sc_2,\Sc_1) - f(\theta) -\Dc^1 f(\theta)[\bar W_{\Sc_2}-\theta]\big)}{\sigma_f}.
\end{align*} Conventional Berry--Esséen  bound \citep{Berry1941GaussianApproximation,Esseen1942Liapunoff} implies that $\sup_{t\in\Real}|\Pb_P(\Bv_n\leq t)-\Phi(t)|\leq C\frac{\mu_{3,f}}{\sigma_f^3\sqrt{n}}$ for a universal constant $C>0$. Let $\Fc_n = \sigma(\Sc_2)$, $\Fc_n' = \sigma(\Sc_1)$ be a sigma-algebra generated by $\Sc_2$ and $\Sc_1$, respectively, and define $\Fc_n$- and $\Fc_n'$- measurable, respectively, events as
\begin{equation*}
    \Ec_n = \Set{\norm{\hat\theta_{\Sc_1}-\theta}\leq r_{n,P}(\delta_{n,P})},~ \Ec_n' = \Set{\norm{\hat\theta_{\Sc_2}-\theta}\leq r_{n,P}(\delta_{n,P})}.
\end{equation*} Assumption~\ref{asmp:theta_tilde_general} implies that $\Pb_P(\Ec_n\cap\Ec_n')\geq 1-2\delta_{n,P}$. Also, since $r_{n,P}(\delta_{n,P})\leq r_n$, we get $\hat\theta_{\Sc_1},\hat\theta_{\Sc_2}\in U$ on the event $\Ec_n\cap\Ec_n'$. From Theorem~\ref{thm:condmoment}, we note
\begin{equation*}
    |\Eb[\Cv_n\mid\Fc_n]|\mathbbm{1}(\Ec_n)\leq \frac{\norm{\Dc^s f}_{{\rm Lip}_\rho(U)}\Gamma(\rho+1)}{\sqrt{2}\sigma_f\Gamma(s+\rho+1)}\sqrt{n} r_n^{m}\leq C_0 \sqrt{n}r_n^m.
\end{equation*} Moreover, from \eqref{eq:pf.thm2.2:1}, we have
\begin{align*}
    \sigma_f^2\,{\rm Var}(\Cv_n\mid\Fc_n)\mathbbm{1}(\Ec_n)\leq\nu\sum_{k=2}^s\norm{\Dc^kf(\theta(P))}_{\rm S}^2\bigg(\frac{V}{n}\bigg)^{k-1} + e\nu\norm{\Dc^s f}^2_{{\rm Lip}_\rho(U)}\bigg(r_n^2\vee \frac{V}{n}\bigg)^{m-1} \\
    \leq C_0^2\bigg[\nu\sum_{k=2}^s\bigg(\frac{V}{n}\bigg)^{k-1} + e\nu\bigg(r_n^2\vee \frac{V}{n}\bigg)^{m-1}\bigg]
    \leq eC_0^2\,s\nu\Bigg[\frac{V}{n}\vee\bigg(\frac{V}{n}\bigg)^{m-1}\vee r_n^{2m-2} \Bigg]\\
    \lesssim_{s,\nu}\frac{\dk}{n} + \bigg(\frac{\dk}{n}\bigg)^{m-1}+r_n^{2m-2}.
\end{align*} Here, we may assume $\dk\leq n$, otherwise, the results in Theorem~\ref{thm:normality-msmooth} holds by simply taking a sufficiently large constant. Therefore, we get
\begin{equation*}
    {\rm Var}(\Cv_n\mid\Fc_n)\mathbbm{1}(\Ec_n)\lesssim_{s,\bar \nu} \frac{1}{\sigma_f^2}\bigg[\bigg(\frac{\dk}{n}\bigg)^{1\wedge(m-1)}+r_n^{2m-2}\bigg].
\end{equation*} Similarly, we get
\begin{equation*}
    |\Eb[\Dv_n\mid\Fc_n']|\mathbbm{1}(\Ec_n')\lesssim \sqrt{n}r_n^m,~{\rm Var}(\Dv_n\mid\Fc_n')\mathbbm{1}(\Ec_n')\lesssim_{s,\bar \nu} \frac{1}{\sigma_f^2}\bigg[\bigg(\frac{\dk}{n}\bigg)^{1\wedge(m-1)}+r_n^{2m-2}\bigg].
\end{equation*} Then, Lemma~\ref{lem:BElemma} applies and yields the desired results.

\subsection{Proof of Theorem~\ref{thm:moment-gevrey}}
We begin by stating Corollary~\ref{cor:condmoment} of Theorem~\ref{thm:condmoment} for Gevrey class functionals.

\begin{corollary}\label{cor:condmoment}
    Denote $\hat r_n=\norm{\hat\theta_{\Sc_2}-\theta}$. Let $f\in\Gc^\alpha(U)$ for some $U\supseteq B(\theta,r_n)$ and $R>0$ with $\norm{f}_{\Gc^\alpha(U),R}<\infty$. Then for any $s\geq 0$,
    \begin{align*}
        &\big|\Eb\big[\widehat{f}_s(\Sc_1,\Sc_2)-f(\theta)\mid\Sc_2\big]\big|\mathbbm{1}(\hat r_n\leq r_n)\leq \norm{f}_{\Gc^\alpha(U),R}R^{s+1}((s+1)!)^{\alpha-1}r_n^{s+1},\\
        &{\rm Var}(\widehat{f}_s(\Sc_1,\Sc_2)\mid \Sc_2)\mathbbm{1}(\hat r_n\leq r_n)\leq2\norm{f}_{\Gc^\alpha(U),R}^2\sum_{k=1}^s\bigg(R^{2k}(k!)^{2\alpha} + R^{2(s+1)}((s+1)!)^{2\alpha}\frac{r_n^{2s-2k+2}}{((s-k+1)!)^2}\bigg)\frac{\nu V^{k-1}}{k!(n)_k}. 
    \end{align*} Moreover, for any $s\geq 1$,
    \begin{align*}\label{thm:3.2:3}
        &\Big|\Eb\left[\sqrt{n}\big(\widehat{f}_s(\Sc_1,\Sc_2)-f(\theta)-\Dc^1f(\theta)[\bar W_{\Sc_1}-\theta]\big)\mid\Sc_2\right]\Big|\mathbbm{1}(\hat r_n\leq r_n)\leq \norm{f}_{\Gc^\alpha(U),R}\sqrt{n}R^{s+1}((s+1)!)^{\alpha-1}r_n^{s+1},\nonumber\\
        &{\rm Var}\left(\sqrt{n}\big(\widehat{f}_s(\Sc_1,\Sc_2)-f(\theta)-\Dc^1f(\theta)[\bar W_{\Sc_1}-\theta]\big)\mid\Sc_2\right) \mathbbm{1}(\hat r_n\leq r_n)\nonumber\nonumber\\
        &\qquad\leq 2\norm{f}_{\Gc^\alpha(U),R}^2\sum_{k=2}^sR^{2k}(k!)^{2\alpha-1}\frac{\nu V^{k-1}}{(n-1)_{k-1}}\\
        &\qquad\qquad+2\norm{f}_{\Gc^\alpha(U),R}^2\sum_{k=1}^sR^{2(s+1)}((s+1)!)^{2\alpha}\frac{r_n^{2s-2k+2}}{((s-k+1)!)^2}\frac{\nu V^{k-1}}{k!(n-1)_{k-1}}. 
    \end{align*}.
\end{corollary}
\begin{proof}[proof of Corollary~\ref{cor:condmoment}.]
    First, it follows from the definition of Gevrey semi-norm that $\norm{\Dc^k f(\theta)}_{\rm S}\leq \norm{f}_{\Gc^\alpha(U),R} R^k(k!)^\alpha$ for all $k\geq1$. Also,
    \begin{equation*}
        \norm{\Dc^s f}_{{\rm Lip}_1(B_n)}\leq \sup_{z\in B(\theta,r_n)}\norm{\Dc^{s+1}f(z)}_{\rm S}\leq \norm{f}_{\Gc^\alpha(U),R} R^{s+1}((s+1)!)^\alpha.
    \end{equation*} An application of Theorem~\ref{thm:condmoment} with $\rho=1$ yields the results.
\end{proof}
We set $s = \lfloor\log(en)\rfloor$, and first analyze the expectation bound of Corollary~\ref{cor:condmoment}. Since $(s+1)!\leq 2s^{s}$, we have
\begin{equation*}
    \sqrt{n}R^{s+1}(s+1)!^{\alpha-1}r_n^s\leq 2(Rr_n)\sqrt{n}\big(Rr_ns^{\alpha-1}\big)^s.
\end{equation*} Since $Rr_n\log^{\alpha-1}(en)\leq e^{-1}$ and $\log(n)\leq s\leq \log(en)$, we get
\begin{align*}
    \sqrt{n}\big(Rr_ns^{\alpha-1}\big)^s=\exp\bigg(s\log(Rr_ns^{\alpha-1})+\frac{1}{2}\log(n)\bigg)\leq \exp\bigg(-\frac{1}{2}\log(n)\bigg)=\frac{1}{\sqrt{n}}.
\end{align*} That is,
\begin{align}
    &|\Eb[\widehat{f}_s(\Sc_1,\Sc_2)-f(\theta)\mid\Sc_2]|\mathbbm{1}(\hat r_n\leq r_n)\leq 2\norm{f}_{\Gc^\alpha(U),R}\frac{Rr_n}{n},\label{eq:pf.thm:3.1:1}\\
    &\Big|\Eb\left[\sqrt{n}\big(\widehat{f}_s(\Sc_1,\Sc_2)-f(\theta)-\Dc^1f(\theta)[\bar W_{\Sc_1}-\theta]\big)\mid\Sc_2\right]\Big|\mathbbm{1}(\hat r_n\leq r_n)\leq 2\norm{f}_{\Gc^\alpha(U),R}\frac{Rr_n}{\sqrt{n}}.\label{eq:pf.thm:3.1:2}
\end{align}
Next, we analyze the variance bound. Let
\begin{align*}
    a_k &= 2\norm{f}_{\Gc^\alpha(U),R}^2R^{2k}(k!)^{2\alpha-1}\frac{\nu V^{k-1}}{(n-1)_{k-1}},~k\geq1\\
    b_k &= 2\norm{f}_{\Gc^\alpha(U),R}^2R^{2(s+1)}((s+1)!)^{2\alpha}\frac{r_n^{2s-2k+2}}{((s-k+1)!)^2}\frac{\nu V^{k-1}}{k!(n-1)_{k-1}},~k\geq 1.
\end{align*} Note that
\begin{equation*}
    \frac{a_{k+1}}{a_k} = \frac{R^2(k+1)^{2\alpha-1}V}{n-k},~k\geq 1
\end{equation*} Therefore, as long $2R^2\log^{2\alpha-1}(en)V\leq n$, we have $2a_{k+1}\leq a_k$ for $1\leq k\leq s-1$, and
\begin{align}\label{eq:a_kbound}
    &\sum_{k=1}^s a_k\leq a_1\sum_{k=1}^\infty2^{-k}=2a_1=\norm{f}_{\Gc^\alpha(U),R}^2R^2\nu,\nonumber\\
    &\sum_{k=2}^s a_k\leq a_2\sum_{k=1}^\infty2^{-k}=2a_2=\norm{f}_{\Gc^\alpha(U),R}^2\frac{2^{2\alpha}R^4\nu V}{n-1}\leq 2^{2\alpha+1}\norm{f}_{\Gc^\alpha(U),R}^2R^4\nu\frac{V}{n}.
\end{align} Also, we note that $k!(n)_k\geq n^k$, and thus $k!(n-1)_{k-1}\geq n^{k-1}.$ Moreover, 
\begin{equation*}
    \frac{(s+1)!}{(s-k+1)!}=(s+1)\dots(s-k+2)\leq (2s)s^{k-1}=2s^k,\quad (s+1)!\leq 2s^{s}.
\end{equation*} Therefore, $b_k$ $(k\geq1)$ is bounded as
\begin{equation*}
    b_k\leq 2^{2\alpha+1}\norm{f}_{\Gc^\alpha(U),R}^2R^{2(s+1)} s^{(2\alpha-2)s}s^{2k}\frac{(r_n^2)^{s-k+1}\nu V^{k-1}}{n^{k-1}}.
\end{equation*} This implies that
\begin{align*}
    \sum_{k=1}^sb_k\leq2^{2\alpha+1}\norm{f}_{\Gc^\alpha(U),R}^2R^{2(s+1)} s^{(2\alpha-2)s}s^2r_n^2\nu\sum_{k=1}^s(r_n^2)^{s-k}\bigg(\frac{s^2V}{n}\bigg)^{k-1}\\
    \leq 2^{2\alpha+1}\norm{f}_{\Gc^\alpha(U),R}^2R^{2(s+1)} s^{(2\alpha-2)s+3}r_n^2\nu\bigg(r_n^2\vee\frac{s^2V}{n}\bigg)^{s-1}\\
    \leq 2^{2\alpha+1}\norm{f}_{\Gc^\alpha(U),R}^2R^4 \nu s^{(2\alpha-2)s+3}\bigg(R^2r_n^2\vee\frac{R^2s^2V}{n}\bigg)^{s}
\end{align*}Since $R^2r_n^{2}\log^{2\alpha-2}(en)\vee\frac{R^2\log^{2\alpha}(en)V}{n}\leq e^{-2}$ and $\log(n)\leq s\leq \log(en)$, we get
\begin{align*}
    ((2\alpha-2)s+3)\log(s) +s\log\bigg(R^2r_n^2\vee\frac{R^2s^2V}{n}\bigg)=s\log\bigg(R^2r_n^{2}s^{2\alpha-2}\vee\frac{R^2s^{2\alpha}V}{n}\bigg)+3\log(s)\\
    \leq 3\log\log(en)-2\log(en)+1\leq -\frac{3}{2}\log(en)+4,
\end{align*} where the last inequality follows from $\sup_{x\geq 1}3\log\log(ex)-\frac{1}{2}\log(ex)=3\log(6)-3<3$. Therefore,
\begin{equation}\label{eq:b_kbound}
    \sum_{k=1}^s b_k\leq e^42^{2\alpha+1}\norm{f}_{\Gc^\alpha(U),R}^2R^4 \frac{\nu}{n^{3/2}}.
\end{equation} Combining \eqref{eq:a_kbound} and \eqref{eq:b_kbound} with the conditional variance bound in Corollary~\ref{cor:condmoment} yields that
\begin{align}
    {\rm Var}\left(\widehat{f}_s(\Sc_1,\Sc_2)-f(\theta)\mid\Sc_2\right)\mathbbm{1}(\hat r_n\leq r_n)&\leq \norm{f}_{\Gc^\alpha(U),R}^2R^2\frac{\nu}{n}+e^42^{2\alpha+1}\norm{f}_{\Gc^\alpha(U),R}^2R^4 \frac{\nu}{n^{5/2}}\nonumber\\
    &\leq e^42^{2\alpha+2}\norm{f}_{\Gc^\alpha(U),R}^2(R^2\vee R^4)\frac{\nu}{n},\label{eq:pf.thm:3.1:3}\\
    {\rm Var}\left(\sqrt{n}\big(\widehat{f}_s(\Sc_1,\Sc_2)-f(\theta)-\Dc^1f(\theta)[\bar W_{\Sc_1}-\theta]\big)\mid\Sc_2\right)&\mathbbm{1}(\hat r_n\leq r_n)\nonumber\\
    &\leq2^{2\alpha+1}\norm{f}_{\Gc^\alpha(U),R}^2R^4\bigg(\frac{\nu V}{n}+\frac{\nu}{n^{3/2}}\bigg)\label{eq:pf.thm:3.1:4}.
\end{align}  Therefore, the bias variance decomposition in \eqref{eq:pf.thm2.1:1} combined with \eqref{eq:pf.thm:3.1:1} and \eqref{eq:pf.thm:3.1:3} implies the first result, and \eqref{eq:pf.thm2.1:2} with \eqref{eq:pf.thm:3.1:2} and \eqref{eq:pf.thm:3.1:4} yields the second.
\subsection{Proof of Theorem~\ref{thm:normality-gevrey}} As in the proof of Theorem~\ref{thm:normality-msmooth}, we use the conditional moment bound and Berry--Esséen  bound under conditional perturbation in Lemma~\ref{lem:BElemma}. Recall
\begin{align*}
    \Av_n = \frac{\sqrt{N}\big(\widehat{f}_s-f(\theta(P))\big)}{\sigma_f},~\Bv_n = \frac{\sqrt{N}~\Dc^1f(\theta(P))[\bar W_N-\theta(P)]}{\sigma_f},\\
    \Cv_n = \frac{\sqrt{N/4}\big(\widehat f_s(\Sc_1,\Sc_2) - f(\theta) -\Dc^1 f(\theta)[\bar W_{\Sc_1}-\theta]\big)}{\sigma_f},\\
    \Dv_n = \frac{\sqrt{N/4}\big(\widehat f_s(\Sc_2,\Sc_1) - f(\theta) -\Dc^1 f(\theta)[\bar W_{\Sc_2}-\theta]\big)}{\sigma_f}.
\end{align*} where we have $\sup_{t\in\Real}|\Pb_P(\Bv_n\leq t)-\Phi(t)|\leq C\frac{\mu_{3,f}}{\sigma_f^3\sqrt{n}}$ for a universal constant $C>0$. Let $\Fc_n = \sigma(\Sc_2)$ be a sigma-algebra generated by $\Sc_2$ and let $\Ec_n = \set{\norm{\hat\theta_{\Sc_2}-\theta}\leq r_n=r_{n,P}(\delta_{n,P})}$ be $\Fc_n$-measurable event. From \eqref{eq:pf.thm:3.1:2}, we have
\begin{equation*}
    |\Eb[\Cv_n\mid\Fc_n]|\mathbbm{1}(\Ec_n)\leq \sqrt{2}\norm{f}_{\Gc^\alpha(U),R}\frac{Rr_n}{\sigma_f\sqrt{n}},
\end{equation*} and from \eqref{eq:pf.thm:3.1:4}, we get
\begin{equation*}
    {\rm Var}(\Cv_n\mid\Fc_n)\mathbbm{1}(\Ec_n)\leq \frac{2^{2\alpha}\norm{f}_{\Gc^\alpha(U),R}^2R^4}{\sigma_f^2}\bigg(\frac{\nu V}{n}+\frac{\nu}{n^{3/2}}\bigg).
\end{equation*} Similarly, for a sigma-algebra $\Fc_n' = \sigma(\Sc_1)$ and $\Fc_n'$-measurable $\Ec_n' = \set{\norm{\hat\theta_{\Sc_1}-\theta}\leq r_n=r_{n,P}(\delta_{n,P})}$, we have
\begin{align*}
    |\Eb[\Dv_n\mid\Fc_n']|\mathbbm{1}(\Ec_n')\leq \sqrt{2}\norm{f}_{\Gc^\alpha(U),R}\frac{Rr_n}{\sigma_f\sqrt{n}},\\
    {\rm Var}(\Dv_n\mid\Fc_n')\mathbbm{1}(\Ec_n')\leq \frac{2^{2\alpha}\norm{f}_{\Gc^\alpha(U),R}^2R^4}{\sigma_f^2}\bigg(\frac{\nu V}{n}+\frac{\nu}{n^{3/2}}\bigg).
\end{align*} Lemma~\ref{lem:BElemma} then applies and yields the desired results.
\subsection{Proof of Proposition~\ref{prop:1.1}}
Let $h=\theta-\tilde\theta$. For $0\leq k\leq s$, consider the real-valued function
\begin{equation*}
    g_k(y) = \Dc^k f(\tilde\theta + yh)[\bar U^{(k)}(\theta)],\quad y\in[0,1],
\end{equation*} with the convention $\Dc^0f(\cdot)[\bar U^{(0)}]=f(\cdot)$. Then $g_k$ is $(s-k)$-times differentiable, and its $j$:th $(0\leq j\leq s-k)$ derivative is given by
\begin{equation*}
    g_k^{(j)}(y):=\frac{\mathrm{d}^j}{\mathrm{d}y^j}g_k(y) = \Dc^{k+j} f(\tilde\theta + yh)[h^{\otimes j},\bar U^{(k)}(\theta)],\quad y\in[0,1].
\end{equation*}The $(s-k)$:th order Taylor expansion around $y=0$ gives that, for $y\in[0,1]$,
\begin{equation*}
    g_k(y) = \sum_{j=0}^{s-k}\frac{g_k^{(j)}(0)}{j!}y^j+\int_0^y \frac{\{g_k^{(s-k)}(t)-g_k^{(s-k)}(0)\}(y-t)^{s-k-1}}{(s-k-1)!}\,dt.
\end{equation*} Setting $y=1$ yields
\begin{equation*}
    \Dc^k f(\theta)[\bar U^{(k)}(\theta)]=\sum_{j=0}^{s-k}\frac{\Dc^{k+j}f(\tilde\theta)[h^{\otimes j}, \bar U^{(k)}(\theta)]}{j!}+\Jc^{s-k}
    \Delta^{(s)} [h^{\otimes (s-k)}, \bar U^{(k)}(\theta)],
\end{equation*} Now, we sum the first term over $k=0,\ldots,s$ with weights $1/k!$ to get
\begin{equation}\label{eq:pf.prop1.1:1}
    \sum_{k=0}^s\frac{1}{k!}\sum_{j=0}^{s-k}\frac{\Dc^{k+j} f(\tilde\theta)[h^{\otimes j},\bar U^{(k)}(\theta)]}{j!}=\sum_{l=0}^s\frac{1}{l!}\sum_{j=0}^l\binom{l}{j}\Dc^{l} f(\tilde\theta)[h^{\otimes j},\bar U^{(l-j)}(\theta)].
\end{equation} By expanding $W_i-\tilde\theta = W_i-\theta + \theta-\tilde\theta=W_i-\theta+h$, the definition of $\bar U^{(l)}(\tilde\theta)$ implies the identity that
\begin{equation}\label{eq:pf.prop1.1:2}
    \bar U^{(l)}(\tilde\theta) = \sum_{j=0}^l\binom{l}{j}h^{\otimes j}\otimes_{\rm s} \bar U^{(l-j)}(\theta),
\end{equation} Combining \eqref{eq:pf.prop1.1:1} and \eqref{eq:pf.prop1.1:2}, we get
\begin{equation*}
    \sum_{k=0}^s\frac{1}{k!}\sum_{j=0}^{s-k}\frac{\Dc^{k+j} f(\tilde\theta)[h^{\otimes j},\bar U^{(k)}(\theta)]}{j!} = \sum_{l=0}^s\frac{\Dc^{(l)}f(\tilde\theta)[\bar U^{(l)}(\tilde\theta)]}{l!},
\end{equation*}and therefore,
\begin{equation*}
    \sum_{k=0}^s\frac{\Dc^k f(\theta)[\bar U^{(k)}(\theta)]}{k!} = \sum_{k=0}^s\frac{\Dc^k f(\tilde\theta)[\bar U^{(k)}](\tilde \theta)}{k!}+\sum_{k=0}^s\frac{\Jc^{s-k}
    \Delta^{(s)} [h^{\otimes (s-k)}, \bar U^{(k)}(\theta)]}{k!}.
\end{equation*}

\subsection{Proof of Proposition~\ref{prop:lowerbound}}
We begin by stating the standard results in Gaussian mean estimation problem.
\begin{lemma}\label{lem:Gaussianmean_lower}
    Fix $n\geq 1$ and $\sigma^2>0$. Suppose $Z_1,\ldots,Z_n\sim\Nc(\mu,\sigma^2)$ for unknown $\mu\in\Real$. Then
    \begin{equation*}
        \inf_{\hat \delta_n}\sup_{\mu\in\Real} \Eb[(\hat \delta_n-\mu)^2]=\frac{\sigma^2}{n},
    \end{equation*} where the infimum is taken over all measurable estimators $\hat\delta_n = \hat \delta_n(Z_1,\ldots,Z_n)$.
\end{lemma}

Fix any $v\in\Bb$ with $\norm{v}=1$. By Hahn-Banach Theorem, there exists $u_0\in\Bb^*$ with $\norm{u_0}_*=1$ and $\ip{v,u_0}=1$. For a scalar $\mu\in\Real$, let a distribution $P_\mu$ on $\Bb$ be the law of $W$, defined as
\begin{equation*}
    W \overset{d}{=}Zv,\quad Z\sim\Nc(\mu,\nu_0).
\end{equation*} Then $\theta(P_\mu)=\Eb_{P_\mu}[W]=\mu v$. Since $\ip{W-\theta(P_\mu),u}=(Z-\mu)\ip{v,u}$, we have
\begin{equation*}
    \sup_{\norm{u}_*\leq 1}\Eb_{P_\mu}\abs{\ip{W-\theta(P_\mu),u}}^2=\Eb[(Z-\mu)^2]\sup_{\norm{u}_*\leq 1}\ip{v,u}^2\leq \nu_0\norm{v}^2\norm{u}_*^2=\nu_0.
\end{equation*} Therefore, $P_\mu\in\Pc_1(\nu_0)$. For fixed $n\geq 1$ and any measurable $T_n=T_n(W_1,\ldots,W_n)\in\Bb$, consider the real-valued estimator $\tilde\delta_n = \ip{T_n,u_0}$. Since $\ip{v,u_0}=1$, one has $\ip{T_n-\theta(P_t),u_0}=\tilde\delta_n-\ip{\mu v,u_0}=\tilde\delta_n-\mu$. This implies that
\begin{equation}\label{eq:prop:lowerbound:1}
    \sup_{P\in\Pc_1(\nu_0)}\sup_{\norm{u}_*\leq 1}n\,\Eb_P|\ip{T_n-\theta(P),u}|^2\geq \sup_{\mu\in\Real}n\,\Eb_{P_\mu}|\ip{T_n-\theta(P),u_0}|^2= \sup_{\mu\in\Real}n\,\Eb[(\tilde\delta_n-\mu)^2]\geq \nu_0,
\end{equation} where the last equality follows from Lemma~\ref{lem:Gaussianmean_lower} and that $\tilde\delta_n$ is a function of $Z_1,\ldots,Z_n$ where $W_i=Z_iv$ and $Z_i\sim\Nc(\mu,\nu_0)$. This proves \eqref{prop:lowerbound:1}.

To prove, \eqref{prop:lowerbound:2}, we define a distribution $Q_\mu$ as the law of $W$, defined by
\begin{equation*}
    W = Zv,\quad Z\sim\Nc(\mu,V_0).
\end{equation*} Then $\theta(Q_\mu)=\mu v$ and $\Eb_{Q_\mu}[\norm{W-\theta(Q_\mu)}^2]=\norm{v}^2\Eb[(Z-\mu)^2]=V_0$, and therefore $Q_\mu\in\Pc_2(V_0)$. Since $\norm{T_n-\theta(Q_\mu)}\geq \ip{T_n-\theta(Q_\mu),u_0}$, we get
\begin{equation*}
    \sup_{P\in\Pc_2(V_0)}n\,\Eb_P\norm{T_n-\theta(P)}^2\geq \sup_{\mu\in\Real}n\,\Eb_{Q_\mu}|\ip{T_n-\theta(Q_\mu),u_0}|^2\geq V_0,
\end{equation*} where the last inequality follows from \eqref{eq:prop:lowerbound:1}. This proves \eqref{prop:lowerbound:2}.
\section{Proofs of Results in Section~\ref{sec:appl}}
\subsection{Proof of Theorem~\ref{thm:PM1}}
We begin by noting that the inequality in \eqref{eq:thm:PM1:1} is invariant to the scalings $\norm{\eta_1}_{\Sigma^{-1}}$ and $\norm{\eta_2}_{\Sigma^{-1}}$. Therefore, we may assume that $\norm{\eta_1}_{\Sigma^{-1}}=\norm{\eta_2}_{\Sigma^{-1}}=1$, and thus $\sigma_\omega = \bar\sigma_\omega$. The decomposition in Proposition~\ref{prop:1.1} with $s=s_n$ implies that 
\begin{align*}
    \frac{\sqrt{N}(\eta_1^\top \tilde\Omega \eta_2-\eta_1^\top \Sigma^{-1} \eta_2)}{\sigma_\omega} = \Iv\Fv_n+ \frac{\sqrt{N/4}(\Rv_{1,n}+\Rv_{2,n})}{\sigma_\omega},
\end{align*} where $\Iv\Fv_n=\frac{1}{\sigma_\omega\sqrt{N}}\sum_{i=1}^N\eta_1^\top\Sigma^{-1}(X_iX_i^\top-\Sigma)\Sigma^{-1}\eta_2 ,$ and 
\begin{equation*}
    \Rv_{1,n}=\sum_{k=2}^{s_n}\frac{\Dc^k\omega(\Sigma)[\bar U^{(k)}(\Sigma)]}{k!}-\bigg(\Jc^s\Delta^{(s_n)}[(\Sigma-\tilde\Sigma_2)^{\otimes s_n}]+\sum_{k=1}^{s_n}\frac{\tilde\Delta^{(s_n)}_k[\bar U^{(k)}(\tilde\Sigma_2)]}{k!}\bigg),
\end{equation*} and $\Rv_{2,n}$ is analogously defined by interchanging the roles of $\Sc_1$ and $\Sc_2$. Under \ref{asmp:momentX}, we have $\Iv\Fv_n\xrightarrow{D}\Nc(0,1)$ as $n\to\infty$ by central limit theorem. Moreover, from the conventional Berry--Esséen  bound under $(2+\delta)$:moments (see, e.g., \citep{Bikelis1966RemainderCLT,Shevtsova2020LowerBoundsNonuniformCLT}), we have that
\begin{equation*}
    \sup_{t\in\Real}|\Pb(\Iv\Fv_n\geq t) - (1-\Phi(t))|\leq C_{\rm BE}\frac{\Eb[|\eta_1^\top\Sigma^{-1}(XX^\top-\Sigma)\Sigma^{-1}\eta_2|^{(q_x/2)\wedge 3}]}{\sigma_\omega^{(q_x/2)\wedge 3}N^{((q_x/2)\wedge 3-1)/2}}=:{\rm BE}_n,
\end{equation*} for some universal constant $C_{\rm BE}>0$. From Minkowski's inequality and \ref{asmp:momentX}, we have
\begin{align*}
    &\Eb^{1/q}[|\eta_1^\top\Sigma^{-1}(XX^\top-\Sigma)\Sigma^{-1}\eta_2|^q]\\
    &\quad \leq \Eb^{1/q}[|\eta_1^\top\Sigma^{-1}X|^q |\eta_1^\top\Sigma^{-1}X|^q ]+|\eta_1^\top\Sigma^{-1}\eta_2|\leq K_x^2 +1\leq 2K_x^2.
\end{align*} This implies that
\begin{equation*}
    {\rm BE}_n\leq 8C_{\rm BE}\frac{K_x^{q_x\wedge 6}}{\bar\sigma_\omega^{(q_x/2)\wedge 3}n^{((q_x/2)\wedge 3-1)/2}}.
\end{equation*} We now control the remainder $\Rv_{1,n}$ whose conditional expectation and variance can be bounded as
\begin{align}
    \Eb_{\Sc_1}[\Rv_{1,n}] &= -\Jc^s\Delta^{(s_n)}[(\Sigma-\tilde\Sigma)^{\otimes s_n}],\label{eq:thmPM1.ebound}\\
    {\rm Var}_{\Sc_1}(\Rv_{1,n})&={\rm Var}_{\Sc_1}\bigg(\sum_{k=2}^{s_n}\frac{\Dc^k\omega(\Sigma)[\bar U^{(k)}(\Sigma)]}{k!}-\sum_{k=1}^{s_n}\frac{\tilde\Delta^{(s_n)}_k[\bar U^{(k)}(\tilde\Sigma_2)]}{k!}\bigg)\nonumber\\
    &={\rm Var}_{\Sc_1}\bigg(\tilde\Delta^{(s_n)}_1[\bar U^{(1)}(\tilde\Sigma_2)]\bigg)+\sum_{k=2}^{s_n}{\rm Var}_{\Sc_1}\bigg(\frac{\Dc^k\omega(\Sigma)[\bar U^{(k)}(\Sigma)]}{k!}-\frac{\tilde\Delta^{(s_n)}_k[\bar U^{(k)}(\tilde\Sigma_2)]}{k!}\bigg)\nonumber\\
    &\leq 2\sum_{k=2}^{s_n}{\rm Var}\bigg(\frac{\Dc^k\omega(\Sigma)[\bar U^{(k)}(\Sigma)]}{k!}\bigg) + 2\sum_{k=1}^{s_n}{\rm Var}_{\Sc_1}\bigg(\frac{\tilde\Delta^{(s_n)}_k[\bar U^{(k)}(\tilde\Sigma_2)]}{k!}\bigg),\label{eq:thmPM1.vbound}
\end{align} where the decomposition of variance follows from the fact that $T_k[\bar U^{(k)}]$ and $T_{k'}[\bar U^{(k')}]$ are uncorrelated for non-stochastic symmetric linear form $T_k$ and $T_{k'}$ with $k\neq k'$. Now, Theorem~\ref{thm:B1} implies that for $1\leq k\leq s$, we have
\begin{equation*}
    {\rm Var}\left(\frac{\Dc^k\omega(\Sigma)[\bar U^{(k)}(\Sigma)]}{k!}\right)\leq k!(2K_x^2)^{2k}\frac{d^{k-1}}{(n)_k}\leq (4kK_x^4)^{k}\frac{d^{k-1}}{n^k},
\end{equation*} where the last inequality follows from $\binom{n}{k}=\frac{n(n-1)\cdots(n-k+1)}{k(k-1)\cdots1}\geq(\frac{n}{k})^k$ for $1\leq k\leq n$. Let $a_k = (4kK_x^4)^{k}\frac{d^{k-1}}{n^k}$ for $1\leq k\leq s_n-1$, then
\begin{equation*}
    \frac{a_{k+1}}{a_k}=\frac{4K_x^4d}{n}\frac{(k+1)^{k+1}}{k^k}=\frac{4K_x^4d}{n}(k+1)\big(1+\frac{1}{k}\big)^k\leq \frac{4eK_x^4s_nd}{n}.
\end{equation*} As long as $8eK_x^4s_nd\leq n$, which holds under the choice of $s_n=\lfloor\log(en)\rfloor$ and the assumptions of Theorem~\ref{thm:PM1}, we have
\begin{equation}\label{eq:varpart1}
    \sum_{k=2}^s{\rm Var}\left(\frac{\Dc^k\omega(\theta)[\bar U^{(k)}]}{k!}\right)\leq\sum_{k=2}^{s_n}a_k\leq a_2\sum_{k=0}^\infty\big(\frac{1}{2}\big)^k\leq 2a_2=\frac{128K_x^8d}{n^2}.
\end{equation} Meanwhile, an application of Theorem~\ref{thm:B2} yields that for $1\leq k\leq s_n$,
\begin{align*}
    {\rm Var}_{\Sc_1}\left(\frac{\tilde\Delta^{(s_n)}_k[\bar U^{(k)}(\tilde\Sigma_2)]}{k!}\right)\leq \frac{(4s_nK_x^2)^{2k}\norm{\Sigma^{-1/2}\tilde\Sigma_2\Sigma^{-1/2}-I_d}_{\rm op}^{2(s_n-k+1)}}{1\wedge\lambda_{\rm min}^{2s_n+4}(\Sigma^{-1/2}\tilde\Sigma_2\Sigma^{-1/2})}\frac{d^{k-1}}{k!(n)_k}\\
    \leq \frac{(4s_nK_x^2)^{2k}\norm{\Sigma^{-1/2}\tilde\Sigma_2\Sigma^{-1/2}-I_d}_{\rm op}^{2(s_n-k+1)}}{1\wedge\lambda_{\rm min}^{2s_n+4}(\Sigma^{-1/2}\tilde\Sigma_2\Sigma^{-1/2})}\frac{d^{k-1}}{n^k}.
\end{align*} The last inequality follows from
$\frac{n^k}{(n)_k}=\frac{n}{n}\cdot\frac{n}{n-1}\cdots\frac{n}{n-k+1}\leq k!$. Under Assumption~\ref{asmp:Sigmapilot}, we have
\begin{equation*}
    \Pb(\Ec_1)\geq 1-\delta_n,\quad \Ec_1 = \Set{\norm{\Sigma^{-1/2}\tilde\Sigma_2\Sigma^{-1/2}-I_d}_{\rm op}\leq r_n(\delta_n)}.
\end{equation*} On $\Ec_1$, we have $\lambda_{\rm min}(\Sigma^{-1/2}\tilde\Sigma_2\Sigma^{-1/2})\geq (1- \norm{\Sigma^{-1/2}\tilde\Sigma_2\Sigma^{-1/2}-I_d}_{\rm op})_+\geq (1-r_n(\delta_n))_+\geq 1/2$ as long as $r_n(\delta_n)\leq 1/2$. Therefore, we have
\begin{align}\label{eq:varpart2}
    &\sum_{k=1}^{s_n}{\rm Var}_{\Sc_1}\left(\frac{\tilde\Delta^{(s)}_k[\bar U^{(k)}]}{k!}\right)\mathbbm{1}(\Ec_1)\leq \sum_{k=1}^{s_n}\frac{(4s_nK_x^2)^{2k}r_n(\delta_n)^{2(s_n-k+1)}}{(1/2)^{2s_n+4}}\frac{d^{k-1}}{n^k}\nonumber\\
    &= \frac{2^{2s_n+4}(4s_nK_x^2)^2r_n(\delta_n)^2}{n}\sum_{k=1}^{s_n}r_n(\delta_n)^{2(s_n-k)}\bigg(\frac{16K_x^4s_n^2d}{n}\bigg)^{k-1}\nonumber\\
    & \leq \frac{2^{2s_n+4}(4s_nK_x^2)^2r_n(\delta_n)^2}{n}s_n\bigg(r_n(\delta_n)^2\vee \frac{16K_x^4s_n^2d}{n}\bigg)^{s_n-1}\nonumber\\
    &\leq \frac{256s_n^3K_x^4}{n}\bigg(4r_n(\delta_n)^2\vee\frac{64K_x^4s_n^2d}{n}\bigg)^{s_n}.
\end{align} Combining \eqref{eq:thmPM1.vbound}, \eqref{eq:varpart1}, and \eqref{eq:varpart2}, we get
\begin{equation*}
    {\rm Var}_{\Sc_1}(\sqrt{n}\Rv_{1,n}/\sigma_\omega)\mathbbm{1}(\Ec_1)\leq \frac{1}{\bar\sigma_\omega^2}\bigg[\frac{128K_x^8d}{n}+256s_n^3K_x^4\bigg(4r_n(\delta_n)^2\vee\frac{64K_x^4s_n^2d}{n}\bigg)^{s_n}\bigg],
\end{equation*} as long as
$8eK_x^4s_nd\leq n$. In order to control the conditional expectation of $\Rv_{1,n}$, we apply Theorem~\ref{thm:B2} with \eqref{eq:thmPM1.ebound} which yields
\begin{equation*}
    \big|\Eb_{\Sc_1}[\Rv_{1,n}]\big|\mathbbm{1}(\Ec_1)\leq \frac{\norm{I-\Sigma^{-1/2}\tilde\Sigma_2\Sigma^{-1/2}}_{\rm op}^{s_n+1}}{1\wedge\lambda_{\rm min}^{s_n+2}(\Sigma^{-1/2}\tilde\Sigma_2\Sigma^{-1/2})}\mathbbm{1}(\Ec_1)\leq 2(2r_n(\delta_n))^{s_n+1}.
\end{equation*}Analogous moment bound applies to $\Rv_{2,n}$. Precisely, for $\Ec_2 = \set{\norm{\Sigma^{-1/2}\tilde\Sigma_1\Sigma^{-1/2}-I_d}_{\rm op}\leq r_n(\delta_n)}$, we have
\begin{align*}
    &\big|\Eb_{\Sc_2}[\Rv_{2,n}]\big|\mathbbm{1}(\Ec_2)\leq 2(2r_n(\delta_n))^{s_n+1},\\
    &{\rm Var}_{\Sc_2}(\sqrt{n}\Rv_{2,n}/\sigma_\omega)\mathbbm{1}(\Ec_2)\leq \frac{1}{\bar\sigma_\omega^2}\bigg[\frac{128K_x^8d}{n}+256s_n^3K_x^4\bigg(4r_n(\delta_n)^2\vee\frac{64K_x^4s_n^2d}{n}\bigg)^{s_n}\bigg].
\end{align*}

Combining these with Lemma~\ref{lem:BElemma} implies that
\begin{align}\label{eq:thm:PM1:final1}
    &\sup_{t\in\Real}\bigg|\Pb\bigg(\frac{\sqrt{N}(\eta_1^\top\tilde\Omega\,\eta_2-\eta_2^\top\Sigma^{-1}\eta_2)}{\sigma_\omega}\geq t\bigg)-(1-\Phi(t))\bigg|\leq 8C_{\rm BE}\frac{K_x^{q_x\wedge 6}}{\bar\sigma_\omega^{(q_x/2)\wedge 3}N^{((q_x/2)\wedge 3-1)/2}}\nonumber\\
    &+\frac{\sqrt{n}}{\bar\sigma_\omega\sqrt{\pi/2}}(2r_n(\delta_n))^{s_n+1}\nonumber\\
    &+\frac{1}{\bar\sigma_\omega^{2/3}}\left[\bigg(\frac{16K_x^8}{\pi}\bigg)^{1/3}\bigg(\frac{d}{n}\bigg)^{1/3}+\bigg(\frac{64K_x^4}{\pi}\bigg)^{1/3}s_n\bigg(4r_n(\delta_n)^2\vee\frac{64K_x^4s_n^2d}{n}\bigg)^{s_n/3}\right]+2\delta_n.
\end{align} We let $s_n=\lfloor\log(en)\rfloor$, so that
\begin{equation}\label{eq:PM:s_n}
    s_n \geq \left\lceil\frac{\log(n)}{\log(1/\{2r_n(\delta_n)\})}\right\rceil\vee 1,
\end{equation} since $r_n(\delta_n)\leq 1/(2e)$. It follows from \eqref{eq:PM:s_n} that 
\begin{equation}\label{eq:thm:PM1:final2}
    \sqrt{n}(2r_n(\delta_n))^{s_n+1}=\sqrt{n}\exp\big(-(s_n+1)\log(1/\{2r_n(\delta_n)\})\big)\leq \sqrt{n}\exp\big(-\log(n)\big)=1/\sqrt{n}.
\end{equation} Since
\begin{equation*}
    4r_n(\delta_n)^2\vee\frac{64K_x^4\log^2(en)d}{n}\leq e^{-2},
\end{equation*} we have
\begin{equation}\label{eq:thm:PM1:final3}
s_n\bigg(4r_n(\delta_n)^2\vee\frac{64K_x^4s_n^2d}{n}\bigg)^{s_n/3}\leq s_n\exp\bigg(-\frac{2s_n}{3}\bigg)\leq \frac{\log(en)}{(en)^{2/3}}\leq \frac{2}{\sqrt{n}},
\end{equation} where the last inequality follows from $\sup_{x\geq 1}\frac{\log(ex)}{e^{2/3}x^{1/6}}=\frac{6}{e^{3/2}}<2$. Combining \eqref{eq:thm:PM1:final1}, \eqref{eq:thm:PM1:final2}, and \eqref{eq:thm:PM1:final3}, we have
\begin{align*}
    \sup_{t\in\Real}\bigg|\Pb\bigg(\frac{\sqrt{N}(\eta_1^\top\tilde\Omega\,\eta_2-\eta_2^\top\Sigma^{-1}\eta_2)}{\sigma_\omega}\geq t\bigg)-(1-\Phi(t))\bigg|\leq \\
    \Ck\bigg[\frac{K_x^{q_x\wedge 6}}{\bar\sigma_\omega^{(q_x/2)\wedge 3}N^{((q_x/2)\wedge 3-1)/2}}+\frac{1}{\bar\sigma_\omega\sqrt{n}}+\frac{1}{\bar\sigma_\omega^{2/3}}\bigg(K_x^{8/3}\bigg(\frac{d}{n}\bigg)^{1/3}+\frac{K_x^{4/3}}{\sqrt{n}}\bigg)\bigg]+2\delta_n.
\end{align*} Since $K_x\geq 1$ and $N=2n$, this implies the conclusion.
\subsection{Proof of Corollary~\ref{cor:PM1}} From the rotation invariance, the $L_4-L_2$ moment ratio of $X$ can be bounded as
\begin{align}\label{eq:momentratio}
    &\sup_{u\in\Sb^{d-1}}\frac{\Eb^{1/4}[|u^\top X|^4]}{\Eb^{1/2}[|u^\top X|^2]}=\sup_{u\in\Real^d\setminus\{0_d\}}\frac{\Eb^{1/4}[|u^\top X|^4]}{\Eb^{1/2}[|u^\top X|^2]}\nonumber\\
    &\quad =\sup_{u\in\Real^d\setminus\{0_d\}}\frac{\Eb^{1/4}[|u^\top\Sigma^{-1/2} X|^4]}{\Eb^{1/2}[|u^\top\Sigma^{-1/2} X|^2]}=\sup_{u\in\Sb^{d-1}}\frac{\Eb^{1/4}[|u^\top\Sigma^{-1/2} X|^4]}{\Eb^{1/2}[|u^\top\Sigma^{-1/2} X|^2]}\leq K_x.
\end{align} Therefore, Theorem~1.3 of \cite{OliveiraRico2024ImprovedCovariance} applies with $\eta=0$ and $\delta=1/n$ implies that there exists a measurable estimator $\widehat{\mathsf{E}}_{1}=\widehat{\mathsf{E}}_1(X_1,\ldots,X_n)\in\mathbb{R}^{d\times d}$ and a universal constant $c>0$ such that, whenever $n \ge c(r(\Sigma)+\log(en))$,
\begin{equation}\label{cond:robcov}
\mathbb{P}\!\left(
\bigl\|\widehat{\mathsf{E}}_1-\Sigma\bigr\|_{\rm op}
\le cK_x^2\,\|\Sigma\|_{\rm op}\sqrt{\frac{r(\Sigma)+\log(en)}{n}}
\right)\ge 1-\frac{1}{n},
\end{equation} Since $\norm{\Sigma^{-1/2}\widehat{\mathsf{E}}_1\Sigma^{-1/2}-I_d}_{\rm op}\leq \norm{\Sigma}_{\rm op}^{-1} \norm{\widehat{\mathsf{E}}_1-\Sigma}_{\rm op}$ and $r(\Sigma)\leq d$, $\widehat{\mathsf{E}}_1$ satisfies Assumption~\ref{asmp:Sigmapilot} with 
\begin{equation*}
    r_n^* = cK_x^2\sqrt{\frac{d+\log(en)}{n}}\quad \mbox{and}\quad \delta_n = \frac{1}{n},
\end{equation*} provided that $n\geq c(d+\log(en))$. Let $\hat\Omega^\dagger$ be the estimator in \eqref{eq:precisionmatrixest.} with $\tilde\Sigma_2 = \widehat{\mathsf{E}}_1$ and $s_n = \lfloor\log(en)\rfloor$, then the second part of Theorem~\ref{thm:PM1} applies and yields the conclusion.

\subsection{Proof of Corollary~\ref{cor:PM2}} 
\paragraph{Under \ref{asmp:sg}} We begin with the standard results, see e.g., Theorem~4.7.1. of \cite{Vershynin2018HighDimProb} or Theorem 1 of \cite{KoltchinskiiLounici2017Concentration}; for some universal constant $c>0$,
\begin{equation}\label{cond:samcov_sg}
\mathbb{P}\!\left(
\bigl\|\Sigma^{-1/2}\hat\Sigma_2\Sigma^{-1/2}-I_d\bigr\|_{\rm op}
\le cK_{\rm sg}^2\sqrt{\frac{d+\log(e/\delta)}{n}}
\right)\ge 1-\delta,\quad\forall\delta\in(0,1).
\end{equation} Taking $\delta = 1/n$ above, the sample gram matrix $\hat\Sigma_2$ satisfies the Assumption~\ref{asmp:Sigmapilot} with
\begin{equation*}
    r_n^* = cK_{\rm sg}^2\sqrt{\frac{d+\log(en)}{n}}\quad \mbox{and}\quad \delta_n = \frac{1}{n}.
\end{equation*} Assumption~\ref{asmp:sg} implies that Assumption~\ref{asmp:momentX} holds with $q_x=6$ an $K_x = K_{\rm sg}\sqrt{12e}$. Hence, an application of the second part of Theorem~\ref{thm:PM1} with $q_x=6$, where the rate $1/\sqrt{n}$ can be absorbed to $(d/n)^{1/3}$, yields the results.
\paragraph{Under \ref{asmp:momentX} with $q_x>4$.} 
An application of Proposition~\ref{prop:samcov}, where the assumptions therein are met by \eqref{eq:momentratio}, yields that there exists a constant $\Ck = \Ck(K_x,q_x)>0$ such that
if $n\geq \Ck d$,
\begin{equation*}
\mathbb{P}\!\left(
\bigl\|\Sigma^{-1/2}\hat\Sigma_2\Sigma^{-1/2}-I_d\bigr\|_{\rm op}
\le \Ck\bigg[\sqrt{\frac{d+\log(e/\delta)}{n}}+\frac{d}{n^{1-2/q_x}\delta^{2/q_x}}
\bigg]\right)\ge 1-\delta,\quad\forall\delta\in(0,1).
\end{equation*} Set
\begin{equation*}
    \delta = \bigg(\frac{d}{n^{1-2/q_x}}\bigg)^{\frac{q_x}{q_x+2}},
\end{equation*} then $\hat\Sigma_2$ satisfies Assumption~\ref{asmp:Sigmapilot} with 
\begin{equation}\label{cond:samcov}
    r_n^* = \Ck\bigg[\sqrt{\frac{d+\log(en)}{n}}+\bigg(\frac{d}{n^{1-2/q_x}}\bigg)^{\frac{q_x}{q_x+2}}\bigg]\quad\mbox{and}\quad\delta_n=\bigg(\frac{d}{n^{1-2/q_x}}\bigg)^{\frac{q_x}{q_x+2}}.
\end{equation} for a possibly different constant $\Ck = \Ck(K_x,q_x)>0$. Now, an application of Theorem~\ref{thm:PM1} implies that the asymptotic normality,
\begin{equation*}
        \frac{\sqrt{n}(\eta_1^\top\widehat\Omega\,\eta_2-\eta_1^\top\Sigma^{-1}\eta_2)}{\sigma_\omega}\xrightarrow{D}\Nc(0,1),\quad\mbox{as}\quad n\to\infty,
\end{equation*} hold for all $\eta_1,\eta_2\in\Real^d\setminus\{0_d\}$ if $d=o(n^{1-2/q_x})$. Also, as long as
\begin{equation}\label{eq:cor3.3.cond}
    (2r_n^*)^2\vee (64K_x^4s^2d/n)\leq e^{-2},
\end{equation} there exists a universal constant $\Ck'>0$ such that
\begin{align*}\label{eq:cor3.3.concl}
    {\rm BE}_n(\eta_1^\top\tilde\Omega\,\eta_2)\leq \frac{\Ck' K_x^{\bar q_x}}{(\bar \sigma_\omega\wedge 1)^{\bar q_x/2}}\bigg[\frac{1}{n^{\bar q_x/4-1}}+\bigg(\frac{d}{n}\bigg)^{1/3}\bigg]+\bigg(\frac{d}{n^{1-2/q_x}}\bigg)^{\frac{q_x}{q_x+2}}\\
    \leq \frac{\Ck' K_x^{\bar q_x}}{(\bar \sigma_\omega\wedge 1)^{\bar q_x/2}}\bigg[\frac{1}{n^{\bar q_x/4-1}}+\bigg(\frac{d}{n}\bigg)^{1/3}+\bigg(\frac{d}{n^{1-2/q_x}}\bigg)^{\frac{q_x}{q_x+2}}\bigg].
\end{align*} Otherwise, if  $(2r_n^*)^2\vee (64K_x^4s^2d/n)> e^{-2}$, then the conclusion of Corollary~\ref{cor:PM2} holds by simply taking a sufficiently large constant. This proves the corollary.
\subsection{Proof of Theorem~\ref{thm:PP1}}

Without loss of generality, we may assume that $\norm{\eta}_{\Sigma^{-1}}=1$, and thus $\sigma_\beta = \bar\sigma_\beta$. An application of Proposition~\ref{prop:1.1} yields that 
\begin{align*}
    &\frac{\sqrt{N}(\eta^\top \hat\beta-\eta^\top\beta)}{\sigma_\beta} =\Iv\Fv_n + \frac{\sqrt{N/4}(\Rv_{1,n}+\Rv_{2,n})}{\sigma_\beta},\quad \mbox{where}\\
    &\Iv\Fv_n =\frac{1}{\sigma_\beta\sqrt{N}}\sum_{i=1}^N \eta^\top\Sigma^{-1}X_i(Y_i-X_i^\top\beta),\\
    &\Rv_{1,n} = \sum_{k=2}^{s_n}\frac{\Dc^k\beta_\eta(\Sigma,\Gamma)[\bar U^{(k)}(\Sigma,\Gamma)]}{k!}-\bigg(\Jc^{s_n}\Delta^{(s_n)}[(\Sigma-\tilde\Sigma_2,\Gamma-\tilde\Gamma_2)^{\otimes s_n}]+\sum_{k=1}^{s_n}\frac{\tilde\Delta^{(s_n)}_k[\bar U^{(k)}(\tilde\Sigma_2,\tilde\Gamma_2)]}{k!}\bigg),
\end{align*} and $\Rv_{2,n}$ is defined analogously interchanging the roles of $\Sc_1$ and $\Sc_2$. The central limit theorem implies that $\Iv\Fv_n\xrightarrow{D}\Nc(0,1)$ as $n\to\infty$. Let $q_{x\varepsilon}=(1/q_x+1/q_\varepsilon)^{-1}$. From the quantitative CLT results under $(2+\delta)$:moments \citep{Bikelis1966RemainderCLT,Shevtsova2020LowerBoundsNonuniformCLT}, we have that
\begin{equation*}
    \sup_{t\in\Real}|\Pb(\Iv\Fv_n\leq t) - \Phi(t)|\leq C_{\rm BE}\frac{\Eb[|\eta^\top\Sigma^{-1}X(Y-X^\top\beta)|^{q_{x\varepsilon}\wedge 3}]}{\sigma_\beta^{q_{x\varepsilon}\wedge 3}N^{(q_{x\varepsilon}\wedge 3-1)/2}}=:{\rm BE}_n,
\end{equation*} for a universal constant $C_{\rm BE}>0$. Jensen's inequality and H\"older's inequality imply that for $1\leq q\leq q_{x\varepsilon}$,
\begin{align*}
    &\Eb^{1/q}[|\eta^\top\Sigma^{-1}X(Y-X^\top\beta)|^q]\leq \Eb^{1/q_{x\varepsilon}}[|\eta^\top\Sigma^{-1}X(Y-X^\top\beta)|^{q_{x\varepsilon}}]\\
    &\qquad\leq\Eb^{1/q_x}[|\eta^\top\Sigma^{-1}X|^{q_x}]\Eb^{1/q_{\varepsilon}}[|Y-X^\top\beta|^{q_{\varepsilon}}]\leq K_xK_\varepsilon.
\end{align*} This implies that
\begin{equation*}
    {\rm BE}_n\leq C_{\rm BE}\frac{(K_xK_\varepsilon)^{q_{x\varepsilon}\wedge 3}}{\bar\sigma_\beta^{q_{x\varepsilon}\wedge 3}N^{(q_{x\varepsilon}\wedge 3-1)/2}}.
\end{equation*} Next, following the derivation of \eqref{eq:thmPM1.ebound} and \eqref{eq:thmPM1.vbound}, we get
\begin{align}
    &\big|\Eb_{\Sc_1}[\Rv_{1,n}]\big|= \big|\Jc^{s_n}\Delta^{(s_n)}[(\Sigma-\tilde\Sigma_2,\Gamma-\tilde\Gamma_2)^{\otimes s_n}]\big|,\label{eq:thmPP1.ebound}\\
    &{\rm Var}_{\Sc_1}(\Rv_{1,n})\leq2\sum_{k=1}^{s_n}{\rm Var}\bigg(\frac{\Dc^k\beta_\eta(\Sigma,\Gamma)[\bar U^{(k)}(\Sigma,\Gamma)]}{k!}\bigg)+2\sum_{k=1}^{s_n}{\rm Var}_{\Sc_1}\bigg(\frac{\tilde\Delta_k^{(s_n)}[\bar U^{(k)}(\tilde\Sigma_2,\tilde\Gamma_2)]}{k!}\bigg).\label{eq:thmPP1.vbound}
\end{align}
Next, Theorem~\ref{thm:B3} implies that for $1\leq k\leq s_n$, we have
\begin{equation*}
    {\rm Var}\left(\frac{\Dc^k\beta_\eta(\Sigma,\Gamma)[\bar U^{(k)}(\Sigma,\Gamma)]}{k!}\right)\leq4^{k-1}k^kK_x^{4k-2}K_\varepsilon^2\frac{d^{k-1}}{n^k}=:b_k.
\end{equation*} Note that for $1\leq k\leq s_n-1$,
\begin{equation*}
    \frac{b_{k+1}}{b_k}=4K_x^4\frac{d}{n}\frac{(k+1)^{k+1}}{k^k}=4K_x^4\frac{d}{n}(k+1)\bigg(1+\frac{1}{k}\bigg)^k\leq \frac{4eK_x^4s_nd}{n}.
\end{equation*} Therefore, as long as $8eK_x^4s_nd\leq n$, we have
\begin{align}\label{eq:betavarpart1}
    &\sum_{k=2}^{s_n}{\rm Var}\left(\frac{\Dc^k\beta_\eta(\Sigma,\Gamma)[\bar U^{(k)}(\Sigma,\Gamma)]}{k!}\right)\leq\sum_{k=2}^{s_n}b_k\leq b_2\sum_{k=0}^\infty\bigg(\frac{1}{2}\bigg)^k\leq 2b_2=32K_x^6K_\varepsilon^2\frac{d}{n^2}.
\end{align} Meanwhile, an application of Theorem~\ref{thm:B4} yields that for $1\leq k\leq s_n$, we have
\begin{equation*}
    {\rm Var}_{\Sc_1}\left(\frac{\tilde\Delta^{(s)}_k[\bar U^{(k)}(\Sigma,\Gamma)]}{k!}\right)\leq 4s_n^2\frac{(K_x\vee K_\varepsilon)^2}{K_x^2}(2s_n^2K_x^4)^k\frac{\Dc_{\hat\Sigma_2}^{2s_n-2k}(\Dc_{\hat\Sigma_2}\vee\Rc_{\hat\Sigma_2,\hat\Gamma_2}\vee\Dc_{\hat\Sigma_2}\Rc_{\hat\Sigma_2,\hat\Gamma_2})^2}{1\wedge\underline\lambda_{\hat\Sigma_2}^{2s_n+4}}\frac{d^{k-1}}{n^k},
\end{equation*} where $\underline\lambda_{\hat\Sigma_2}=\lambda_{\rm min}(\Sigma^{-1/2}\hat\Sigma_2\Sigma^{-1/2})$, $\Dc_{\hat\Sigma_2} = \norm{\Sigma^{-1/2}\hat\Sigma_2\Sigma^{-1/2}-I_d}_{\rm op}$, and $\Rc_{\hat\Sigma_2,\hat\Gamma_2} = \norm{\hat\Gamma_2-\hat\Sigma_2\beta}_{\Sigma^{-1}}$. Note that $\underline\lambda_{\hat\Sigma_2}\geq (1-\norm{\Sigma^{-1/2}\hat\Sigma_2\Sigma^{-1/2}-I_d}_{\rm op})_+$ and denote the event
\begin{equation*}
    \Ec_1 :=\big\{\norm{\Sigma^{-1/2}\tilde\Sigma_2\Sigma^{-1/2}-I_d}_{\rm op}\leq r_{1,n}(\delta_{1,n})\big\}\cap\big\{\norm{\tilde\Gamma_2-\tilde\Sigma_2\beta}_{\Sigma^{-1}}\leq r_{2,n}(\delta_{2,n})\big\}.
\end{equation*}Under Assumption~\ref{asmp:SigmapilotGammapilot} and that $r_{1,n}(\delta_{1,n})\leq 1/2$, we have $\Pb(\Ec_1)\geq 1-\delta_{1,n}-\delta_{2,n}$ and
\begin{align*}
    {\rm Var}_{\Sc_1}\left(\frac{\tilde\Delta^{(s)}_k[\bar U^{(k)}(\tilde\Sigma_2,\tilde\Gamma_2)]}{k!}\right)\mathbbm{1}(\Ec_1)\leq 4s_n^2\frac{(K_x\vee K_\varepsilon)^2}{K_x^2}(2s_n^2K_x^4)^k\\
    \times\frac{r_{1,n}(\delta_{1,n})^{2s_n-2k}(r_{1,n}(\delta_{1,n})\vee r_{2,n}(\delta_{2,n}))^2}{(1/2)^{2s_n+4}}\frac{d^{k-1}}{n^k},
\end{align*} Write $r_{1,n}^*=r_{1,n}(\delta_{1,n})$ and $r_{2,n}^*=r_{2,n}(\delta_{2,n})$, and we note that
\begin{align*}
    &n\sum_{k=1}^{s_n}{\rm Var}_{\Sc_1}\left(\frac{\tilde\Delta^{(s)}_k[\bar U^{(k)}(\tilde\Sigma_2,\tilde\Gamma_2)]}{k!}\right)\mathbbm{1}(\Ec_1)\\
    &\leq 2^{2s_n+7}s_n^4(K_x\vee K_\varepsilon)^2K_x^2(r_{1,n}^*\vee r_{2,n}^*)^2\sum_{k=1}^{s_n}(r_{1,n}^*)^{2s_n-2k}\bigg(\frac{2s_n^2K_x^4d}{n}\bigg)^{k-1}\\
    &\leq 2^{2s_n+7}s_n^5(K_x\vee K_\varepsilon)^2K_x^2(r_{1,n}^*\vee r_{2,n}^*)^2\bigg((r_{1,n}^*)^2\vee\frac{2s_n^2K_x^4d}{n}\bigg)^{s_n-1}.
\end{align*} Combining this with \eqref{eq:thmPP1.vbound} and \eqref{eq:betavarpart1}, we get
\begin{align*}
    &{\rm Var}_{\Sc_1}(\sqrt{n}\Rv_{1,n}/\sigma_\beta)\mathbbm{1}(\Ec_1)\leq \frac{1}{\bar\sigma_\beta^2}\bigg[32K_x^6K_\varepsilon^2\frac{d}{n}\\
    &\qquad+2^{9}(K_x\vee K_\varepsilon)^2K_x^2(r_{1,n}^*\vee r_{2,n}^*)^2s_n^5\bigg((2r_{1,n}^*)^2\vee\frac{8s_n^2K_x^4d}{n}\bigg)^{s_n-1}\bigg].
\end{align*} To control the conditional expectation of $\Rv_{1,n}$, we apply Theorem~\ref{thm:B4} which shows that $|\Eb_{\Sc_1}[\Rv_{1,n}]|\mathbbm{1}(\Ec_1)\leq 2^{s_n+2}(r_{1,n}^*)^{s_n}r_{2,n}^*,$ where we used that $r_{1,n}^*\leq 1/2$. 
Similarly, for $\Ec_2 = \big\{\norm{\Sigma^{-1/2}\tilde\Sigma_1\Sigma^{-1/2}-I_d}_{\rm op}\leq r_{1,n}(\delta_{1,n})\big\}\cap\big\{\norm{\tilde\Gamma_1-\tilde\Sigma_1\beta}_{\Sigma^{-1}}\leq r_{2,n}(\delta_{2,n})\big\}$, we have
\begin{align*}
    &|\Eb_{\Sc_2}[\Rv_{2,n}]|\mathbbm{1}(\Ec_2)\leq 2^{s_n+2}(r_{1,n}^*)^{s_n}r_{2,n}^*,\\
    &{\rm Var}_{\Sc_2}(\sqrt{n}\Rv_{2,n}/\sigma_\beta)\mathbbm{1}(\Ec_2)\leq \frac{1}{\bar\sigma_\beta^2}\bigg[32K_x^6K_\varepsilon^2\frac{d}{n}\\
    &\qquad+2^{9}(K_x\vee K_\varepsilon)^2K_x^2(r_{1,n}^*\vee r_{2,n}^*)^2s_n^5\bigg((2r_{1,n}^*)^2\vee\frac{8s_n^2K_x^4d}{n}\bigg)^{s_n-1}\bigg].
\end{align*}Combining the conditional expectation and variance bounds for $\Rv_{1,n}$ and $\Rv_{2,n}$ with Lemma~\ref{lem:BElemma} implies
\begin{align*}
    &\sup_{t\in\Real}\bigg|\Pb\bigg(\frac{\sqrt{n}(\eta^\top\hat\beta-\eta^\top\beta)}{\sigma_\beta}\leq t\bigg)-\Phi(t)\bigg|\leq C_{\rm BE}\frac{(K_xK_\varepsilon)^{q_{x\varepsilon}\wedge 3}}{\bar\sigma_\beta^{q_{x\varepsilon}\wedge 3}n^{(q_{x\varepsilon}\wedge 3-1)/2}}\\
    &+\frac{\sqrt{n}}{\bar\sigma_\beta\sqrt{\pi/8}}(2r_{1,n}^*)^{s_n}r_{2,n}^*+\frac{1}{\bar\sigma_\beta^{2/3}}\bigg[\bigg(\frac{4K_x^6K_\varepsilon^2}{\pi}\bigg)^{1/3}\bigg(\frac{d}{n}\bigg)^{1/3}\\
    &+\bigg(\frac{2^6(K_x\vee K_\varepsilon)^2K_x^2}{\pi}\bigg)^{1/3}(r_{1,n}^*\vee r_{2,n}^*)^{2/3}s_n^{5/3}\bigg((2r_{1,n}^*)^2\vee\frac{8s_n^2K_x^4d}{n}\bigg)^{\frac{s_n-1}{3}}\bigg]\Bigg)\\
    &+\delta_{1,n}+\delta_{2,n}.
\end{align*} This readily proves the first part of the theorem. Next, we let $s_n:=\lfloor\log(en)\rfloor$. Since $2r_{1,n}^*\leq 1/e$ and $r_{2,n}\leq 1$, we have
\begin{equation}\label{eq:PP1:s_n}
    s_n\geq \left\lceil\frac{\log(nr_{2,n}^*)}{\log(1/(2r_{1,n}^*))}\right\rceil\vee 1.
\end{equation} From \eqref{eq:PP1:s_n}, we immediately get
\begin{equation*}
    \sqrt{n}(2r_{1,n}^*)^{s_n}r_{2,n}^*\leq \sqrt{n}r_{2,n}^*\exp\big(-s_n\log\big(1/(2r_{1,n}^*)\big)\big)\leq \sqrt{n}r_{2,n}^*\exp\big(-\log(nr_{2,n}^*)\big)\leq  1/\sqrt{n}.
\end{equation*} Moreover, since
\begin{equation*}
    (2r_{1,n}^*)^2\vee\frac{8K_x^4d\log^2(en)}{n}\leq e^{-2},
\end{equation*} we have
\begin{align*}
    (r_{1,n}^*\vee r_{2,n}^*)^{2/3}s_n^{5/3}\bigg((2r_{1,n}^*)^2\vee\frac{8s_n^2K_x^4d}{n}\bigg)^{\frac{s_n-1}{3}} \leq s_n^{5/3}\bigg((2r_{1,n}^*)^2\vee\frac{8s_n^2K_x^4d}{n}\bigg)^{\frac{s_n-1}{3}} \\
    \leq s_n^{5/3}\exp\big(-\frac{2}{3}(s_n-1)\big)\leq \frac{\log^{5/3}(en)}{n^{2/3}}\leq\frac{11}{\sqrt{n}}.
\end{align*}where the last inequality follows from the fact that $\sup_{x\geq 1}\frac{\log^{5/3}(ex)}{x^{1/6}}=\frac{10^{5/3}}{e^{3/2}}<11$. Combining all, there exists a universal constant $\Ck>0$ such that
\begin{align*}
    &\sup_{t\in\Real}\bigg|\Pb\bigg(\frac{\sqrt{n}(\eta^\top\hat\beta-\eta^\top\beta)}{\sigma_\beta}\leq t\bigg)-\Phi(t)\bigg|\leq \Ck\Bigg(\frac{(K_xK_\varepsilon)^{q_{x\varepsilon}\wedge 3}}{\bar\sigma_\beta^{q_{x\varepsilon}\wedge 3}n^{(q_{x\varepsilon}\wedge 3-1)/2}}\\
    &+\frac{1}{\bar\sigma_\beta\sqrt{n}}+\frac{1}{\bar\sigma_\beta^{2/3}}\bigg[\big(K_x^6K_\varepsilon^2\big)^{1/3}\bigg(\frac{d}{n}\bigg)^{1/3}+\big((K_x\vee K_\varepsilon)^2K_x^2\big)^{1/3}\frac{1}{\sqrt{n}}\bigg]\Bigg)+\delta_{1,n}+\delta_{2,n}.
\end{align*} Since $K_x,K_\varepsilon\geq1$, we get the desired result.

\subsection{Proof of Corollary~\ref{cor:PP1}}
\paragraph{Under \ref{asmp:momentX}}
Define $T_i = \Sigma^{-1/2} X_i(Y_i-X_i^\top \beta)$ for $i=n+1,\ldots,2n$. Then $\Eb T_i = 0_d$ and $\Eb[T_iT_i^\top]=\Sigma^{-1/2}H\Sigma^{-1/2}$. We apply Theorem~3.1 of \cite{EinmahlLi2008LILBanach} (with $s=q>2$, $\eta=1$, and $\delta=1$) to get that for every $t>0$,
\begin{equation}\label{eq:Sn_tail_from_EL}
\Pb\Bigg(\Norm{\sum_{i=n+1}^{2n}T_i}_2\geq 2 \Eb\Norm{\sum_{i=n+1}^{2n}T_i}_2 +t \Bigg)\le
\exp\!\left(-\frac{t^2}{3\,\Lambda_n}\right)
+
C_q\,\frac{n\, \mathbb{E}\|T_1\|_2^q}{t^q}.
\end{equation}
where
\begin{equation*}
\Lambda_n := \sup\left\{\sum_{i=n+1}^{2n} \mathbb{E}\big[f(T_i)^2\big] : f\in (\mathbb{R}^d,\|\cdot\|_2)_1^*\right\},
\end{equation*}and $(\mathbb{R}^d,\|\cdot\|_2)_1^*$ denotes the unit ball of the dual space of $(\mathbb{R}^d,\|\cdot\|_2)$. Every $f\in \Xc^*$ is of the form $f_u(x)=u^\top x$ for some $u\in\mathbb{R}^d$, and $\|f_u\|_{\Xc^*}=\|u\|_2$. Thus $\Xc_1^*=\{f_u:\|u\|_2\le 1\}$ and therefore
\[
\Lambda_n
=
\sup_{\|u\|_2\le 1}\sum_{i=n+1}^{2n} \mathbb{E}\,(u^\top T_i)^2
=
\sup_{\|u\|_2\le 1} n\,u^\top \Sigma^{-1/2}H\Sigma^{-1/2} u
=
n\,\|\Sigma^{-1/2}H\Sigma^{-1/2}\|_{\rm op}.
\] Using independence and $\Eb T_i=0$, we have
\begin{equation*}
    \Eb\Norm{\sum_{i=n+1}^{2n}T_i}_2\leq \Eb^{1/2}\Norm{\sum_{i=n+1}^{2n}T_i}_2^2=n~{\rm tr}(\Sigma^{-1/2}H\Sigma^{-1/2}).
\end{equation*}From Jensen's and H\"older's inequality that for any $1\leq q\leq q_{x\varepsilon}$,
\begin{align}\label{eq:momentTbound}
    &\sup_{u\in\Sb^{d-1}}\Eb^{1/q}[|u^\top T|^q]\leq \sup_{u\in\Sb^{d-1}}\Eb^{1/q_{x\varepsilon}}\big[|u^\top\Sigma^{-1/2}X(Y-X^\top\beta)|^{q_{x\varepsilon}}\big]\nonumber\\
    &\quad \leq \sup_{u\in\Sb^{d-1}}\Eb^{1/q_{x}}\big[|u^\top\Sigma^{-1/2}X|^{q_{x}}\big]\Eb^{1/q_{\varepsilon}}\big[|Y-X^\top\beta|^{q_{\varepsilon}}\big]\leq K_xK_\varepsilon.
\end{align} The last inequality follows from Assumption~\ref{asmp:momentX} and \ref{asmp:momente}. Therefore, $\norm{\Sigma^{-1/2}H\Sigma^{-1/2}}_{\rm op}\leq (K_xK_\varepsilon)^2$ and ${\rm tr}(\Sigma^{-1/2}H\Sigma^{-1/2})\leq (K_xK_\varepsilon)^2d$. Also, from Jensen's inequality and \eqref{eq:momentTbound}, we get
\begin{equation*}
    \Eb\norm{T}_2^q =d^{q/2}\Eb\bigg[\bigg(\frac{1}{d}\sum_{j=1}^d(\ev_j^\top T)^2\bigg)^{q/2}\bigg]\leq d^{q/2}\cdot\frac{1}{d}\sum_{j=1}^d\Eb|\ev_j^\top T|^q\leq (K_x^2K_\varepsilon^2~d)^{q/2}.
\end{equation*} Set $t=nx$ in \eqref{eq:Sn_tail_from_EL} to deduce
\begin{equation*}
    \Pb\bigg(\left\|\frac1n\sum_{i=n+1}^{2n} T_i\right\|_2\geq 2K_xK_\varepsilon\sqrt{\frac{d}{n}}+x\bigg)\leq \exp\!\left(-\frac{n x^2}{3(K_xK_\varepsilon)^2}\right)
+
C_q\,\frac{(K_x^2K_\varepsilon^2~d)^{q/2}}{n^{q-1}x^q}.
\end{equation*} Note that $\norm{n^{-1}\sum_{i=n+1}^{2n}T_i}_2= \norm{\hat\Gamma_2-\hat\Sigma_2\beta}_{\Sigma^{-1}}$, and taking
\begin{equation*}
    x=\max\Set{K_xK_\varepsilon\sqrt{\frac{3 \log(2/\delta)}{n}},\frac{C_q^{1/q}K_xK_\varepsilon \sqrt{d}}{n^{1-1/q}\delta^{1/q}}},
\end{equation*} implies that
\begin{equation*}
\mathbb{P}\!\left( \norm{\hat\Gamma_2-\hat\Sigma_2\beta}_{\Sigma^{-1}}
\ge K_xK_\varepsilon\bigg[\sqrt{\frac{4d+3\log(2/\delta)}{n}}+C_q\frac{\sqrt{d}}{\delta^{1/q}n^{1-1/q}}\bigg]
\right)\le \delta,
\end{equation*} for any $2< q\leq q_{x\varepsilon}.$ We choose $q = q_{x\varepsilon}$ and $\delta = n^{1-q_{x\varepsilon}/2}$ to get for some $\Ck = \Ck(q_{x\varepsilon})>0$,
\begin{equation}\label{cond:samgamma}
    \mathbb{P}\!\left( \norm{\hat\Gamma_2-\hat\Sigma_2\beta}_{\Sigma^{-1}}
\ge \Ck K_xK_\varepsilon\sqrt{\frac{d+\log(e n)}{n}}
\right)\le \frac{1}{n^{q_{x\varepsilon}/2-1}}.
\end{equation} Note that \eqref{cond:samgamma} and \eqref{cond:samcov} imply that the method of moment estimators $(\hat\Sigma_2,\hat\Gamma_2)$ satisfies Assumption~\ref{asmp:SigmapilotGammapilot} with
\begin{align*}
    &r_{1,n}^* = \Ck\bigg[\sqrt{\frac{d+\log(en)}{n}}+\bigg(\frac{d}{n^{1-2/q_x}}\bigg)^{\frac{q_x}{q_x+2}}\bigg],\quad\delta_{1,n}=\bigg(\frac{d}{n^{1-2/q_x}}\bigg)^{\frac{q_x}{q_x+2}},\\
    &r_{2,n}^* = \Ck\sqrt{\frac{d+\log(en)}{n}},\quad\mbox{and}\quad \delta_{2,n}=\frac{1}{n^{q_{x\varepsilon}/2-1}}.
\end{align*} for some constant $\Ck=\Ck(K_x,K_\varepsilon,q_x,q_\varepsilon)>0$. Here, we may assume $(2r_{1,n}^*)^2\vee (8K_x^4d\log^2(en))/n\leq e^{-2}$ and $r_{2,n}^*\leq 1$, otherwise the conclusion holds by simply taking a sufficiently large constant. Hence, the second part of Theorem~\ref{thm:PP1} applies and yields the results.

\paragraph{Under~\ref{asmp:sg}} For $0<\epsilon<q_\varepsilon-2$, let $q^\dagger=q^\dagger(\epsilon\wedge 1,q_\varepsilon)>0$ be such that
\begin{equation*}
    \bigg(\frac{1}{q^\dagger}+\frac{1}{q_\varepsilon}\bigg)^{-1}=2+(\epsilon\wedge 1).
\end{equation*} Under~\ref{asmp:sg}, \ref{asmp:momentX} holds with $q_x = q^\dagger$ and $K_x = K_{\rm sg}\sqrt{2eq^\dagger}.$ Combining \eqref{cond:samcov_sg} and \eqref{cond:samgamma} imply that $(\hat\Sigma_2,\hat\Gamma_2)$ satisfies Assumption~\ref{asmp:SigmapilotGammapilot} with 
\begin{align*}
    &r_{1,n}^* = cK_{\rm sg}^2\sqrt{\frac{d+\log(en)}{n}},\quad\delta_{1,n}=\frac{1}{n},\\
    &r_{2,n}^* = \Ck~K_{\rm sg}K_\varepsilon\sqrt{\frac{d+\log(en)}{n}},\quad\mbox{and}\quad \delta_{2,n}=\frac{1}{n^{(\epsilon\wedge 1)/2}},
\end{align*} for some constant $\Ck = \Ck(\epsilon\wedge1,q_\varepsilon)>0$. Again, we may assume $(2r_{1,n}^*)^2\leq e^{-2}$ and $r_{2,n}^*\leq 1$, otherwise the conclusion holds by simply taking a sufficiently large constant. Hence, the second part of Theorem~\ref{thm:PP1} applies and yields the results.

\subsection{Proof of Corollary~\ref{cor:PP2}}
We set $\tilde\Sigma_2 = \widehat{\mathsf{E}}_{1}$ and $\tilde\Gamma_2 = \widehat{\mathsf{E}}_{1}\hat\beta_2$. Recall the concentration of the robust Gram matrix $\widehat{\mathsf{E}}_{1}$ from \eqref{cond:robcov}. We note the inequality $\norm{\tilde\Gamma_2-\tilde\Sigma_2\beta}_{\Sigma^{-1}}\leq \norm{\Sigma^{-1/2}\widehat{\mathsf{E}}_1\Sigma^{-1/2}}_{\rm op}\norm{\hat\beta_2-\beta}_{\Sigma}$. Furthermore, since $\hat\beta_2 = \hat\Sigma_2^{-1}\hat\Gamma_2$, we have
\begin{align}\label{eq:cor4.6:1}
    \norm{\hat\beta_2-\beta}_{\Sigma}^2&=(\hat\beta_2-\beta)^\top\Sigma (\hat\beta_2-\beta)=(\hat\Gamma_2-\hat\Sigma_2\beta)^\top \hat\Sigma_2^{-1}\Sigma\hat\Sigma_2^{-1}(\hat\Gamma_2-\hat\Sigma_2\beta)\nonumber\\
    &\quad\leq \lambda_{\rm min}^{-2}(\Sigma^{-1/2}\hat\Sigma_2\Sigma^{-1/2})\norm{\hat\Gamma_2-\hat\Sigma_2\beta}_{\Sigma^{-1}}^2.
\end{align} Theorem~4.1 of \cite{Oliveira2016LowerTailQuadraticForms} applies and yields that for all $\delta>0$,
\begin{equation}\label{eq:oliveiralower}
    \Pb\bigg(\lambda_{\rm min}(\Sigma^{-1/2}\hat\Sigma_2\Sigma^{-1/2})\geq 1-9K_x\sqrt{\frac{d+2\log(2/\delta)}{n}}\bigg)\geq 1-\delta.
\end{equation} We take $\delta=1/n$ in \eqref{eq:oliveiralower}, and assume $10K_x\sqrt{\frac{d+2\log(2n)}{n}}\leq 1$ (otherwise, we simply take a large constant in the statement of Corollary), to get $\Pb(\lambda_{\rm min}(\Sigma^{-1/2}\hat\Sigma_2\Sigma^{-1/2})\geq 1/2)\leq 1/n$. Combining this with \eqref{cond:robcov}, \eqref{cond:samgamma}, and \eqref{eq:cor4.6:1} implies that with probability at least $1-2/n-1/n^{q_{x\varepsilon}/2-1}$,
\begin{align*}
    \norm{\tilde\Gamma_2-\tilde\Sigma_2\beta}_{\Sigma^{-1}}&\leq \bigg(1+cK_x^2\sqrt{\frac{d+\log(en)}{n}}\bigg)\frac{1}{10}\Ck K_xK_\varepsilon\sqrt{\frac{d+\log(en)}{n}}\\
    &\leq \Ck K_xK_\varepsilon\sqrt{\frac{d+\log(en)}{n}}.
\end{align*} for a universal $c>0$ and $\Ck = \Ck(q_{x\varepsilon})$, potentially different by lines. For the last inequality, we assume $cK_x^2\sqrt{\frac{d+\log(en)}{n}}\leq 1$ by the routine logic. Therefore, $(\tilde\Sigma_2,\tilde\Gamma_2)$ satisfies Assumption~\ref{asmp:SigmapilotGammapilot} with
\begin{align*}
    &r_{1,n}^* = cK_x^2\sqrt{\frac{d+\log(en)}{n}},\quad\delta_{1,n}=\frac{1}{n},\\
    &r_{2,n}^* = \Ck K_xK_\varepsilon\sqrt{\frac{d+\log(en)}{n}},\quad\mbox{and}\quad \delta_{2,n}=\frac{2}{n}+\frac{1}{n^{q_{x\varepsilon}/2-1}}.
\end{align*} for some constant $\Ck=\Ck(q_{x\varepsilon})>0$. Routinely, we may suppose $(2r_{1,n}^*)^2\leq e^{-2}$ and $r_{2,n}^*\leq 1$, otherwise we simply take a sufficiently large constant. Now, the second part of Theorem~\ref{thm:PP1} applies and yields the results.

\subsection{Proof of Theorem~\ref{thm:normality-msmooth-MC}}Theorem~\ref{thm:normality-msmooth-MC} follows from an application of Lemma~\ref{lem:BElemma} with Theorem~\ref{thm:normality-msmooth} and conditional moment bound below.

\begin{lemma}\label{lem:PRE}
    For $1\leq k\leq s$,
    \begin{align*}
        &\Eb[F_k(\pi)\mid \Sc_1,\Sc_2] = \frac{1}{k!}\Dc^k(\hat\theta_{\Sc_2})[\bar U^{(k)}(\hat\theta_{\Sc_2})],\\
        &{\rm Var}\big(F_k(\pi)\mid \pi,\Sc_2\big)\leq \norm{T_k(\hat\theta_{\Sc_2})}_{\rm S}^2\sum_{r=1}^k\binom{n}{r}\frac{\binom{n-r}{k-r}^2}{\binom{n}{k}^2}\norm{h}^{2(k-r)}\nu\,V^{r-1}.
    \end{align*}
\end{lemma}
\begin{proof} The expectation identity follows immediately. For a fixed $\pi\in\Sk_n$, relabel the sequence as $Z_t = W_{\pi(t)} -\theta$, $t\in[n]$ and write $h=\theta-\hat\theta_{\Sc_2}$. By multi-linearity, we get
    \begin{equation*}
        T_k(\hat\theta_{\Sc_2})[Z_{i_1}+h,\ldots,Z_{i_k}+h]=\sum_{S\subset [k]}T_k(\hat\theta_{\Sc_2})[y_1^{(S)},\ldots,y_k^{(S)}],\quad y_j^{(S)}=\begin{cases}
            Z_{i_{j}},&j\in S,\\
            h,&j\notin S;
        \end{cases}.
    \end{equation*} For each $A\subset[k]$, define
    \begin{equation}\label{eq:F_KA}
        F_{k,A}(\pi)=\binom{n}{k}^{-1}\sum_{1\le i_1<\cdots<i_k\le n}T_k(\hat\theta_{\Sc_2})[y_1^{(S)},\ldots,y_k^{(S)}],
    \end{equation} then $F_k(\pi)=\sum_{A\subseteq[k]}F_{k,A}(\pi)$. Now, fix $A = \{1\leq a_1<\ldots<a_r\leq k\}$ with $1\leq r\leq k$, and define the asymmetric order $r$ kernel as
    \begin{equation*}
        \psi_A(x_1,\ldots,x_r)=T_k(\hat\theta_{\Sc_2})[w_1,\ldots,w_k],\quad w_{a_m}=x_m,~(1\leq m\leq r),~w_j=h,~(j\notin A).
    \end{equation*} This kernel is $r$-linear and (conditional on $\Sc_2$) degenerate of order $r$. Moreover, we have $\norm{\psi_A}_{\rm S}\leq \norm{h}^{k-r}\norm{T_k(\hat\theta_{\Sc_2})}_{\rm S}$. For $J\subset[n]$ with $|J|=r$, let $c_{J,A}$ be the number of $k$ subsets $\{i_1<\ldots<i_k\}$ such that $i_{a_m}=j_m$ for $1\leq m\leq r$. Now, we regroup the sum over all $k$-subsets in \eqref{eq:F_KA} by their induced $J$, yielding the exact identity that for $|A|=r$,
    \begin{equation*}
        F_{k,A} = \binom{n}{k}^{-1}\sum_{J\subset \Ic_{n,r}} c_{J,A}\psi_A(Z_J).
    \end{equation*} Here, $\Ic_{n,r}$ denotes a collection of $r$ subset of $[n]$. Let $A,A'\subset[k]$ with $|A|=r$ and $|A'|=r'$. If $r\neq r'$, then ${\rm Cov}(F_{k,A}(\pi),F_{k,A'}(\pi)\mid\Sc_2)=0$ because every summand of $F_{k,A}$ is a multilinear form of degree $r$ in $(Z_1,\ldots,Z_n)$, and every summand of $F_{k, A'}$ has degree $r'$; in the product, some $Z_j$ must appear exactly once unless $r=r'$ and the same index set is used. This also implies that for $|A|=|A'|=r$ but $J\neq J'$, ${\rm Cov}(\psi_A(Z_J),\psi_{A'}(Z_{J'})\mid\Sc_2)=0.$ Therefore,
    \begin{align*}
        {\rm Var}\big(F_k(\pi)\mid \pi,\Sc_2\big)=\sum_{r=1}^k\sum_{\substack{A,A'\subset[k]\\|A|=|A'|=r}}\binom{n}{k}^{-2}\sum_{J\in \Ic_{n,r}} c_{J,A}c_{J,A'}\Eb[\psi_A(Z_J)\psi_{A'}(Z_J)\mid\Sc_2]\\
        \leq \sum_{r=1}^k\sum_{J\in\Ic_{n,r}}\Bigg(\binom{n}{k}^{-1}\sum_{\substack{A\subset[k]\\|A|=r}}c_{J,A}\Eb^{1/2}[\psi_A(Z_J)^2\mid\Sc_2]\Bigg)^2.
    \end{align*} The last ineqaulity is from Cauchy Scahwrz inequality. Apply Lemma~\ref{lem:varbound} with $p=2$ to the independent $r$ inputs $Z_J$ to get
    \begin{equation*}
        \Eb^{1/2}[\psi_A(Z_J)^2\mid\Sc_2]\leq \norm{T_k(\hat\theta_{\Sc_2})}_{\rm S}\norm{h}^{k-|A|}\sqrt{\nu} (\sqrt{V})^{r-1}.
    \end{equation*}Furthermore, we note that the sum $\sum_{A \subset [k], |A|=|J|} c_{J,A}=\binom{n-|J|}{k-|J|}$ because it counts every $k$-subset of $[n]$ that contains the fixed $|J|$-subset $J$ exactly once, which is equivalent to choosing the remaining $k-|J|$ elements from the available values in $[n] \setminus J$. Therefore,
    \begin{align*}
        {\rm Var}\big(F_k(\pi)\mid \pi,\Sc_2\big)\leq \norm{T_k(\hat\theta_{\Sc_2})}_{\rm S}^2\sum_{r=1}^k\binom{n}{r}\frac{\binom{n-r}{k-r}^2}{\binom{n}{k}^2}\norm{h}^{2(k-r)}\nu\,V^{r-1}.
    \end{align*}
\end{proof}

Denote
\begin{align*}
    \Av_n &= \frac{\sqrt{N}(\widehat{f}_s^{(b)}-f(\theta))}{\sigma_f},~&&\Bv_n = \frac{\sqrt{N}(\widehat{f}_s-f(\theta))}{\sigma_f},\\
    \Cv_n &= \frac{\sqrt{N/4}(\widehat{f}_s^{(b)}(\Sc_1,\Sc_2)-\widehat{f}_s(\Sc_1,\Sc_2))}{\sigma_f},~&&\Dv_n = \frac{\sqrt{N/4}(\widehat{f}_s^{(b)}(\Sc_2,\Sc_1)-\widehat{f}_s(\Sc_2,\Sc_1))}{\sigma_f},
\end{align*} In order to apply Lemma~\ref{lem:BElemma}, we control the conditional expectation and variance of $\Cv_n$ and $\Dv_n$. Since $\frac{1}{k!}\Dc^k(\hat\theta_{\Sc_2})[\bar U^{(k)}(\hat\theta_{\Sc_2})]=\Eb[F_k(\pi)|\Sc_1,\Sc_2]$, we write
\begin{align*}
    \sigma_f\,\Cv_n = \sqrt{n/2}\sum_{k=2}^s\frac{1}{b}\sum_{r=1}^b\Delta_k(\pi_{kr}),\quad \Delta_k(\pi_{kr})=F_k(\pi_{kr})-\Eb[F_k(\pi)|\Sc_1,\Sc_2].
\end{align*} In order to apply Lemma~\ref{lem:BElemma}, we set $\Fc_n = \sigma(\Sc_2)$ and $\Ec_n = \set{\norm{\hat\theta_{\Sc_2}-\theta}\leq r_n}$. From Lemma~\ref{lem:PRE}, we have $\Eb[\Delta_k(\pi_{kr})\mid\Sc_1,\Sc_2]=0$, and thus $\Eb[\Delta_k(\pi_{kr})\mid\Sc_2]=0$, i.e. $\Eb[\Cv_n\mid\Fc_n]=0$. Moreover, from the (conditional) independence, we get
\begin{equation*}
    {\rm Var}\bigg(\sum_{k=2}^s\frac{1}{b}\sum_{r=1}^b\Delta_k(\pi_{kr})\mid\Sc_2\bigg) = \frac{1}{b}\sum_{k=2}^s{\rm Var}\big(\Delta_k(\pi)\mid\Sc_2\big).
\end{equation*} From the tower law, we have 
\begin{align*}
    {\rm Var}\big(\Delta_k(\pi)\mid\Sc_2\big) = \Eb[{\rm Var}\big(\Delta_k(\pi)\mid\pi,\Sc_2\big)\mid\Sc_2]+{\rm Var}\big(\Eb[\Delta_k(\pi)\mid\pi,\Sc_2]\mid\Sc_2\big).
\end{align*} We note from the definition~\eqref{eq:Fk_pi_general} that $\Eb[F_k(\pi)\mid\pi,\Sc_2] = T_k(\hat\theta_{\Sc_2})[(\theta-\hat\theta_{\Sc_2})^{\otimes k}]$ is independent of $\pi$, that is,
\begin{align*}
    \Eb[\Delta_k(\pi)\mid\pi,\Sc_2] &= \Eb[F_k(\pi)\mid\pi,\Sc_2] - \Eb[\Eb[F_k(\pi)\mid\Sc_1,\Sc_2]\mid\pi,\Sc_2]\\
    &=\Eb[F_k(\pi)\mid\Sc_2] - \Eb[F_k(\pi)\mid\Sc_2] = 0.
\end{align*} We next analyze the variance bound in Lemma~\ref{lem:PRE}. Since
\begin{equation*}
    \frac{\binom{n-r}{k-r}}{\binom{n}{k}}=\frac{k(k-1)\cdots(k-r+1)}{n(n-1)\cdots(n-r+1)}\leq \bigg(\frac{k}{n}\bigg)^r,
\end{equation*} we have
\begin{equation*}
    {\rm Var}\big(F_k(\pi)\mid \pi,\Sc_2\big)\leq\norm{T_k(\hat\theta_{\Sc_2})}_{\rm S}^2\sum_{r=1}^k\frac{n^r}{r!}\bigg(\frac{k}{n}\bigg)^{2r}\norm{h}^{2(k-r)}\nu\,V^{r-1}.
\end{equation*} Note that the right-hand side is independent of $\pi$. Therefore,
\begin{align*}
    {\rm Var}\big(F_k(\pi)\mid \Sc_2\big)\mathbbm{1}(\Ec_n)\leq \sup_{x\in U}\norm{T_k(x)}_{\rm S}^2\sum_{r=1}^k\frac{1}{r!}\bigg(\frac{k^2}{n}\bigg)^{r}\norm{h}^{2(k-r)}\nu\,V^{r-1}\\
    =\sup_{x\in U}\norm{T_k(x)}_{\rm S}^2\frac{\nu k^2}{n}\sum_{r=1}^k\frac{1}{r!}\bigg(\frac{k^2V}{n}\bigg)^{r-1}r_n^{2(k-r)}\\
    \leq\sup_{x\in U}\norm{T_k(x)}_{\rm S}^2\frac{\nu k^2}{n}\bigg(\frac{k^2V}{n}\vee r_n^2\bigg)^{k-1}\sum_{r=1}^k\frac{1}{r!}\\
    \leq \sup_{x\in U}\norm{T_k(x)}_{\rm S}^2\frac{(e-1)\nu k^2}{n}\bigg(\frac{k^2V}{n}\vee r_n^2\bigg)^{k-1},
\end{align*} where the last inequality follows from $\sum_{r\geq 1}1/r!=e-1.$ Consequently,
\begin{align*}
    b{\rm Var}(\sigma_f\Cv_n\mid\Sc_2)\mathbbm{1}(\Ec_n)\leq (e-1)\sup_{x\in U}\norm{T_k(x)}_{\rm S}^2\nu \sum_{k=2}^sk^2\bigg(\frac{k^2V}{n}\vee r_n^2\bigg)^{k-1}\lesssim_{\bar \nu, s}\bigg(\frac{V}{n}\vee r_n^2\bigg)\vee \bigg(\frac{V}{n}\vee r_n^2\bigg)^{s-1}.
\end{align*} Here, we may assume $\frac{V}{n}\vee r_n^2\leq 1$, otherwise we take a sufficiently large constant in the statement of Theorem~\ref{thm:normality-msmooth-MC}. Therefore,
\begin{equation*}
    {\rm Var}(\sigma_f\Cv_n)\lesssim_{\bar \nu, s}\frac{1}{b\sigma_f^2}\bigg(\frac{V}{n}\vee r_n^2\bigg).
\end{equation*} Similar bound applies to $\Dv_n$ with $\Fc'_n = \sigma(\Sc_1)$ and $\Ec_n' = \set{\norm{\hat\theta_{\Sc_1}-\theta}\leq r_n}$. Now an application of Lemma~\ref{lem:BElemma} with the Berry--Esséen  bound for $\Bv_n$ in Theorem~\ref{thm:normality-msmooth} concludes the proof.
    
\subsection{Proof of Theorem~\ref{thm:normality-gevrey-MC}}
We endorse the proof of Theorem~\ref{thm:normality-msmooth-MC}. In particular, with $\sup_{x\in U}\norm{T_k(x)}_{\rm S}\leq\norm{T_{\geq 0}}_{L_\infty(U),R} R^k$, we have
\begin{align*}
    {\rm Var}\big(F_k(\pi)\mid \Sc_2\big)\mathbbm{1}(\Ec_n)\leq \norm{T_{\geq 0}}_{L_\infty(U),R}^2\frac{(e-1)\nu R^{2k}k^2}{n}\bigg(\frac{k^2V}{n}\vee r_n^2\bigg)^{k-1}.
\end{align*} Consequently,
\begin{align*}
    b{\rm Var}(\sigma_f\Cv_n)\mathbbm{1}(\Ec_n)\leq (e-1)\norm{T_{\geq 0}}_{L_\infty(U),R}^2\nu R^2\sum_{k=2}^sk^2\bigg(\frac{R^2k^2V}{n}\vee R^2r_n^2\bigg)^{k-1}.
\end{align*} Let $c_k = k^2(\frac{R^2k^2V}{n}\vee R^2r_n^2)^{k-1}$ for $2\leq k\leq s$, then for $2\leq k\leq s-1$,
\begin{equation*}
    \frac{c_{k+1}}{c_k}\leq \frac{(k+1)^2}{k^2}\bigg(\frac{R^2(k+1)^2V}{n}\vee R^2r_n^2\bigg)\leq\frac{9}{4}\bigg(\frac{R^2s^2V}{n}\vee R^2r_n^2\bigg)\leq\frac{9}{4e^2}<1.
\end{equation*} This implies that
\begin{align*}
    &b{\rm Var}(\sigma_f\Cv_n)\mathbbm{1}(\Ec_n)\leq (e-1)\norm{T_{\geq 0}}_{L_\infty(U),R}^2\nu R^2\sum_{k=2}^s c_k\\
    &\leq \frac{4(e-1)}{1-(9/4e^2)}\norm{T_{\geq 0}}_{L_\infty(U),R}^2\nu\bigg(\frac{4R^4V}{n}\vee R^4r_n^2\bigg)
    \leq 10\norm{T_{\geq 0}}_{L_\infty(U),R}^2\nu\bigg(\frac{4R^4V}{n}\vee R^4r_n^2\bigg).
\end{align*} That is,
\begin{equation*}
    {\rm Var}(\sigma_f\Cv_n)\lesssim\frac{\norm{T_{\geq 0}}_{L_\infty(U),R}^2\nu}{b\sigma_f^2}\bigg(\frac{R^4V}{n}\vee R^4r_n^2\bigg),
\end{equation*} and analogously for $\Dv_n$ with $\Fc'_n = \sigma(\Sc_1)$ and $\Ec_n' = \set{\norm{\hat\theta_{\Sc_1}-\theta}\leq r_n}$. Since $f\in\Gc^1(U)$ and $\norm{f}_{\Gc^1(U),R}\leq \norm{T_{\geq 0}}_{L_\infty(U),R}<\infty$, we have from Theorem~\ref{thm:normality-gevrey} that
\begin{align*}
    &\sup_{t\in\Real}\abs{\Pb_P(\Bv_n\leq t)-\Phi(t)}\leq C\bigg[ \frac{\mu_{3,f}}{\sigma_f^3\sqrt{n}}+\frac{\norm{T_{\geq 0}}_{L_\infty(U),R}Rr_n}{\sigma_f\sqrt{n}}\\
    &\quad +\bigg(\frac{\norm{T_{\geq 0}}_{L_\infty(U),R}R^{2}}{\sigma_f}\bigg)^{2/3}\bigg\{\bigg(\frac{\nu(P) V(P)}{n}\bigg)^{1/3} + \frac{\nu(P)^{1/3}}{\sqrt{n}}\bigg\}\bigg]+2\delta_{n,P}.
\end{align*} Now, an application of Lemma~\ref{lem:BElemma} concludes the proof.

\begin{proposition}\label{prop:recursion}
    Define $Y_j^{(t)}$ for $t=0,1,\dots,n$ and $j=0,1,\dots,k$ by the recursion:
\begin{equation}\label{eq:dp_recursion}
Y_0^{(t)} := v_0,\qquad
Y_j^{(0)} := 0\ (j\ge 1),\qquad
Y_j^{(t)} := Y_j^{(t-1)} + L_t\big(Y_{j-1}^{(t-1)}\big)\ (t\ge 1, j\ge 1).
\end{equation} Then, for all $0\leq j\leq k$,
\begin{equation}\label{eq:inductionhyp}
        Y_j^{(t)}=\sum_{1\le t_1<\cdots<t_j\le t} L_{t_j}\circ \cdots \circ L_{t_1}\, v_0,\quad\forall~ 0\leq j\leq k,
    \end{equation}
\end{proposition}
\begin{proof}
    We shall prove \eqref{eq:inductionhyp} holds for all $0\leq t\leq n$ by induction on $t$. The case $t=0$ follows from the initialization. For $t\geq 1$ and $j\geq 1$,
    \begin{align*}
        &\sum_{1\le t_1<\cdots<t_j\le t} L_{t_j}\circ \cdots \circ L_{t_1}\, v_0 = \sum_{1\le t_1<\cdots<t_j\le t-1} L_{t_j}\circ \cdots \circ L_{t_1}\, v_0 + \sum_{1\le t_1<\cdots<t_j= t} L_{t_j}\circ \cdots \circ L_{t_1}\, v_0\\
        &\qquad =Y_j^{(t-1)} + L_t\circ\sum_{1\le t_1<\cdots<t_{j-1}\leq t-1} L_{t_{j-1}}\circ \cdots \circ L_{t_1}\, v_0=Y_j^{(t-1)} + L_t\big(Y_{j-1}^{(t-1)}\big)=Y_j^{(t)}.
    \end{align*} Finally, $j=0$ case follows from the initialization.
\end{proof}
For a fixed $\pi \in \Sk_n$, $x\in\Bb$, and Algorithm~\ref{alg:ro_dp_general} defines the sequence $Y_j^{(t)}$ for $t = 0, \dots, n$ and $j = 0, \dots, k$ defined through: $Y_j^{(0)} = 0$ for $j \ge 1$, $Y_0^{(t)} := Y_0^{(t-1)} G_{k,t}(x)$ and $Y_j^{(t)} := Y_j^{(t-1)} G_{k,t}(x) + Y_{j-1}^{(t-1)} H_{\pi(t)} G_{k,t}(x)$ for $t\geq 1,j\geq 1$. Therefore, the terminal state satisfies:
\begin{equation*}
\Lambda_k(x)(Y_k^{(n)}) = \sum_{1\le i_1<\cdots<i_k\le n} \Lambda_{k}(x)\!\big(
G_{k,0}(x)\,H_{\pi(1)}\,G_{k,1}(x)\cdots H_{\pi(k)}\,G_{k,k}(x)\big)
\end{equation*}

\section{Auxiliary Results for Section~\ref{sec:appl}}
\begin{theorem}\label{thm:B1}
    Consider a functional $\omega(\Sigma) = \eta_1^\top \Sigma^{-1}\eta_2$ and recall $k$:th Fréchet derivative $\Dc^k\omega$ of $\omega$ in \eqref{eq:precisionmatrix_Fréchetderivatives} for $1\leq k\leq s$. Suppose that Assumption~\ref{asmp:momentX} holds. Let
    \begin{equation*}
        \bar h^{(k)}=\sum_{\substack{1\leq j_1,\ldots,j_k\leq n\\\mbox{\tiny distinct}}}\frac{h_{j_1}\otimes\cdots\otimes h_{j_k}}{n(n-1)\cdots (n-k+1)},
    \end{equation*}
    with $h_i = X_iX_i^\top-\Sigma$. Then $\Eb[\Dc^k \omega(\Sigma)[\bar h^{(k)}]]=0$ for $1\leq k\leq s$ and
    \begin{equation}\label{eq:B1:var}
        {\rm Var}\!\left(\Dc^k \omega(\Sigma)[\bar h^{(k)}]\right)\leq \norm{\eta_1}_{\Sigma^{-1}}^2\norm{\eta_2}_{\Sigma^{-1}}^2 (k!)^3 2^k K_x^{4k}\frac{d^{k-1}}{(n)_k},\quad 1\leq k\leq s.
    \end{equation}
\end{theorem}

\begin{proof}
    Since $\Eb[h_i]=0$ and $\Dc^k\omega(\Sigma)$ is $k$-linear, $\Eb[\Dc^k \omega(\Sigma)[\bar h^{(k)}]]=0$ follows immediately.

    Since $\Dc^k\omega(\Sigma)$ is symmetric,
    \begin{equation*}\label{eq:B1:ustat}
        \Dc^k \omega(\Sigma)[\bar h^{(k)}]
        = \binom{n}{k}^{-1}\sum_{1\leq j_1<\ldots<j_k\leq n}\varphi_\omega(h_{j_1},\ldots,h_{j_k}),
    \end{equation*}
    where $\varphi_\omega(h_1,\ldots,h_k)=\ip{\Dc^k \omega(\Sigma), h_1,\ldots, h_k}.$ By multilinearity and $\Eb[h_i]=0$, $\varphi_\omega$ is canonical, i.e., $\Eb\big[\varphi_\omega(h_1,\ldots,h_k)\mid h_1,\ldots,h_{k-1}\big]=0.$ Therefore, the variance formula for canonical $U$-statistics gives
    \begin{equation}\label{eq:B1:var-reduce}
        {\rm Var}\!\left(\Dc^k \omega(\Sigma)[\bar h^{(k)}]\right)
        =\binom{n}{k}^{-1}{\rm Var}\!\left(\varphi_\omega(h_1,\ldots,h_k)\right)
        \le \binom{n}{k}^{-1}\Eb\!\left[\varphi_\omega(h_1,\ldots,h_k)^2\right].
    \end{equation}

    Recall from \eqref{eq:precisionmatrix_Fréchetderivatives} that, writing $\tilde\eta_1=\Sigma^{-1/2}\eta_1$ and $\tilde\eta_2=\Sigma^{-1/2}\eta_2$,
    \begin{equation*}\label{eq:B1:kernel-expand}
        \varphi_\omega(h_1,\ldots,h_k)
        =(-1)^k\sum_{\pi\in\mathfrak S_k}\tilde\eta_1^\top A_{\pi(1)}\cdots A_{\pi(k)}\tilde\eta_2,
        \qquad
        A_i=Z_iZ_i^\top-I_d,
    \end{equation*} where $Z_i = \Sigma^{-1/2}X_i$ for $i\in[n]$ and $\mathfrak S_k$ denotes a collection of permutations of $[k]$. Expanding each centered factor, 
    \begin{equation*}\label{eq:B1:center-expand}
        \prod_{r=1}^k (Z_{\pi'(r)}Z_{\pi'(r)}^\top-I_d)
        =\sum_{J\subset[k]}(-1)^{k-|J|}\prod_{r\in J}{}^\uparrow Z_{\pi'(r)}Z_{\pi'(r)}^\top,
    \end{equation*} where $\prod{}^\uparrow$ means multiplying in increasing order. Using $\Eb[Z_iZ_i^\top]=I_d$, we obtain
    \begin{align}\label{eq:B1:second-moment}
        &\Eb\!\left[\varphi_\omega(h_1,\ldots,h_k)^2\right]
        \nonumber\\
        &\le\sum_{\pi,\pi'\in\mathfrak S_k}\Big|\Eb\big[\tilde\eta_1^\top \prod_{r=1}^k (Z_{\pi(r)}Z_{\pi(r)}^\top-I_d)\tilde\eta_2\times\tilde\eta_1^\top \prod_{r=1}^k (Z_{\pi'(r)}Z_{\pi'(r)}^\top-I_d)\tilde\eta_2\big]\Big|\nonumber\\
        &=\sum_{\pi,\pi'\in\mathfrak S_k}\Big|\Eb\big[\tilde\eta_1^\top \prod_{r=1}^k (Z_{\pi(r)}Z_{\pi(r)}^\top)\tilde\eta_2\times\tilde\eta_1^\top \prod_{r=1}^k (Z_{\pi'(r)}Z_{\pi'(r)}^\top-I_d)\tilde\eta_2\big]\Big|\nonumber\\
        &\leq  \sum_{\pi,\pi'\in\mathfrak S_k}\sum_{J\subset[k]}\Eb\big[\tilde\eta_1^\top \prod_{r=1}^k (Z_{\pi(r)}Z_{\pi(r)}^\top)\tilde\eta_2\times\tilde\eta_1^\top \prod_{r\in J} {}^\uparrow(Z_{\pi'(r)}Z_{\pi'(r)}^\top)\tilde\eta_2\big]\Big|
    \end{align}
    Applying Theorem~\ref{thm:C2} with $\sigma=1$ and $\tau=d$ yields
    \begin{equation*}
        \Big|\Eb[\mbox{Summand in \eqref{eq:B1:second-moment}}]\Big|
        \le \|\tilde\eta_1\|_2^2\|\tilde\eta_2\|_2^2\,K_x^{2k+2|J|}\,d^{k-1},
    \end{equation*}
    and hence
    \begin{align}\label{eq:B1:second-moment-final}
        \Eb\!\left[\varphi_\omega(h_1,\ldots,h_k)^2\right]
        \le (k!)^2 \,\|\tilde\eta_1\|_2^2\|\tilde\eta_2\|_2^2\,d^{k-1}\sum_{J\subset[k]}K_x^{2k+2|J|}\nonumber\\
        \leq (k!)^2 \,\|\tilde\eta_1\|_2^2\|\tilde\eta_2\|_2^2\,K_x^{2k}(1+K_x^2)^{2k}\,d^{k-1}
    \end{align}
    Combining \eqref{eq:B1:var-reduce} and \eqref{eq:B1:second-moment-final}, we obtain
    \begin{equation*}
        {\rm Var}\!\left(\Dc^k \omega(\Sigma)[\bar h^{(k)}]\right)
        \le \|\eta_1\|_{\Sigma^{-1}}^2\|\eta_2\|_{\Sigma^{-1}}^2\,(k!)^3\,2^k\,K_x^{4k}\,\frac{d^{k-1}}{(n)_k}.
    \end{equation*}
    This proves \eqref{eq:B1:var}.
\end{proof}

\begin{theorem}\label{thm:B2}
    Consider a functional $\omega(\Sigma) = \eta_1^\top \Sigma^{-1}\eta_2$ and recall $k$:th Fréchet derivative $\Dc^kw$ of $\omega$ in \eqref{eq:precisionmatrix_Fréchetderivatives} for $1\leq k\leq s$. Suppose that Assumption~\ref{asmp:momentX} holds. Let
    \begin{equation*}
        \bar h^{(k)}=\sum_{\substack{1\leq j_1,\ldots,j_k\leq n\\\mbox{\tiny distinct}}}\frac{h_{j_1}\otimes\cdots\otimes h_{j_k}}{n(n-1)\cdots (n-k+1)},
    \end{equation*} with $h_i = X_iX_i^\top-\Sigma$. For any invertible $\tilde\Sigma$, independent of $X_1,\ldots,X_n$, let $\Delta^{(s)}(t) = \Dc^s\omega(\tilde\Sigma + t(\Sigma-\tilde\Sigma))-\Dc^s\omega(\tilde\Sigma)$ and let $\tilde \Delta^{(s)}_k[v_1,\ldots,v_k] = \Jc^{s-k}\Delta^{(s)}[v_1,\ldots,v_k,(\Sigma-\tilde\Sigma)^{\otimes s-k}]$ for $1\leq k\leq s$. Define
    \begin{equation*}
    \underline\lambda_{\tilde\Sigma} = \lambda_{\min}(\Sigma^{-1/2}\tilde\Sigma\Sigma^{-1/2})\quad\mbox{and}\quad \Dc_{\tilde\Sigma} = \norm{I-\Sigma^{-1/2}\tilde\Sigma\Sigma^{-1/2}}_{\rm op}.
\end{equation*}Then $\Eb[\tilde \Delta^{(s)}_k[\bar h^{(k)}]]=0$ for $1\leq k\leq s$ and $$\abs{\tilde \Delta^{(s)}_0}\leq \norm{\eta_1}_{\Sigma^{-1}}\norm{\eta_2}_{\Sigma^{-1}}\Dc_{\tilde\Sigma}^{s+1}/(1\wedge\underline\lambda_{\tilde\Sigma}^{s+2}).$$Also, for all $1\leq k\leq s$,
    \begin{equation*}
        {\rm Var}\left(\tilde \Delta^{(s)}_k[\bar h^{(k)}]\right)\leq \norm{\eta_1}_{\Sigma^{-1}}^2\norm{\eta_2}_{\Sigma^{-1}}^2\frac{k!(4sK_x^2)^{2k}\Dc_{\tilde\Sigma}^{2(s-k+1)}}{1\wedge\underline\lambda_{\tilde\Sigma}^{2s+4}}\frac{d^{k-1}}{(n)_k}.
    \end{equation*}
\end{theorem}
\begin{proof}
Denote $\Eb_{\Sc_1}$ and ${\rm Var}_{\Sc_1}$ as the conditional expectation and conditional variance taken over
$\{X_1,\ldots,X_n\}$ only, respectively. Recall $h_i=X_iX_i^\top-\Sigma$ and $\Eb_{\Sc_1}[h_i]=0$.
For $t\in[0,1]$ and $1\le k\le s$, the mapping
\[
(v_1,\ldots,v_k)\mapsto \Delta^{(s)}(t)\big[v_1,\ldots,v_k,(\Sigma-\tilde\Sigma)^{\otimes(s-k)}\big]
\]
is $k$-linear and (conditionally on $\tilde\Sigma$) deterministic, hence
\begin{equation}\label{eq:B2:mean0}
\Eb_{\Sc_1}\Big[\Delta^{(s)}(t)\big[\bar h^{(k)},(\Sigma-\tilde\Sigma)^{\otimes(s-k)}\big]\Big]=0,
\qquad t\in[0,1].
\end{equation}
Since $\tilde\Delta_k^{(s)}=\Jc^{s-k}\Delta^{(s)}$ is obtained by integrating $\Delta^{(s)}(t)$ against a deterministic weight,
\eqref{eq:B2:mean0} implies $\Eb[\tilde\Delta_k^{(s)}[\bar h^{(k)}]]=0$ for all $1\le k\le s$.

We next bound the variance. For any square-integrable centered process $\{S_t:t\in[0,1]\}$ and any deterministic weight $w\in L_2[0,1]$,
Fubini and Cauchy--Schwarz yield
\begin{align}\label{eq:B2:var-weight}
&{\rm Var}\Big(\int_0^1 w(t)S_t\,dt\Big)
=\int_0^1\int_0^1 w(t)w(t')\,{\rm Cov}(S_t,S_{t'})\,dt\,dt' \nonumber\\
&\quad \le \int_0^1\int_0^1 |w(t)w(t')|\sqrt{{\rm Var}(S_t){\rm Var}(S_{t'})}\,dt\,dt'
\le \int_0^1 w(t)^2{\rm Var}(S_t)\,dt .
\end{align}
Fix $1\le k\le s$ and set
\[
S_t:=\Delta^{(s)}(t)\big[\bar h^{(k)},(\Sigma-\tilde\Sigma)^{\otimes(s-k)}\big],
\qquad
w(t):=\begin{cases}
(1-t)^{s-k-1}/\Gamma(s-k), & k\le s-1,\\
1, & k=s.
\end{cases}
\]
Then $\tilde\Delta_k^{(s)}[\bar h^{(k)}]=\int_0^1 w(t)S_t\,dt$ (with the convention $\Jc^0 g=g(1)$), and hence by \eqref{eq:B2:var-weight},
\begin{equation}\label{eq:B2:var-tildeD-step}
{\rm Var}_{\Sc_1}\big(\tilde\Delta_k^{(s)}[\bar h^{(k)}]\big)\le \int_0^1 w(t)^2\,{\rm Var}_{\Sc_1}(S_t)\,dt .
\end{equation}

We now bound ${\rm Var}_{\Sc_1}(S_t)$. For fixed $t$, define the symmetric kernel
\[
\psi_t(v_1,\ldots,v_k)
:=\Delta^{(s)}(t)\big[v_1,\ldots,v_k,(\Sigma-\tilde\Sigma)^{\otimes(s-k)}\big].
\]
By multilinearity and $\Eb_{\Sc_1}[h_i]=0$, $\psi_t$ is canonical (conditionally on $\tilde\Sigma$), and the variance formula for
canonical U-statistics gives
\begin{equation}\label{eq:B2:Uvar}
{\rm Var}_{\Sc_1}(S_t)=\binom{n}{k}^{-1}{\rm Var}_{\Sc_1}\big(\psi_t(h_1,\ldots,h_k)\big)
\le \binom{n}{k}^{-1}\Eb_{\Sc_1}\big[\psi_t(h_1,\ldots,h_k)^2\big].
\end{equation}

It remains to bound $\Eb_{\Sc_1}[\psi_t(h_1,\ldots,h_k)^2]$. Let $Z_i:=\Sigma^{-1/2}X_i$, so $\Eb[Z_iZ_i^\top]=I_d$, and set
\[
A_i:=\Sigma^{-1/2}h_i\Sigma^{-1/2}=Z_iZ_i^\top-I_d,\qquad 1\le i\le k,
\qquad
A_{k+1}=\cdots=A_s:=\Sigma^{-1/2}(\Sigma-\tilde\Sigma)\Sigma^{-1/2}.
\]
Also write $\tilde\eta_j:=\Sigma^{-1/2}\eta_j$ ($j=1,2$) and define for $t\in[0,1]$,
\begin{equation}\label{eq:B2:Bt}
\Bv_t:=\Sigma^{1/2}\big(\tilde\Sigma+t(\Sigma-\tilde\Sigma)\big)^{-1}\Sigma^{1/2}
=\Big((1-t)\Sigma^{-1/2}\tilde\Sigma\Sigma^{-1/2}+tI_d\Big)^{-1}.
\end{equation}
By the explicit formula for $\Dc^s\omega$,
\[
\Dc^s\omega(H)[u_1\otimes\cdots\otimes u_s]
=(-1)^s\sum_{\pi\in\mathfrak S_s}\eta_1^\top H^{-1}u_{\pi(1)}H^{-1}\cdots H^{-1}u_{\pi(s)}H^{-1}\eta_2,
\]
we may write, with the convention $h_{k+1}=\cdots=h_s=\Sigma-\tilde\Sigma$,
\begin{align*}
\psi_t(h_1,\ldots,h_k)
&=\Delta^{(s)}(t)\big[h_1\otimes\cdots\otimes h_s\big] \\
&=(-1)^s\sum_{\pi\in\mathfrak S_s}\Big\{
\tilde\eta_1^\top \Bv_t A_{\pi(1)}\Bv_t\cdots \Bv_t A_{\pi(s)}\Bv_t\tilde\eta_2
-\tilde\eta_1^\top \Bv_0 A_{\pi(1)}\Bv_0\cdots \Bv_0 A_{\pi(s)}\Bv_0\tilde\eta_2
\Big\}.
\end{align*}
For each fixed $\pi$, apply the telescoping product identity:
\begin{align}\label{eq:B2:telescoping}
&\tilde\eta_1^\top \Bv_t A_{\pi(1)}\Bv_t\cdots \Bv_t A_{\pi(s)}\Bv_t\tilde\eta_2
-\tilde\eta_1^\top \Bv_0 A_{\pi(1)}\Bv_0\cdots \Bv_0 A_{\pi(s)}\Bv_0\tilde\eta_2 \nonumber\\
&\quad=\sum_{j=0}^{s}\tilde\eta_1^\top
\Big(\Bv_0A_{\pi(1)}\Bv_0\cdots \Bv_0A_{\pi(j)}\Bv_0\Big)
(\Bv_t-\Bv_0)
\Big(\Bv_tA_{\pi(j+1)}\Bv_t\cdots \Bv_tA_{\pi(s)}\Bv_t\Big)\tilde\eta_2,
\end{align}
where the empty product equals $I_d$. Hence $\psi_t(h_1,\ldots,h_k)$ is a sum of at most $(s+1)s!=(s+1)!$ scalar terms.
Therefore, $\psi_t(h_1,\ldots,h_k)^2$ expands into at most $\{(s+1)!\}^2$ products.

To bound each product, note that every summand in \eqref{eq:B2:telescoping} has the schematic form
\begin{equation}\label{eq:B2:chain}
    c_1^\top A_{\tau(1)}C_1A_{\tau(2)}\cdots C_{k-1}A_{\tau(k)}c_2,
\end{equation}where $\tau$ is a permutation of $[k]$ and where the vectors/matrices $c_1,c_2,C_1,\ldots,C_{k-1}$ are $\sigma(\tilde\Sigma)$-measurable
functions of $\tilde\eta_1,\tilde\eta_2,\Bv_0,\Bv_t,\Bv_t-\Bv_0$, and $A_{k+1}=\cdots=A_s$.
Moreover, there exists a deterministic upper bound
\begin{equation}\label{eq:B2:Pit}
\|c_1\|_2\|c_2\|_2\prod_{r=1}^{k-1}\|C_r\|_{\op}\le \Pi_t:=\|\tilde\eta_1\|_2\|\tilde\eta_2\|_2\,
\|\Bv_t-\Bv_0\|_{\op}\,(\|\Bv_t\|_{\op}\vee\|\Bv_0\|_{\op})^{s}\,
\|A_{k+1}\|_{\op}^{\,s-k}.
\end{equation}

Now fix one product of two such chains of the form in \eqref{eq:B2:chain}. Writing $A_i=Z_iZ_i^\top-I_d$ for $i\le k$ and expanding only the second chain,
\[
\prod_{r=1}^{k}\tilde C_r(Z_{\tau'(r)}Z_{\tau'(r)}^\top-I_d)
=\sum_{J\subset[k]}(-1)^{k-|J|}\prod_{r\in J}{}^\uparrow \tilde C_rZ_{\tau'(r)}Z_{\tau'(r)}^\top\prod_{r\in [k]\setminus J}{}^\uparrow \tilde C_r,
\]
we represent the product as a sum of at most $2^k$ terms, each of which is a product of two non-centered chains involving only factors
$Z_iZ_i^\top$. A generic term can be written as
\begin{align}
&c_1^\top Z_{\tau(1)}Z_{\tau(1)}^\top C_1Z_{\tau(2)}Z_{\tau(2)}^\top\cdots C_{k-1}Z_{\tau(k)}Z_{\tau(k)}^\top c_2 \nonumber\\
&\qquad\times {d'_1}^\top Z_{i_1}Z_{i_1}^\top D'_1\cdots D'_{h-1}Z_{i_h}Z_{i_h}^\top d'_2,\label{eq:generic-term}
\end{align}
for some $0\le h\le k$, where
\[
\|d'_1\|_2\|d'_2\|_2\prod_{r=1}^{h-1}\|D'_r\|_{\op}\le \Pi_t.
\]
Since $Z_i$ is isotropic and obeys Assumption (M)(X) with constant $K_x$, we apply Theorem~C.3 (inequality (49)) in Section~C
with $\sigma=1$, $\tau=d$, $m_1=k-1$, $m_2=h-1$, and $\kappa_z=K_x$. This yields the explicit bound
\[
\Big|\Eb_{\Sc_1}[\mbox{\eqref{eq:generic-term}}]\Big|
\le K_x^{2(m_1+m_2)+4}\,d^{(m_1+m_2)+1-r}
\cdot \|c_1\|_2\|c_2\|_2\prod_{r=1}^{k-1}\|C_r\|_{\op}\cdot
\|d'_1\|_2\|d'_2\|_2\prod_{r=1}^{h-1}\|D'_r\|_{\op},
\]
where $r$ is the number of distinct indices in the two-chain product. Since the first chain uses $k$ distinct indices, $r\ge k$,
hence $(m_1+m_2)+1-r=(k+h-1)-r\le h-1\le k-1$, and since $h\le k$ we have $K_x^{2k+2h}\le K_x^{4k}$. Therefore,
\begin{equation}\label{eq:B2:chain-bound}
\Big|\Eb_{\Sc_1}[\mbox{\eqref{eq:generic-term}}]\Big|
\le \Pi_t^2\,K_x^{4k}\,d^{k-1}.
\end{equation}
Combining \eqref{eq:B2:chain-bound} with the factor $2^k$ from the expansion and the factor $\{(s+1)!\}^2$ from the square expansion,
we obtain
\begin{equation}\label{eq:B2:psibound}
\Eb_{\Sc_1}\big[\psi_t(h_1,\ldots,h_k)^2\big]
\le \{(s+1)!\}^2\,2^k\,K_x^{4k}\,\Pi_t^2\,d^{k-1}.
\end{equation}

It remains to bound $\Pi_t$. Recall $A_{k+1}=I_d-\Sigma^{-1/2}\tilde\Sigma\Sigma^{-1/2}$ so $\|A_{k+1}\|_{\op}=\Dc_{\tilde\Sigma}$.
Moreover, from \eqref{eq:B2:Bt},
\begin{equation}\label{eq:B2:Bt-op}
\|\Bv_t\|_{\op}\le \big((1-t)\underline\lambda_{\tilde\Sigma}+t\big)^{-1}\le (1\wedge \underline\lambda_{\tilde\Sigma})^{-1},
\qquad
\|\Bv_0\|_{\op}=\underline\lambda_{\tilde\Sigma}^{-1}.
\end{equation}
Using $\Bv_t-\Bv_0=\Bv_0(\Bv_0^{-1}-\Bv_t^{-1})\Bv_t$ and $\Bv_0^{-1}-\Bv_t^{-1}=t(I_d-\Sigma^{-1/2}\tilde\Sigma\Sigma^{-1/2})$,
we obtain
\begin{equation}\label{eq:B2:Bt-diff}
\|\Bv_t-\Bv_0\|_{\op}
\le t\,\|\Bv_0\|_{\op}\,\Dc_{\tilde\Sigma}\,\|\Bv_t\|_{\op}
\le t\,\Dc_{\tilde\Sigma}\,(1\wedge \underline\lambda_{\tilde\Sigma})^{-2}.
\end{equation}
Plugging \eqref{eq:B2:Bt-op}--\eqref{eq:B2:Bt-diff} into \eqref{eq:B2:Pit} and using $\|\tilde\eta_j\|_2=\|\eta_j\|_{\Sigma^{-1}}$ gives
\begin{equation}\label{eq:B2:Pit-bound}
\Pi_t^2
\le \|\eta_1\|_{\Sigma^{-1}}^2\|\eta_2\|_{\Sigma^{-1}}^2\,
\frac{t^2\,\Dc_{\tilde\Sigma}^{2(s-k+1)}}{1\wedge \underline\lambda_{\tilde\Sigma}^{2s+4}}.
\end{equation}
Combining \eqref{eq:B2:psibound} and \eqref{eq:B2:Pit-bound} with \eqref{eq:B2:Uvar} yields, for all $t\in[0,1]$,
\begin{equation}\label{eq:B2:var-St}
{\rm Var}_{\Sc_1}(S_t)
\le \binom{n}{k}^{-1}\,
\frac{\{(s+1)!\}^2\,2^k\,K_x^{4k}\,\|\eta_1\|_{\Sigma^{-1}}^2\|\eta_2\|_{\Sigma^{-1}}^2\,
\Dc_{\tilde\Sigma}^{2(s-k+1)}}{1\wedge \underline\lambda_{\tilde\Sigma}^{2s+4}}\,
t^2\,d^{k-1}.
\end{equation}

We now insert \eqref{eq:B2:var-St} into \eqref{eq:B2:var-tildeD-step}. If $k=s$, then $w(t)\equiv 1$ and
$\tilde\Delta_s^{(s)}[\bar h^{(s)}]=S_1$, so the variance bound follows directly from \eqref{eq:B2:var-St} at $t=1$ and
$\binom{n}{s}^{-1}\le s!/n^s$.

Assume $1\le k\le s-1$ and write $m:=s-k\ge 1$. Then $w(t)=(1-t)^{m-1}/\Gamma(m)$, and \eqref{eq:B2:var-tildeD-step} and \eqref{eq:B2:var-St} give
\begin{align*}
{\rm Var}_{\Sc_1}\big(\tilde\Delta_k^{(s)}[\bar h^{(k)}]\big)
&\le \binom{n}{k}^{-1}\,
\frac{\{(s+1)!\}^2\,2^k\,K_x^{4k}\,\|\eta_1\|_{\Sigma^{-1}}^2\|\eta_2\|_{\Sigma^{-1}}^2\,
\Dc_{\tilde\Sigma}^{2(m+1)}}{1\wedge \underline\lambda_{\tilde\Sigma}^{2s+4}}\,
d^{k-1}\int_0^1 \frac{t^2(1-t)^{2m-2}}{\Gamma(m)^2}\,dt .
\end{align*}
Since $m\in\mathbb N$,
\[
\int_0^1 t^2(1-t)^{2m-2}\,dt
=\frac{\Gamma(3)\Gamma(2m-1)}{\Gamma(2m+2)}
=\frac{2}{(2m+1)(2m)(2m-1)}
\le \frac{1}{m^2},
\]
and $\Gamma(m)=(m-1)!$, we obtain
\[
\int_0^1 \frac{t^2(1-t)^{2m-2}}{\Gamma(m)^2}\,dt
\le \frac{1}{m^2 (m-1)!^2}
=\frac{1}{(m!)^2}.
\]
Therefore,
\[
{\rm Var}_{\Sc_1}\big(\tilde\Delta_k^{(s)}[\bar h^{(k)}]\big)
\le \binom{n}{k}^{-1}\,
\frac{\{(s+1)!\}^2}{(s-k)!^2}\,
\frac{2^k K_x^{4k}\,\|\eta_1\|_{\Sigma^{-1}}^2\|\eta_2\|_{\Sigma^{-1}}^2\,
\Dc_{\tilde\Sigma}^{2(s-k+1)}}{1\wedge \underline\lambda_{\tilde\Sigma}^{2s+4}}\,
d^{k-1}.
\]
Since $\frac{(s+1)!}{(s-k)!}=(s+1)s\cdots(s-k+1)\le (s+1)^k,
$, we conclude
\begin{equation}\label{eq:B2:var-k-pre}
{\rm Var}\big(\tilde\Delta_k^{(s)}[\bar h^{(k)}]\big)
\le \|\eta_1\|_{\Sigma^{-1}}^2\|\eta_2\|_{\Sigma^{-1}}^2\,
\frac{k!\,(s+1)^{2k}\,2^k\,K_x^{4k}\,\Dc_{\tilde\Sigma}^{2(s-k+1)}}{1\wedge \underline\lambda_{\tilde\Sigma}^{2s+4}}\,
\frac{d^{k-1}}{(n)_k}.
\end{equation}
Finally, for all $s\ge 1$ we have $2(s+1)^2\le 8s^2\le (4s)^2$, hence
\[
(s+1)^{2k}2^kK_x^{4k}=\big(2(s+1)^2K_x^4\big)^k\le \big((4s)^2K_x^4\big)^k=(4sK_x^2)^{2k}.
\]
Thus \eqref{eq:B2:var-k-pre} implies the stated variance bound.

It remains to bound $\tilde\Delta_0^{(s)}$. Using \eqref{eq:B2:telescoping} with all inputs equal to $(\Sigma-\tilde\Sigma)$ yields, for $t\in[0,1]$,
\begin{align*}
\big|\Delta^{(s)}(t)[(\Sigma-\tilde\Sigma)^{\otimes s}]\big|
&\le (s+1)!\,\|\tilde\eta_1\|_2\|\tilde\eta_2\|_2\,
\|\Bv_t-\Bv_0\|_{\op}\,(\|\Bv_t\|_{\op}\vee\|\Bv_0\|_{\op})^{s}\,\|A_{k+1}\|_{\op}^{\,s} \\
&=(s+1)!\,\|\eta_1\|_{\Sigma^{-1}}\|\eta_2\|_{\Sigma^{-1}}\,
\|\Bv_t-\Bv_0\|_{\op}\,(\|\Bv_t\|_{\op}\vee\|\Bv_0\|_{\op})^{s}\,\Dc_{\tilde\Sigma}^{s}.
\end{align*}
Applying \eqref{eq:B2:Bt-op} and \eqref{eq:B2:Bt-diff} gives
\[
\big|\Delta^{(s)}(t)[(\Sigma-\tilde\Sigma)^{\otimes s}]\big|
\le t(s+1)!\,\|\eta_1\|_{\Sigma^{-1}}\|\eta_2\|_{\Sigma^{-1}}\,
\frac{\Dc_{\tilde\Sigma}^{s+1}}{1\wedge \underline\lambda_{\tilde\Sigma}^{s+2}}.
\]
Since $\tilde\Delta_0^{(s)}=\Jc^{s}\Delta^{(s)}$ and $\int_0^1 t(1-t)^{s-1}\,dt/\Gamma(s)=1/(s+1)!$, we obtain
\begin{equation*}
|\tilde\Delta_0^{(s)}|
\le \|\eta_1\|_{\Sigma^{-1}}\|\eta_2\|_{\Sigma^{-1}}\,
\frac{\Dc_{\tilde\Sigma}^{s+1}}{1\wedge \underline\lambda_{\tilde\Sigma}^{s+2}}.
\end{equation*}
This completes the proof.
\end{proof}

\begin{theorem}\label{thm:B3}
    Consider a functional $\beta_\eta(\Sigma,\Gamma) = \eta^\top \Sigma^{-1}\Gamma$ and recall $k$:th Fréchet derivative $\Dc^k\beta_\eta$ in \eqref{eq:regressionfunctional_derivative} for $1\leq k\leq s$. Suppose that Assumption~\ref{asmp:momentX} and Assumption~\ref{asmp:momente} hold with $q_x\geq 4$ and $(1/q_x+1/q_\varepsilon)^{-1}\geq 1/2$. Let
    \begin{equation*}
        \bar h^{(k)}=\sum_{\substack{1\leq j_1,\ldots,j_k\leq n\\\mbox{\tiny distinct}}}\frac{h_{j_1}\otimes\cdots\otimes h_{j_k}}{n(n-1)\cdots (n-k+1)},
    \end{equation*}
    with $h_i = (X_iX_i^\top-\Sigma,X_iY_i-\Gamma)$. Then $\Eb[\Dc^k \beta_\eta(\Sigma,\Gamma)[\bar h^{(k)}]]=0$ and
    \begin{equation}\label{eq:B3:var}
        {\rm Var}\!\left(\Dc^k \beta_\eta(\Sigma,\Gamma)[\bar h^{(k)}]\right)\leq \|\eta\|_{\Sigma^{-1}}^2 (k!)^3 2^{k-1} K_x^{4k-2}K_\varepsilon^2\frac{d^{k-1}}{(n)_k},\quad 1\leq k\leq s.
    \end{equation}
\end{theorem}

\begin{proof}
    Since $\Eb[h_i]=0$ and $\Dc^k\beta_\eta(\Sigma,\Gamma)$ is $k$-linear, $\Eb[\Dc^k \beta_\eta(\Sigma,\Gamma)[\bar h^{(k)}]]=0$.

    Recall from \eqref{eq:regressionfunctional_derivative} that
    \begin{equation*}
        \Dc^k \beta_\eta(\Sigma,\Gamma)[\bar h^{(k)}]
        =\binom{n}{k}^{-1}\sum_{1\leq j_1<\ldots<j_k\leq n}\varphi_\beta(h_{j_1},\ldots,h_{j_k}),
    \end{equation*}
    where
    \begin{equation}\label{eq:B3:kernel}
        \varphi_\beta(h_1,\ldots,h_k)
        =(-1)^{k-1}\sum_{\pi\in\mathfrak S_k}\tilde\eta^\top A_{\pi(1)}\cdots A_{\pi(k-1)} Z_{\pi(k)}\varepsilon_{\pi(k)}
    \end{equation} for $\tilde\eta=\Sigma^{-1/2}\eta$ and $A_i=Z_iZ_i^\top-I_d.$ By multilinearity in $(h_1,\ldots,h_k)$ and $\Eb[h_i]=0$, $\varphi_\beta$ is canonical: $\Eb\big[\varphi_\beta(h_1,\ldots,h_k)\mid h_1,\ldots,h_{k-1}\big]=0.$ Hence the canonical $U$-statistic variance identity yields
    \begin{equation}\label{eq:B3:var-reduce}
        {\rm Var}\!\left(\Dc^k \beta_\eta(\Sigma,\Gamma)[\bar h^{(k)}]\right)
        =\binom{n}{k}^{-1}{\rm Var}\!\left(\varphi_\beta(h_1,\ldots,h_k)\right)
        \le \binom{n}{k}^{-1}\Eb\!\left[\varphi_\beta(h_1,\ldots,h_k)^2\right].
    \end{equation} Expand the centered factor as
    \begin{equation}\label{eq:B3:last-expand}
         \prod_{r=1}^{k-1} (Z_{\pi'(r)}Z_{\pi'(r)}^\top-I_d)
        =\sum_{J\subset[k-1]}(-1)^{k-1-|J|}\prod_{r\in J} {}^\uparrow Z_{\pi'(r)}Z_{\pi'(r)}^\top.
    \end{equation}
    Then, expanding the centered factors as in \eqref{eq:B1:center-expand}, $\Eb[\varphi_\beta(h_1,\ldots,h_k)^2]$ is bounded by
    a sum of at most $k!^2 2^{k-1}$ terms, each of which is of the form
    \begin{align}\label{eq:B3:term}
        \Big|\Eb\Big[\tilde\eta^\top \prod_{r=1}^{k-1} (Z_{\pi'(r)}Z_{\pi'(r)}) ~Z_{\pi(k)}\varepsilon_{\pi(k)}\times \tilde\eta^\top \prod_{r\in J}{}^\uparrow (Z_{\pi'(r)}Z_{\pi'(r)}^\top) ~Z_{\pi'(k)}\varepsilon_{\pi'(k)}\Big]\Big|.
    \end{align} Applying Theorem~\ref{thm:C4} yields $|\Eb[\mbox{term in \eqref{eq:B3:term}}]|
        \le \|\tilde\eta\|_2^2\,K_x^{4k-2}K_\varepsilon^2\,d^{k-1}.$ Therefore,
    \begin{equation}\label{eq:B3:second-moment}
        \Eb\!\left[\varphi_\beta(h_1,\ldots,h_k)^2\right]
        \le (k!)^2 2^{k-1}\cdot \|\tilde\eta\|_2^2\,K_x^{4k-2}K_\varepsilon^2\,d^{k-1}.
    \end{equation}
    Combining \eqref{eq:B3:var-reduce} and \eqref{eq:B3:second-moment}, we get
    \begin{equation*}
        {\rm Var}\!\left(\Dc^k \beta_\eta(\Sigma,\Gamma)[\bar h^{(k)}]\right)
        \le \|\eta\|_{\Sigma^{-1}}^2 (k!)^3 2^{k-1} K_x^{4k-2}K_\varepsilon^2\frac{d^{k-1}}{(n)_k}.
    \end{equation*}
\end{proof}

\begin{theorem}\label{thm:B4}
    Consider a functional $\beta_\eta(\Sigma,\Gamma) = \eta^\top \Sigma^{-1}\Gamma$ and recall $k$:th Fr\'echet derivative $\Dc^k\beta_\eta$ of $\beta_\eta$ in \eqref{eq:regressionfunctional_derivative} for $1\leq k\leq s$. Suppose that Assumption~\ref{asmp:momentX} and \ref{asmp:momente} holds with $q_x\geq 4$ and $(1/q_x+1/q_y)^{-1}\geq 1/2$. Recall
    \begin{equation*}
        \bar h^{(k)}=\sum_{\substack{1\leq j_1,\ldots,j_k\leq n\\\mbox{\tiny distinct}}}\frac{h_{j_1}\otimes\cdots\otimes h_{j_k}}{n(n-1)\cdots (n-k+1)},
    \end{equation*}
    with $h_i = (X_iX_i^\top-\Sigma,X_iY_i-\Gamma)$. For $\tilde\Sigma\in\Sb\Pb^d$ and $\tilde\Gamma\in\Real^d$, independent of $X_1,\ldots,X_n$, let $\tilde\Sigma_t = \tilde\Sigma + t(\Sigma-\tilde\Sigma)$ and $\tilde\Gamma_t = \tilde\Gamma + t(\Gamma-\tilde\Gamma)$. Define
    \begin{equation*}
        \Delta^{(s)}(t) = \Dc^s\beta_\eta(\tilde\Sigma_t, \tilde\Gamma_t)-\Dc^s\beta_\eta(\tilde\Sigma,\tilde\Gamma).
    \end{equation*}
    For $0\leq k\leq s$ and $v_1,\ldots,v_k\in\Sb\Pb^{d}\oplus\Real^d$, let
    \begin{equation*}
        \tilde\Delta^{(s)}_k[v_1,\ldots,v_k]
        = \Jc^{s-k}\Delta^{(s)}[v_1,\ldots,v_k, (\Sigma-\tilde\Sigma,\Gamma-\tilde\Gamma)^{\otimes (s-k)}].
    \end{equation*}
    Define
    \begin{equation*}
        \underline\lambda_{\tilde\Sigma} = \lambda_{\min}(\Sigma^{-1/2}\tilde\Sigma\Sigma^{-1/2}),\quad
        \Dc_{\tilde\Sigma} = \norm{I-\Sigma^{-1/2}\tilde\Sigma\Sigma^{-1/2}}_{\rm op},\quad
        \Rc_{\tilde\Sigma,\tilde\Gamma}=\norm{\tilde\Gamma-\tilde\Sigma\beta}_{\Sigma^{-1}}.
    \end{equation*}
    Then $\Eb_{\Sc_1}[\tilde \Delta^{(s)}_k[\bar h^{(k)}]]=0$ for $1\leq k\leq s$ and
    \begin{equation*}
        \abs{\tilde \Delta^{(s)}_0}\leq \norm{\eta}_{\Sigma^{-1}}(1\vee\underline\lambda_{\tilde\Sigma}^{-s-2})\Dc_{\tilde\Sigma}^s\Rc_{\tilde\Sigma,\tilde\Gamma}.
    \end{equation*}
    Also, for all $1\leq k\leq s$,
    \begin{equation*}
        {\rm Var}_{\Sc_1}\big(\tilde \Delta^{(s)}_k[\bar h^{(k)}]\big)
        \leq \norm{\eta}_{\Sigma^{-1}}^2k!s^{2k+2}2^{k+2} K_x^{4k-2}(K_x\vee K_\varepsilon)^2
        \frac{\Dc_{\tilde\Sigma}^{2s-2k}(\Dc_{\tilde\Sigma}\vee\Rc_{\tilde\Sigma,\tilde\Gamma}\vee\Dc_{\tilde\Sigma}\Rc_{\tilde\Sigma,\tilde\Gamma})^2}{1\wedge\underline\lambda_{\tilde\Sigma}^{2s+4}}
        \frac{d^{k-1}}{(n)_k}.
    \end{equation*}
\end{theorem}

\begin{proof}
    Denote $\Eb_{\Sc_1}$ and ${\rm Var}_{\Sc_1}$ as the conditional expectation and conditional variance taken over $\{(X_i,Y_i):i\leq n\}$ only. Since each $h_i$ is centered and $(\tilde\Sigma,\tilde\Gamma)$ is independent of $\Sc_1$, it follows that for every $t\in[0,1]$ and $1\leq k\leq s$,
    \begin{equation}\label{eq:B4:mean0}
        \Eb_{\Sc_1}\Big[\Delta^{(s)}(t)\big[\bar h^{(k)},(\Sigma-\tilde\Sigma,\Gamma-\tilde\Gamma)^{\otimes(s-k)}\big]\Big]=0,
    \end{equation}
    and hence $\Eb_{\Sc_1}[\tilde\Delta^{(s)}_k[\bar h^{(k)}]]=0$ for $1\leq k\leq s$.

    For $1\leq k\leq s-1$, by \eqref{eq:B2:var-weight} and the definition of $\Jc^{s-k}$,
    \begin{align}\label{eq:B4:var-int}
        {\rm Var}_{\Sc_1}\big(\tilde\Delta^{(s)}_k[\bar h^{(k)}]\big)
        &\leq \int_0^1 {\rm Var}_{\Sc_1}\Big(\Delta^{(s)}(t)\big[\bar h^{(k)},(\Sigma-\tilde\Sigma,\Gamma-\tilde\Gamma)^{\otimes(s-k)}\big]\Big)
        \frac{(1-t)^{2s-2k-2}}{\Gamma(s-k)^2}\,dt .
    \end{align}
    For $k=s$, $\tilde\Delta^{(s)}_s[\bar h^{(s)}]=\Delta^{(s)}(1)[\bar h^{(s)}]$. Moreover, the variance formula for degenerate $U$-statistics of order $k$ gives
    \begin{align}\label{eq:B4:ustatvar}
        {\rm Var}_{\Sc_1}\Big(\Delta^{(s)}(t)\big[\bar h^{(k)},(\Sigma-\tilde\Sigma,\Gamma-\tilde\Gamma)^{\otimes(s-k)}\big]\Big)\nonumber\\
        = \binom{n}{k}^{-1}\,
        {\rm Var}_{\Sc_1}\Big(\Delta^{(s)}(t)\big[h_1,\ldots, h_k,(\Sigma-\tilde\Sigma,\Gamma-\tilde\Gamma)^{\otimes(s-k)}\big]\Big).
    \end{align}
    Write $M_i=X_iX_i^\top-\Sigma$ and $\gamma_i=X_iY_i-\Gamma$ for $1\leq i\leq k$, and set $M_i=\Sigma-\tilde\Sigma$ and $\gamma_i=\Gamma-\tilde\Gamma$ for $k+1\leq i\leq s$. Let $\tv_i=(M_i,\gamma_i)$. Then
    \begin{equation*}
        \Delta^{(s)}(t)\big[h_1,\ldots, h_k,(\Sigma-\tilde\Sigma,\Gamma-\tilde\Gamma)^{\otimes(s-k)}\big]
        =\Delta^{(s)}(t)[\tv_1\otimes\cdots\otimes\tv_s].
    \end{equation*}

    Let $\tilde\eta=\Sigma^{-1/2}\eta$, $\Bv_t=\Sigma^{1/2}\tilde\Sigma_t^{-1}\Sigma^{1/2}$, and $\tilde\beta_t=\tilde\Sigma_t^{-1}\tilde\Gamma_t$. Using \eqref{eq:regressionfunctional_derivative},
    \begin{align}\label{eq:B4:chain-diff}
        \Delta^{(s)}(t)[\tv_1\otimes\cdots\otimes\tv_s]
        &=(-1)^s\sum_{\pi:[s]\to[s]}\Big\{
        \tilde\eta^\top\Bv_t(\Sigma^{-1/2}M_{\pi(1)}\Sigma^{-1/2})\cdots \Bv_t\Sigma^{-1/2}(M_{\pi(s)}\tilde\beta_t-\gamma_{\pi(s)}) \nonumber\\
        &\qquad\quad-\tilde\eta^\top\Bv_0(\Sigma^{-1/2}M_{\pi(1)}\Sigma^{-1/2})\cdots \Bv_0\Sigma^{-1/2}(M_{\pi(s)}\tilde\beta_0-\gamma_{\pi(s)})\Big\}.
    \end{align}
    Applying the telescoping identity to the difference in braces yields
    \begin{equation}\label{eq:B4:gvpi}
        \Delta^{(s)}(t)[\tv_1\otimes\cdots\otimes\tv_s]
        =(-1)^s\sum_{\pi:[s]\to[s]}\gv_\pi(t),
    \end{equation}
    where $\gv_\pi(t)$ is the sum of $(s+1)$ terms:
    \begin{align}\label{eq:B4:gvpi-expand}
        \gv_\pi(t)
        &=\sum_{r=0}^{s-1}\tilde\eta^\top
        \Big(\Bv_0(\Sigma^{-1/2}M_{\pi(1)}\Sigma^{-1/2})\cdots \Bv_0(\Sigma^{-1/2}M_{\pi(r)}\Sigma^{-1/2})\Bv_0\Big)
        (\Bv_t-\Bv_0) \nonumber\\
        &\qquad\qquad\qquad\quad\times
        \Big((\Sigma^{-1/2}M_{\pi(r+1)}\Sigma^{-1/2})\Bv_t\cdots \Bv_t\Sigma^{-1/2}(M_{\pi(s)}\tilde\beta_t-\gamma_{\pi(s)})\Big) \\
        &\quad+\tilde\eta^\top\Bv_0(\Sigma^{-1/2}M_{\pi(1)}\Sigma^{-1/2})\cdots \Bv_0(\Sigma^{-1/2}M_{\pi(s)}\Sigma^{-1/2})\Sigma^{1/2}(\tilde\beta_t-\tilde\beta_0). \nonumber
    \end{align}

    We first consider $k=0$, so that $M_i=\Sigma-\tilde\Sigma$ and $\gamma_i=\Gamma-\tilde\Gamma$ for all $i\leq s$. From \eqref{eq:B4:gvpi-expand},
    \begin{align}\label{eq:mean}
        \big|\Delta^{(s)}(t)[(\Sigma-\tilde\Sigma,\Gamma-\tilde\Gamma)^{\otimes s}]\big|
        &\le (s+1)!\,\|\tilde\eta\|_2\,(\|\Bv_0\|_{\rm op}\vee\|\Bv_t\|_{\rm op})^{s-1}\Dc_{\tilde\Sigma}^{s-1} \nonumber\\
        &\qquad\times \max\Big\{\|\Bv_t-\Bv_0\|_{\rm op}\,\|\tilde\beta_t-\beta\|_{\Sigma},
        \ \|\Bv_0\|_{\rm op}\Dc_{\tilde\Sigma}\,\|\tilde\beta_t-\tilde\beta_0\|_{\Sigma}\Big\}.
    \end{align}

    Partition permutations into
    \begin{equation*}
        \Pi_1=\{\pi:[s]\to[s]: 1\le \pi(s)\le k\},\qquad
        \Pi_2=\{\pi:[s]\to[s]: k+1\le \pi(s)\le s\}.
    \end{equation*}
    Since $\Eb_{\Sc_1}[\gv_\pi(t)]=0$ for all $\pi$, we have
    \begin{align}\label{eq:varofDelta_s}
        {\rm Var}_{\Sc_1}\Big(\Delta^{(s)}(t)[\tv_1\otimes\cdots\otimes\tv_s]\Big)
        &=\sum_{\pi,\pi'\in\Pi_1}\Eb_{\Sc_1}\big[\gv_\pi(t)\gv_{\pi'}(t)\big]
        +2\sum_{\pi\in\Pi_1,\tau\in\Pi_2}\Eb_{\Sc_1}\big[\gv_\pi(t)\gv_\tau(t)\big] \\
        &\quad+\sum_{\tau,\tau'\in\Pi_2}\Eb_{\Sc_1}\big[\gv_\tau(t)\gv_{\tau'}(t)\big]. \nonumber
    \end{align}

    For $\tau\in\Pi_2$, $M_{\tau(s)}=\Sigma-\tilde\Sigma$ and $\gamma_{\tau(s)}=\Gamma-\tilde\Gamma$, so each summand of $\gv_\tau(t)$ is a chain of the form
    \begin{equation*}
        c_1^\top (Z_{j_1}Z_{j_1}^\top-I_d)C_1\cdots C_{k-1}(Z_{j_k}Z_{j_k}^\top-I_d)c_2,
        \qquad \{j_1,\ldots,j_k\}=[k],
    \end{equation*}
    with
    \begin{align}\label{eq:N1}
        \norm{c_1}_2\norm{c_2}_2\prod_{r=1}^{k-1}\norm{C_r}_{\rm op}
        &\leq \max\bigg\{\norm{\tilde\eta}_2\norm{\Bv_t-\Bv_0}_{\rm op}(\norm{\Bv_t}_{\rm op}\vee\norm{\Bv_0}_{\rm op})^{s-1}\norm{\tilde\beta_t-\beta}_{\Sigma}\,
        \Dc_{\tilde\Sigma}^{\,s-k-1}, \nonumber\\
        &\hspace{2.2cm}\norm{\tilde\eta}_2\norm{\Bv_0}_{\rm op}^s\norm{\tilde\beta_t-\tilde\beta_0}_{\Sigma}\,\Dc_{\tilde\Sigma}^{\,s-k}\bigg\}
        =:\Nv_1(t).
    \end{align}

        For $\pi\in\Pi_1$, write $\pi(s)=\iota\in[k]$. Using $\gamma_\iota=M_\iota\beta+X_\iota\varepsilon_\iota$, we decompose the last chain of $\gv_\pi$ as
    \begin{equation*}
        \Sigma^{-1/2}(M_\iota\tilde\beta_t-\gamma_\iota)
        =(Z_\iota Z_\iota^\top-I_d)\Sigma^{1/2}(\tilde\beta_t-\beta)-Z_\iota\varepsilon_\iota.
    \end{equation*}
    Decompose $\gv_\pi(t)=\hv_\pi(t)+\lv_\pi(t)$ accordingly. Then $\hv_\pi(t)$ is a sum of $(s+1)$ terms of the form
    \begin{equation*}
        d_1^\top (Z_{j_1}Z_{j_1}^\top-I_d)D_1\cdots D_{k-1}(Z_{j_k}Z_{j_k}^\top-I_d)d_2,
        \qquad \{j_1,\ldots,j_k\}=[k],
    \end{equation*}
    and $\lv_\pi(t)$ is a sum of $s$ terms of the form
    \begin{equation*}
        e_1^\top (Z_{j_1}Z_{j_1}^\top-I_d)E_1\cdots (Z_{j_{k-1}}Z_{j_{k-1}}^\top-I_d)E_{k-1}(Z_{j_k}\varepsilon_{j_k}),
        \qquad \{j_1,\ldots,j_k\}=[k].
    \end{equation*}
    Moreover, the corresponding deterministic envelopes satisfy
    \begin{align}\label{eq:N2N3}
        &\norm{d_1}_2\norm{d_2}_2\prod_{r=1}^{k-1}\norm{D_r}_{\rm op}\leq \max\bigg\{
        \norm{\tilde\eta}_2\norm{\Bv_t-\Bv_0}_{\rm op}(\norm{\Bv_t}_{\rm op}\vee\norm{\Bv_0}_{\rm op})^{s-1}\norm{\tilde\beta_t-\beta}_{\Sigma}\Dc_{\tilde\Sigma}^{\,s-k},\nonumber\\
        &\hspace{3.35cm}
        \norm{\tilde\eta}_2\norm{\Bv_0}_{\rm op}^{s}\norm{\tilde\beta_t-\tilde\beta_0}_{\Sigma}\Dc_{\tilde\Sigma}^{\,s-k}\bigg\}
        =:\Nv_2(t),\\
        &\norm{e_1}_2\prod_{r=1}^{k-1}\norm{E_r}_{\rm op}
        \leq \norm{\tilde\eta}_2\norm{\Bv_t-\Bv_0}_{\rm op}(\norm{\Bv_t}_{\rm op}\vee\norm{\Bv_0}_{\rm op})^{s-1}\Dc_{\tilde\Sigma}^{\,s-k}
        =:\Nv_3(t).\nonumber
    \end{align} Applying the moment bound in Theorem~\ref{thm:C4} and expanding centered factors in one chain yields
    \begin{align*}
        \Eb_{\Sc_1}\big[\gv_\pi(t)\gv_{\pi'}(t)\big]
        &\le (s+1)^22^k\Big(\Nv_2(t)^2 K_x^{4k} + 2\Nv_2(t)\Nv_3(t)K_x^{4k-1}K_\varepsilon + \Nv_3(t)^2K_x^{4k-2}K_\varepsilon^2\Big)d^{k-1} \\
        &\le s^22^{k+2}K_x^{4k-2}\big(\Nv_2(t)K_x+\Nv_3(t)K_\varepsilon\big)^2 d^{k-1}, \qquad \pi,\pi'\in\Pi_1,\\[0.2cm]
        \Eb_{\Sc_1}\big[\gv_\pi(t)\gv_{\tau}(t)\big]
        &\le (s+1)s2^{k}\Big(\Nv_1(t)\Nv_2(t)K_x^{4k}+\Nv_1(t)\Nv_3(t)K_x^{4k-1}K_\varepsilon\Big)d^{k-1} \\
        &\le s^22^{k+1}K_x^{4k-1}\Nv_1(t)\big(\Nv_2(t)K_x+\Nv_3(t)K_\varepsilon\big)d^{k-1},\qquad \pi\in\Pi_1,\ \tau\in\Pi_2,\\[0.2cm]
        \Eb_{\Sc_1}\big[\gv_\tau(t)\gv_{\tau'}(t)\big]
        &\le s^22^kK_x^{4k}\Nv_1(t)^2 d^{k-1},\qquad \tau,\tau'\in\Pi_2.
    \end{align*}
    Combining these bounds with \eqref{eq:varofDelta_s} yields
    \begin{align}\label{eq:varintermeidate}
        {\rm Var}_{\Sc_1}\Big(\Delta^{(s)}(t)[\tv_1\otimes\cdots\otimes\tv_s]\Big)
        &\le s^2 (s!)^2 2^{k+2}K_x^{4k-2}(K_x\vee K_\varepsilon)^2\big(\Nv_1(t)\vee\Nv_2(t)\vee\Nv_3(t)\big)^2 d^{k-1}.
    \end{align}

    Using $\sup_{t\in[0,1]}\|\Bv_t\|_{\rm op}\le 1\vee \underline\lambda_{\tilde\Sigma}^{-1}$ and $\|\Bv_t-\Bv_0\|_{\rm op}\le t(1\vee\underline\lambda_{\tilde\Sigma}^{-2})\Dc_{\tilde\Sigma}$, and
    \begin{align*}
    \tilde\beta_t - \beta = \tilde\Sigma^{-1}_t\tilde\Gamma_t-\beta=\tilde\Sigma^{-1}_t (\tilde\Gamma_t-\tilde\Sigma_t\beta)=(1-t)\tilde\Sigma_t^{-1}(\tilde\Gamma_0-\tilde\Sigma_0\beta),\\
    \tilde\beta_t - \tilde\beta_0 = (\tilde\beta_t - \beta) - (\tilde\beta_0 - \beta)=\big[(1-t)\tilde\Sigma_t^{-1}-\tilde\Sigma_0^{-1}\big](\tilde\Gamma_0-\tilde\Sigma_0\beta)\\
    =\tilde\Sigma_0^{-1}\big[(1-t)\tilde\Sigma_0-\tilde\Sigma_t\big]\tilde\Sigma_t^{-1}(\tilde\Gamma_0-\tilde\Sigma_0\beta)=-t\tilde\Sigma_0^{-1}\Sigma\tilde\Sigma_t^{-1}(\tilde\Gamma_0-\tilde\Sigma_0\beta).
\end{align*}
    we obtain
    \begin{equation*}
        \|\tilde\beta_t-\beta\|_{\Sigma}\le (1-t)(1\vee\underline\lambda_{\tilde\Sigma}^{-1})\Rc_{\tilde\Sigma,\tilde\Gamma},\qquad
        \|\tilde\beta_t-\tilde\beta_0\|_{\Sigma}\le t(1\vee\underline\lambda_{\tilde\Sigma}^{-2})\Rc_{\tilde\Sigma,\tilde\Gamma}.
    \end{equation*}
    Plugging these bounds into the definitions of $\Nv_1(t),\Nv_2(t),\Nv_3(t)$ yields
    \begin{equation}\label{eq:n1n2n3}
        \Nv_1(t)\vee\Nv_2(t)\vee\Nv_3(t)
        \le \norm{\eta}_{\Sigma^{-1}} (1\vee\underline\lambda_{\tilde\Sigma}^{-s-2})\Dc_{\tilde\Sigma}^{s-k}
        \big(\Dc_{\tilde\Sigma}\vee\Rc_{\tilde\Sigma,\tilde\Gamma}\vee\Dc_{\tilde\Sigma}\Rc_{\tilde\Sigma,\tilde\Gamma}\big)t.
    \end{equation}
    Combining \eqref{eq:B4:ustatvar}--\eqref{eq:n1n2n3}, we obtain
    \begin{align}\label{eq:55}
        &{\rm Var}_{\Sc_1}\Big(\Delta^{(s)}(t)\big[\bar h^{(k)},(\Sigma-\tilde\Sigma,\Gamma-\tilde\Gamma)^{\otimes(s-k)}\big]\Big)\nonumber\\
        &\leq \norm{\eta}_{\Sigma^{-1}}^2 s^2 k!(s!)^2 2^{k+2} K_x^{4k-2}(K_x\vee K_\varepsilon)^2
        \frac{\Dc_{\tilde\Sigma}^{2s-2k}\big(\Dc_{\tilde\Sigma}\vee\Rc_{\tilde\Sigma,\tilde\Gamma}\vee\Dc_{\tilde\Sigma}\Rc_{\tilde\Sigma,\tilde\Gamma}\big)^2}{1\wedge\underline\lambda_{\tilde\Sigma}^{2s+4}}
        \frac{d^{k-1}}{(n)_k}\,t^2.
    \end{align}

    If $1\le k\le s-1$, combining \eqref{eq:B4:var-int} with \eqref{eq:55} and using $(s!)^2/\{(s-k)!\}^2\le s^{2k}$ yields the desired variance bound. If $k=s$, the same bound follows directly from \eqref{eq:55} with $t=1$. Finally, the bound for $\tilde\Delta_0^{(s)}$ follows from \eqref{eq:mean} and the identity $\int_0^1 t(1-t)^{s-1}\,dt/\Gamma(s)=1/(s+1)!$.
\end{proof}

\section{Moments bounds for Monomials of Random Vectors}
This section studies moment bounds for the polynomial of $d$-dimensional vectors. Consider a random vector $Z\in\Real^d$ with $\Eb[ZZ^\top]=\Sigma$.

\begin{enumerate}[label=\textbf{(M)(Z)},leftmargin=2cm]
    \item\label{asmp:momentZ} Suppose there exists some $q_z\geq2$ and a constant $\kappa_z\geq 1$ satisfying that
    \begin{equation*}
        \sup_{u\in\Sb^{d-1}}\frac{\left(\Eb[|u^\top Z|^{q_z}]\right)^{1/q_z}}{(u^\top \Sigma u)^{1/2}}\leq \kappa_z.
    \end{equation*}
\end{enumerate}

We consider the following matrix class,
\begin{equation*}
    \Ac = \set{A\in\Real^{d\times d}:\norm{A}_{\rm op}\leq 1}.
\end{equation*} Denote
\begin{equation*}
    \sigma = \norm{\Sigma}_{\rm op}\quad\mbox{and}\quad \tau = {\rm tr}(\Sigma).
\end{equation*}

\begin{lemma}\label{lem:C1}
    Under Assumption~\ref{asmp:momentZ}, it holds that
    \begin{equation*}
        \sup_{a_1,\ldots,a_M\in\Sb^{d-1}}\sup_{A_1,\ldots,A_N\in\Ac}\Eb\left[\prod_{i=1}^M|a_i^\top Z|^{r_i}\prod_{j=1}^N|Z^\top A_j Z|^{s_j}\right]
        \leq \kappa_z^{\sum_{i=1}^M r_i + 2\sum_{j=1}^Ns_j}\, \sigma^{\frac12\sum_{i=1}^M r_i}\, \tau^{\sum_{j=1}^Ns_j},
    \end{equation*}
    for any $r_i,s_j\geq 0$ and $M,N\in\Nb$ as long as $\sum_{i=1}^M r_i + 2\sum_{j=1}^Ns_j\leq q_z$.
\end{lemma}

\begin{proof}[proof of Lemma~\ref{lem:C1}.]
Let $\bar q = \sum_{i=1}^M r_i + 2\sum_{j=1}^Ns_j\leq q_z$. From H\"older's inequality, Jensen's inequality and monotonicity of $L_p$-norms, we get
\begin{align*}
    &\Eb\left[\prod_{i=1}^M|a_i^\top Z|^{r_i}\prod_{j=1}^N|Z^\top A_j Z|^{s_j}\right]\leq \prod_{i=1}^M\left(\Eb[|a_i^\top Z|^{\bar q}]\right)^{r_i/\bar q}\prod_{j=1}^N\left(\Eb[|Z^\top A_j Z|^{\bar q/2}]\right)^{2s_j/\bar q}\\
    &\leq \prod_{i=1}^M\left(\Eb[|a_i^\top Z|^{q_z}]\right)^{r_i/q_z}\prod_{j=1}^N\left(\Eb[|Z^\top A_j Z|^{q_z/2}]\right)^{2s_j/q_z}.
\end{align*}
By Assumption~\ref{asmp:momentZ} and the bound $a_i^\top\Sigma a_i\le \|\Sigma\|_{\rm op}$,
\begin{equation*}
    \left(\Eb[|a_i^\top Z|^{q_z}]\right)^{1/q_z}\le \kappa_z (a_i^\top\Sigma a_i)^{1/2}\le \kappa_z \|\Sigma\|_{\rm op}^{1/2}.
\end{equation*}
Moreover, since $\|A_j\|_{\rm op}\le 1$, we have $|Z^\top A_j Z|\le \|Z\|_2^2$, so
\begin{equation*}
    \left(\Eb[|Z^\top A_j Z|^{q_z/2}]\right)^{2/q_z}\le \left(\Eb[\|Z\|_2^{q_z}]\right)^{2/q_z}.
\end{equation*}
It remains to bound $\Eb[\|Z\|_2^{q_z}]$. Let $\Sigma=U\Lambda U^\top$ with $\Lambda=\diag(\lambda_1,\ldots,\lambda_d)$ and set $Y=U^\top Z$. Then
\begin{equation*}
    \|Z\|_2^2=\sum_{k=1}^d Y_k^2=\tr(\Sigma)\sum_{k=1}^d w_k\frac{Y_k^2}{\lambda_k},\qquad w_k:=\lambda_k/\tr(\Sigma),\ \sum_{k=1}^d w_k=1.
\end{equation*}
By Jensen's inequality (since $q_z/2\ge 1$),
\begin{align*}
    \Eb[\|Z\|_2^{q_z}]
    &=\tr(\Sigma)^{q_z/2}\Eb\left[\Big(\sum_{k=1}^d w_k\frac{Y_k^2}{\lambda_k}\Big)^{q_z/2}\right]
    \le \tr(\Sigma)^{q_z/2}\sum_{k=1}^d w_k\,\Eb\left[\left|\frac{Y_k}{\sqrt{\lambda_k}}\right|^{q_z}\right].
\end{align*}
Applying Assumption~\ref{asmp:momentZ} to $u=Ue_k\in\Sb^{d-1}$ yields
$\big(\Eb[|Y_k|^{q_z}]\big)^{1/q_z}\le \kappa_z \lambda_k^{1/2}$, hence
$\Eb\big[|Y_k/\sqrt{\lambda_k}|^{q_z}\big]\le \kappa_z^{q_z}$. Therefore
\begin{equation*}
    \Eb[\|Z\|_2^{q_z}]\le \kappa_z^{q_z}\tr(\Sigma)^{q_z/2},
    \qquad\text{so}\qquad
    \left(\Eb[\|Z\|_2^{q_z}]\right)^{2/q_z}\le \kappa_z^{2}\tr(\Sigma).
\end{equation*}
Combining the above displays gives the desired bound.
\end{proof}
\begin{theorem}\label{thm:C2}
    Consider a collection of indices $(i_1,\ldots,i_{m+1})~(m\geq 0)$, which need not be distinct. Let each $Z_i$ obey Assumption~\ref{asmp:momentZ}, and  let $Z_i$ and $Z_{i'}$ be independent if $i\neq i'$. Let  $r=|\set{i_1,\ldots,i_{m+1}}|$, $s_j$ be a non-negative integer, and $S = \sum_{j=1}^{m+1}s_j$. If $\max_{k\in\{i_1,\ldots,i_{m+1}\}}\sum_{j:i_j=k}(2+s_j) \leq q_z$, then
\begin{equation}\label{eq:thm:C2}
    \sup_{\substack{u,v,\theta_1,\ldots,\theta_m\in\Sb^{d-1}\\
    A_1,\ldots,A_m\in\Ac
    }}\Eb\left[u^\top Z_{i_1}Z_{i_1}^\top A_1 Z_{i_2}Z_{i_2}^\top\cdots A_{m}Z_{i_{m+1}}Z_{i_{m+1}}^\top v\prod_{j=1}^{m+1}(\theta_j^\top Z_{i_j})^{s_j}\right]
    \leq \kappa_z^{2m+2+S}\,\sigma^{1+S/2}\,\tau^{m+1-r}.
\end{equation}
\end{theorem}

\begin{proof}[proof of Theorem~\ref{thm:C2}.]
We prove \eqref{eq:thm:C2} using the induction on $r$. Without loss of generality, we let $\set{i_1,\ldots,i_{m+1}}=\set{1,\ldots,r}$. We construct a graph with $r$ nodes where $j$th node corresponds to index set $N_j:=\set{k:i_k=j}$ for $j=1,\ldots,r$. We connect two nodes $N_j$ and $N_{j'}$ whenever $k\in N_j$ and $k+1\in N_{j'}$ for some $1\leq k \leq m$, i.e. $Z_j^\top A_* Z_{j'}$ appears in the product. A total number of connections is $m$ and there are two types of connections:
\begin{itemize}
    \item If $j=j'$, we call this line a \emph{self loop}, corresponding to the term $Z_j^\top A Z_j$ in the product.
    \item If $j\neq j'$, we call this line an \emph{edge}, corresponding to the term $Z_j^\top A Z_{j'}$ in the product.
\end{itemize}
We let $l_j$ be the number of self loops at $N_j$ and let $e_j$ be the number of edges whose one end is $N_j$. Since there are exactly $m$ connections, the following identity holds:
\begin{equation*}
    m = \sum_{j=1}^r l_j + \frac{1}{2}\sum_{j=1}^re_j.
\end{equation*}
Denote this graph by $\Gc$ and define
\begin{align*}
    f(\Gc) &= (u^\top Z_{i_1})(Z_{i_1}^\top A_1 Z_{i_2})\cdots(Z_{i_m}^\top A_m Z_{i_{m+1}})(Z_{i_{m+1}}^\top v)\prod_{j=1}^{m+1}(\theta_j^\top Z_{i_j})^{s_j},\\
    \norm{\Gc}_f^\dagger &= \sup_{\substack{u,v,\theta_1,\ldots,\theta_m\in\Sb^{d-1}\\
    A_1,\ldots,A_m\in\Ac
    }}\Eb f(\Gc),\quad
    \norm{\Gc}_f = \sup_{\substack{u,v,\theta_1,\ldots,\theta_m\in\Sb^{d-1}\\
    A_1,\ldots,A_m\in\Ac
    }} \Eb |f(\Gc)|.
\end{align*}
It is straightforward that $\norm{\Gc}_f^\dagger\leq \norm{\Gc}_f$, and it suffices to show
\begin{equation*}
    \norm{\Gc}_f^\dagger\leq \kappa_z^{2m+2+S}\,\sigma^{1+S/2}\,\tau^{m+1-r}.
\end{equation*}

\medskip
First of all, if $r=1$, then the result follows immediately from Lemma~\ref{lem:C1} (with $R=2+S$ and $S=m$ there). Let $r=2$, Theorem~\ref{thm:C.1.4} applies and yields
\begin{equation*}
    \norm{\Gc}_f\leq \kappa_z^{2m+2+S}\,\sigma^{1+S/2}\,\tau^{m-\frac{1}{2}e_1}.
\end{equation*}
If $e_1\geq 2$, then we are done. If $e_1=1$, then $f(\Gc)$ is the form of
\begin{align*}
f(\Gc)=(u^\top Z_{i_1})(Z_1^\top A_1 Z_1)\cdots (Z_1^\top A_{l_1} Z_1)\times (Z_1^\top A_{l_1+1}Z_2)\\
\times ~(Z_2^\top A_{l_1+2} Z_2)\cdots (Z_2^\top A_{l_1+l_2+1} Z_2)(Z_{i_{m+1}}^\top v)\times \prod_{j=1}^{m+1}(\theta_j^\top Z_{i_j})^{s_j},
\end{align*}
where $l_1+l_2+1=m$. Let
\begin{align*}
    c_1 &= \Eb(u^\top Z_{i_1})^{\mathbbm{1}(i_1=1)}(v^\top Z_{i_{m+1}})^{\mathbbm{1}(i_{m+1}=1)}(Z_1^\top A_1 Z_1)\cdots (Z_1^\top A_{l_1} Z_1)Z_1\prod_{j:i_j=1}(\theta_j^\top Z_{i_j})^{s_j},\\
    c_2 &= \Eb (u^\top Z_{i_1})^{\mathbbm{1}(i_1=2)}(v^\top Z_{i_{m+1}})^{\mathbbm{1}(i_{m+1}=2)}Z_2(Z_2^\top A_{l_1+2} Z_2)\cdots (Z_2^\top A_{l_1+l_2+1} Z_2)\prod_{j:i_j=2}(\theta_j^\top Z_{i_j})^{s_j}.
\end{align*}
Then, $|\Eb f(\Gc)|=|c_1^\top A_{l_1+1}c_2|\leq \norm{c_1}_2\norm{c_2}_2$, and an application of Lemma~\ref{lem:C1} yields
\begin{equation*}
\Abs{\Eb f(\Gc)}\leq \kappa_z^{2l_1+2l_2+4+S}\,\sigma^{1+S/2}\,\tau^{l_1+l_2}
=\kappa_z^{2m+2+S}\,\sigma^{1+S/2}\,\tau^{m-1}.
\end{equation*}

In the following, we assume $r\geq 3$. We consider two cases (1) $i_{m+1}=i_1$ and (2) $i_{m+1}\neq i_1$.

\textbf{(1) $i_{m+1}=i_1$}
We prove
\begin{equation}\label{eq:thmA1:1}
    \Eb(u^\top Z_{i_1})(Z_{i_1}^\top A_1 Z_{i_2})\cdots(Z_{i_m}^\top A_m Z_{i_1})(Z_{i_1}^\top v)\prod_{j=1}^{m+1}(\theta_j^\top Z_{i_j})^{s_j}
    \leq \kappa_z^{2m+2+S}\,\sigma^{1+S/2}\,\tau^{m+1-r}.
\end{equation}
It follows from the construction that all $e_j$'s are even and $e_j\geq2$. If $e_j\geq4$ for all $j$, then an application of Theorem~\ref{thm:C.1.4} yields
\begin{equation*}
\norm{\Gc}_f\leq \kappa_z^{2m+2+S}\,\sigma^{1+S/2}\,\tau^{m-\frac{1}{4}\sum_{j=1}^re_j}\leq \kappa_z^{2m+2+S}\,\sigma^{1+S/2}\,\tau^{m-r}.
\end{equation*}
Hence, we suppose $e_1 = 2$ without loss of generality. If $i_1=1$, then $f(\Gc)$ is the form of
\begin{equation*}
    f(\Gc)=(u^\top Z_1)(Z_1^\top A_1 Z_1)\cdots(Z_1^\top A_{l_1} Z_1)(Z_1^\top MZ_1)(Z_1^\top v)\prod_{j:i_j=1}(\theta_j^\top Z_{i_j})^{s_j},
\end{equation*}
where $M = A_* Z_{i}(\cdots)Z_{i'}^\top A_*\prod_{j:i_j\neq 1}(\theta_j^\top Z_{i_j})^{s_j}$ and $i\neq 1$ and $i'\neq 1$. Meanwhile, let $\Gc^{-N_1}$ denote the graph after removing node $N_1$ from $\Gc$ (removing all self loops and edges involving node $N_1$). Then, $\Gc^{-N_1}$ has $r-1$ nodes and from the induction hypothesis, we get
\begin{align*}
    &\sup_{u,v\in\Sb^{d-1}}u^\top \Eb[M] v
    \leq \norm{\Gc^{-N_1}}_f^\dagger\\
    &\quad\leq \kappa_z^{2(m-2-l_1)+2+\sum_{j:i_j\neq 1}s_j}\,\sigma^{1+\frac12\sum_{j:i_j\neq 1}s_j}\,
    \tau^{(m-2-l_1)-(r-1)+1}\\
    &\qquad=\kappa_z^{2m-2-2l_1+\sum_{j:i_j\neq 1}s_j}\,\sigma^{1+\frac12\sum_{j:i_j\neq 1}s_j}\,\tau^{m-r-l_1}.
\end{align*}
Moreover, it follows from the law of iterated expectation and Lemma~\ref{lem:C1} that
\begin{align*}
    \norm{\Gc}_f^\dagger
    \leq \kappa_z^{2l_1+4+\sum_{j:i_j=1}s_j}\,\sigma^{1+\frac12\sum_{j:i_j=1}s_j}\,\tau^{l_1+1}
    \times
    \kappa_z^{2m-2-2l_1+\sum_{j:i_j\neq 1}s_j}\,\sigma^{1+\frac12\sum_{j:i_j\neq 1}s_j}\,\tau^{m-r-l_1}\\
    = \kappa_z^{2m+2+S}\,\sigma^{1+S/2}\,\tau^{m-r+1}.
\end{align*}
Now suppose that $i_1\neq 1$, and $N_1$ is connected to exactly two nodes $N_{k_1}$ and $N_{k_2}$ where $k_1$ and $k_2$ are not necessarily different. Let $\Gc^{-N_1}$ denote the graph after removing the node $N_1$ from $\Gc$. Then $f(\Gc)$ is of the form:
\begin{align*}
    f(\Gc)&= (Z_{k_1}^\top A_{0} Z_1)(Z_1^\top A_{1} Z_1)\cdots(Z_1^\top A_{{l_1}}Z_1)(Z_1^\top A_{{l_1+1}} Z_{k_2})\prod_{j:i_j=1}(\theta_j^\top Z_{i_j})^{s_j} f(\Gc^{-N_1})\\
    &\quad =(Z_{k_1}^\top M Z_{k_2})f(\Gc^{-N_1}),
\end{align*}
where $M = (A_{0} Z_1)(Z_1^\top A_{1} Z_1)\cdots(Z_1^\top A_{{l_1}}Z_1)(Z_1^\top A_{{l_1+1}} )\prod_{j:i_j=1}(\theta_j^\top Z_{i_j})^{s_j}.$
From Lemma~\ref{lem:C1}, we have
\begin{equation*}
 \sup_{u,v\in\Sb^{d-1}}u^\top \Eb[M] v \leq \kappa_z^{2l_1+2+\sum_{j:i_j=1}s_j}\,\sigma^{1+\frac12\sum_{j:i_j=1}s_j}\,\tau^{l_1}.
\end{equation*}
Consequently, the law of total expectation implies that
\begin{equation}\label{eq:thmA1:2}
 |\Eb f(\Gc)|\leq \kappa_z^{2l_1+2+\sum_{j:i_j=1}s_j}\,\sigma^{1+\frac12\sum_{j:i_j=1}s_j}\,\tau^{l_1} \left|\Eb\left[\left (Z_{k_1}^\top\frac{\Eb M}{\norm{\Eb M}_{\rm op}}Z_{k_2}\right)f(\Gc^{-N_1})\right]\right|.
\end{equation}
We consider a graph $\tilde\Gc$ by adding an edge connecting $Z_{k_1}$ and $Z_{k_2}$ on $\Gc^{-N_1}$. It is easy to verify that $\tilde\Gc$ has $r-1$ nodes and from the induction hypothesis, we get
\begin{equation*}
 \norm{\tilde \Gc}^\dagger \leq \kappa_z^{2(m-1-l_1)+2+\sum_{j:i_j\neq 1}s_j}\,\sigma^{1+\frac12\sum_{j:i_j\neq 1}s_j}\,
 \tau^{(m-1-l_1)-(r-1)+1}.
\end{equation*}
Combining this with \eqref{eq:thmA1:2} proves \eqref{eq:thmA1:1}.

\textbf{(2) $i_{m+1}\neq i_1$}
Since $i_1\neq i_{m+1}$, there exist exactly two nodes in $\Gc$ that have an odd number of edges. In particular, two nodes correspond to $i_1$ and $i_{m+1}$, respectively. Without loss of generality, let $N_1$ and $N_2$ be such nodes. If $e_1+e_2\geq 4$ and $e_j\geq4$ for all $j\neq 1,2$, then we have $\sum_{j=1}^r e_j\geq 4+4(r-2)=4r-4$. An application of Theorem~\ref{thm:C.1.4}  results in
\begin{equation*}
\norm{\Gc}_f\leq \kappa_z^{2m+2+S}\,\sigma^{1+S/2}\,\tau^{m-\frac{1}{4}\sum_{j=1}^re_j}\leq \kappa_z^{2m+2+S}\,\sigma^{1+S/2}\,\tau^{m-r+1}.
\end{equation*}
Therefore, we only need to complement two scenarios; (i) there exists $j\neq 1,2$ such that $e_j=2$ or (ii) $e_1=e_2=1$. Regarding (i), the proof is in a similar vein to that of the case $i_1=i_{m+1}$. Suppose that $e_3 = 2$ without loss of generality, and let $N_{k_1}$ and $N_{k_2}$ be two nodes connected with $N_3$. The law of conditional expectation leads to
\begin{equation}\label{eq:thmA1:3}
 \Eb f(\Gc) =\Eb\left[(Z_{k_1}^\top \Eb[M'] Z_{k_2})f(\Gc^{-N_3})\right],
\end{equation}
where $\Gc^{-N_3}$ denotes the graph after removing the node $N_3$ from $\Gc$, and
\begin{equation*}
M' = (A_{\iota_0} Z_3)(Z_3^\top A_{\iota_1} Z_3)\cdots(Z_3^\top A_{\iota_{l_3}}Z_3)(Z_3^\top A_{\iota_{l_3+1}} )\prod_{j:i_j=3}(\theta_j^\top Z_3)^{s_j}.
\end{equation*}
From Lemma~\ref{lem:C1}, we have
\begin{equation*}
\|\Eb M'\|_{\rm op}\leq \kappa_z^{2l_3+2+\sum_{j:i_j=3}s_j}\,\sigma^{1+\frac12\sum_{j:i_j=3}s_j}\,\tau^{l_3}.
\end{equation*}
Then, we construct a graph $\tilde\Gc$ by adding an edge connecting $Z_{k_1}$ and $Z_{k_2}$ on $\Gc^{-N_3}$. Then, $\tilde\Gc$ has $r-1$ nodes and the induction hypothesis applies and yields
\begin{equation*}
 \Eb f(\tilde \Gc) \leq \kappa_z^{2(m-1-l_3)+2+\sum_{j:i_j\neq 3}s_j}\,\sigma^{1+\frac12\sum_{j:i_j\neq 3}s_j}\,
 d^{(m-1-l_3)-(r-1)+1}.
\end{equation*}
Combining this with \eqref{eq:thmA1:3} completes case (i).

\medskip
In case (ii) $e_1=e_2=1$, $f(\Gc)$ is a form of
\begin{align*}
    f(\Gc)=(u^\top Z_1)(Z_1^\top A_1 Z_1)\cdots (Z_1^\top A_{l_1} Z_1)\times(Z_1^\top M'' Z_2)\\
    \times (Z_2^\top A_{l_1+3} Z_2)\cdots (Z_2^\top A_{l_1+l_2+2} Z_2) (Z_2^\top v)\times \prod_{j:i_j\in\set{1,2}}(\theta_j^\top Z_{i_j})^{s_j},
\end{align*}
where $M'' = A_{l_1+1}Z_{k_1} (Z_{k_1}^\top A_*Z_{k_1}) \cdots (Z_{k_2}^\top A_* Z_{k_2}) Z_{k_2}^\top A_{l_1+2}\prod_{j:i_j\notin\set{1,2}}(\theta_j^\top Z_{i_j})^{s_j}$. Let $\Gc^{-{N_1,N_2}}$ denotes the graph after removing $N_1$ and $N_2$ from $\Gc$. Then, $u^\top M'' v$ is the form of $f(\Gc^{-{N_1,N_2}})$. The graph $\Gc^{-{N_1,N_2}}$ has $r-2$ nodes and the induction hypothesis implies that
\begin{equation}\label{eq:thmA1:4}
    \sup_{u,v\in\Sb^{d-1}}u^\top \Eb[M''] v\leq \kappa_z^{2(m-l_1-l_2-2)+2+\sum_{j:i_j\notin\set{1,2}}s_j}\,\sigma^{1+\frac12\sum_{j:i_j\notin\set{1,2}}s_j}\,
    \tau^{m-l_1-l_2-2-(r-2)+1}.
\end{equation}
The law of total expectation implies that
\begin{align*}
    \Abs{\Eb[f(\Gc)]}&= \norm{\Eb M''}_{\rm op}\bigg|\Eb\bigg[(u^\top Z_1)(Z_1^\top A_1 Z_1)\cdots (Z_1^\top A_{l_1} Z_1)\times(Z_1^\top \frac{\Eb M''}{\norm{\Eb M''}_{\rm op}} Z_2)\\
    &\qquad\qquad\times (Z_2^\top A_{l_1+3} Z_2)\cdots (Z_2^\top A_{l_1+l_2+2} Z_2) (Z_2^\top v)\prod_{j:i_j\in\set{1,2}}(\theta_j^\top Z_{i_j})^{s_j}\bigg]\bigg|\\
    &\leq \kappa_z^{2m+2+S}\,\sigma^{1+S/2}\,\tau^{m-r+1},
\end{align*}
where the last inequality follows from \eqref{eq:thmA1:4} and the induction hypothesis of $r=2$. This completes the proof.
\end{proof}
\begin{theorem}\label{thm:C3}
    Consider two collections of indices $(i_1,\ldots,i_{m_1+1})$ and $(j_1,\ldots,j_{m_2+1})~(m_1,m_2\geq 0)$, such that $r = |\set{i_1,\ldots,i_{m_1+1},j_1,\ldots,j_{m_2+1}}|$ and $|\set{i_1,\ldots,i_{m_1+1}}\cap \set{j_1,\ldots,j_{m_2+1}}|\geq 1$. Let each $Z_i$ obey Assumption~\ref{asmp:momentZ}, and  let $Z_i$ and $Z_{i'}$ be independent if $i\neq i'$. If
\begin{equation*}
\max_{k\in\{i_1,\ldots,i_{m_1+1},j_1,\ldots,j_{m_2+1}\}}|\{j:i_j=k\}| \leq q_z/2,
\end{equation*}
then
\begin{align}\label{eq:thm:C3}
    &\sup_{\substack{u_1,v_1,u_2,v_2\in\Sb^{d-1}\\
    A_1,\ldots,A_{m_1},B_1,\ldots,B_{m_2}\in\Ac
    }}\Eb\bigg[u_1^\top Z_{i_1}Z_{i_1}^\top A_1 Z_{i_2}Z_{i_2}^\top\cdots A_{m_1}Z_{i_{m_1+1}}Z_{i_{m_1+1}}^\top u_2\nonumber\\
    &\quad\times~ v_1^\top Z_{j_1}Z_{j_1}^\top B_1 Z_{j_2}Z_{j_2}^\top\cdots B_{m_2}Z_{j_{m_2+1}}Z_{j_{m_2+1}}^\top v_2 \bigg]\leq \kappa_z^{2(m_1+m_2)+4}\,\sigma^2\,\tau^{(m_1+m_2)+1-r}.
\end{align}
\end{theorem}

\begin{proof}
    If $|\set{i_1,i_{m_1+1}}\cap\set{j_1,j_{m_2+1}}|\geq 1$, then two `chain products';
    \begin{align*}
        \mbox{Chain 1}&: (Z_{i_1}^\top A_1 Z_{i_2})\cdots (Z_{i_{m_1}}^\top A_{m_1}Z_{i_{m_1+1}}),\\
        \mbox{Chain 2}&: (Z_{j_1}^\top B_1 Z_{j_2})\cdots (Z_{j_{m_2}}^\top B_{m_2}Z_{j_{m_2+1}}),
    \end{align*}
    can be attached by their ends to one long chain product of $(m_1+m_2)$ factors of $(ZAZ')$. Hence, Theorem~\ref{thm:C2} applies and yields the result.
    Let $k\in \set{i_1,\ldots,i_{m_1+1}}\cap\set{j_1,\ldots,j_{m_2+1}}$ be a common index. If $i_1=i_{m_1+1}$,  then two chains again can be connected into a long chain. To see this, let the sub-chains be as follows.
    \begin{align*}
        \mbox{Chain 1}: \overbrace{(Z_{i_1}^\top A_1 Z_{i_2})\cdots (Z'A_* Z_k)}^{\mbox{Chain 1a}}&\overbrace{(Z_kA_*Z'')\cdots(Z_{i_{m_1}}A_{m_1}Z_{i_1})}^{\mbox{Chain 1b}}\\
        \mbox{Chain 2}: \underbrace{(Z_{j_1}^\top B_1 Z_{j_2})\cdots (Z'''B_* Z_k)}_{\mbox{Chain 2a}}&\underbrace{(Z_kB_*Z'''')\cdots(Z_{j_{m_2}}B_{m_1}Z_{j_{m_2+1}})}_{\mbox{Chain 2b}}.
    \end{align*}
    Then the products of chains can be expressed as
    \begin{align*}
    \mbox{Chain 1}\times \mbox{Chain 2}:\overbrace{(Z_{j_1}^\top B_1 Z_{j_2})\cdots (Z'''^\top B_* Z_k)}^{\mbox{Chain 2a}}&\overbrace{(Z_k^\top A_* Z')\cdots(Z_{i_2}^\top A_1 Z_{i_1}) }^{\mbox{Reversed Chain 1a}}\\
    \times \underbrace{(Z_{i_1}^\top A_{m_1}Z_{i_{m_1}})\cdots(Z''^\top A_*Z_k)}_{\mbox{Reversed Chain 1b}}&\underbrace{(Z_k^\top B_*Z'''')\cdots(Z_{j_{m_2}}^\top B_{m_1}Z_{j_{m_2+1}})}_{\mbox{Chain 2b}}.
    \end{align*}
    Therefore, Theorem~\ref{thm:C2} applies and yields the desired results. We can make a similar argument for $j_1= j_{m_2+1}$. Therefore, we may only consider the case where $\set{i_1,i_{m_1+1},j_1,j_{m_2+1}}$ are all distinct.

    \medskip
    We prove \eqref{eq:thm:C3} via induction on $r$. The logic above proves the case where $1\leq r\leq 3$. Suppose $r\geq 4$ and, without loss of generality, let $\set{i_1,\ldots,i_{m_1+1},j_1,\ldots,j_{m_2+1}}=\set{1,\ldots,r}$, and let  $i_1=1,i_{m_1+1}=2,j_1=3$, and $j_{m_2+1}=4$. We adopt the notations in the proof of Theorem~\ref{thm:C2}, and let a graph $\Gc$ correspond to the product of two chains. From the construction, $e_j$'s are odd for $j=1,2,3,4$, and even otherwise. Define
    \begin{align*}
        &\quad\quad g(\Gc) = u_1^\top Z_{i_1}Z_{i_1}^\top A_1 Z_{i_2}Z_{i_2}^\top\cdots A_{m_1}Z_{i_{m_1+1}}Z_{i_{m_1+1}}^\top u_2\\
        &\quad\quad \quad\quad\quad\quad \times v_1^\top Z_{j_1}Z_{j_1}^\top B_1 Z_{j_2}Z_{j_2}^\top\cdots B_{m_2}Z_{j_{m_2+1}}Z_{j_{m_2+1}}^\top v_2,\\
        &\norm{\Gc}_g^\dagger = \sup_{\substack{u_1,v_1,u_2,v_2\in\Sb^{d-1}\\A_1,\ldots,A_{m_1},B_1,\ldots,B_{m_2}\in\Ac
    }}\Eb g(\Gc),\quad \norm{\Gc}_g = \sup_{\substack{u_1,v_1,u_2,v_2\in\Sb^{d-1}\\A_1,\ldots,A_{m_1},B_1,\ldots,B_{m_2}\in\Ac
    }} \Eb |g(\Gc)|.
    \end{align*}
    Theorem~\ref{thm:C.1.4} implies that
    \begin{equation}\label{eq:thmA2:1}
        \norm{\Gc}_g \leq \kappa_z^{2(m_1+m_2)+4}\,\sigma^2\,\tau^{(m_1+m_2)-\frac{1}{4}\sum_{i=1}^r e_i}.
    \end{equation}
    If $e_1,e_2,e_3,e_4\geq 3$ and $e_j\geq 4$ for $j\neq 1,2,3,4$, then $\frac{1}{4}\sum_{j=1}^r e_j\geq 4r-4$, and therefore \eqref{eq:thmA2:1} proves \eqref{eq:thm:C3}.
    We now consider two cases (i) $e_j=1$ for some $j=1,2,3,4$ and (ii) $e_j=2$ for some $j\neq 1,2,3,4$.

    For (i), without loss of generality, let $e_1=1$. From the law of iterated expectation, $\Eb g(\Gc)$ is the form of
    \begin{align*}
        \Eb g(\Gc) &= \Eb[\cv_1^\top Z_iZ_i^\top\cdots A_{m_1}Z_{i_{m_1+1}}Z_{i_{m_1+1}}^\top u_2\\
        &\quad\quad\quad\quad \times v_1^\top Z_{j_1}Z_{j_1}^\top B_1 Z_{j_2}Z_{j_2}^\top\cdots B_{m_2}Z_{j_{m_2+1}}Z_{j_{m_2+1}}^\top v_2]\\
        &=\norm{\cv_1}_2\Eb g(\Gc^{-N_1}),
    \end{align*}
    for some $i\neq 1$ and  $\cv_1^\top = \Eb[u_1^\top Z_1(Z_1^\top A_{1}Z_1)\cdots (Z_1^\top A_{l_1}Z_1)Z_1^\top A_{l_1+1}]$. Lemma~\ref{lem:C1} applies and yields $\norm{\cv_1}_2\leq \kappa_z^{2l_1+2}\,\sigma\,\tau^{l_1}$, and the induction hypothesis for $\Gc^{-N_1}$ with $r-1$ nodes gives
    \begin{equation*}
        \norm{\Gc}_g^\dagger\leq \kappa_z^{2l_1+2}\,\sigma\,\tau^{l_1} \norm{\Gc^{-N_1}}_g^\dagger\leq \kappa_z^{2(m_1+m_2)+4}\,\sigma^2\,\tau^{m_1+m_2+1-r}.
    \end{equation*}

    Now, suppose $e_j=2$ for some $j\neq 1,2,3,4$, and, without loss of generality, let $e_5=2$ and let the factor $Z_5Z_5^\top$ belong to Chain 1. Let $N_{k_1},N_{k_2}$ be two nodes connected to $N_5$, then $\Eb g(\Gc)$ is the form of
    \begin{align*}
        \Eb g(\Gc) &= \Eb[u_1^\top Z_1Z_1^\top\cdots Z_{k_1}^\top MZ_{k_2}\cdots A_{m_1}Z_{i_{m_1+1}}Z_{i_{m_1+1}}^\top u_2\\
        &\quad\quad\quad\quad \times v_1^\top Z_{j_1}Z_{j_1}^\top B_1 Z_{j_2}Z_{j_2}^\top\cdots B_{m_2}Z_{j_{m_2+1}}Z_{j_{m_2+1}}^\top v_2]\\
        &=\norm{M}_{\rm op}\Eb g(\tilde \Gc^{-N_5}),
    \end{align*}
    where $\tilde \Gc^{-N_5}$ is a graph constructed by removing node $N_5$ and corresponding edges and self loops, and adding a new edge (if $k_1\neq k_2$) or a self loop (if $k_1=k_2$) connecting $N_{k_1}$ and $N_{k_2}$ (corresponding to $Z_{k_1}^\top MZ_{k_2}/\norm{M}_{\rm op}$), and
    \begin{equation*}
        M = \Eb[A_{\iota_0}Z_5 (Z_5^\top A_{\iota_1}Z_5)\cdots(Z_5^\top A_{\iota_{l_5}}Z_5)Z_5^\top A_{\iota_{l_5+1}}].
    \end{equation*}
    Lemma~\ref{lem:C1} gives $\norm{M}_{\rm op}\leq \kappa_z^{2l_5+2}\,\tau^{l_5}$, and the induction hypothesis applies to $\tilde \Gc^{-N_5}$, who has $r-1$ nodes, and gives
    \begin{equation*}
        \norm{\Gc}_g^\dagger\leq \norm{M}_{\rm op}\norm{\tilde \Gc^{-N_5}}_g^\dagger\leq \kappa_z^{2(m_1+m_2)+4}\,\sigma^2\,\tau^{m_1+m_2+1-r}.
    \end{equation*}
    This proves the Theorem.
\end{proof}

We consider the following assumption.
\begin{enumerate}[label=\textbf{(M)($\varepsilon$)},leftmargin=2cm]
    \item\label{asmp:momentvare} Suppose there exists some $q_\varepsilon\geq2$ and a constant $K_\varepsilon\geq 1$ satisfying that
    \begin{equation*}
        \left(\Eb[|\varepsilon|^{q_\varepsilon}]\right)^{1/q_\varepsilon}\leq K_\varepsilon.
    \end{equation*}
\end{enumerate}

The following Theorem~\ref{thm:C4} mirrors Theorems~\ref{thm:C2} and \ref{thm:C3}, providing moment bounds for monomials in $(Z,\varepsilon)$. In all places where the proofs of Theorems~\ref{thm:C2}, \ref{thm:C3}, and \ref{thm:C.1.4} invoke moment bounds for univariate monomials of the standardized covariates via Lemma~\ref{lem:C1}, it suffices instead to apply Lemma~\ref{lem:C5} below. The graph-based arguments themselves are unchanged; only the moment conditions depend on the H\"older constraint in Lemma~\ref{lem:C5}. For this reason, we only prove Lemma~\ref{lem:C5} and omit the proof of Theorem~\ref{thm:C4} for brevity.

\begin{theorem}\label{thm:C4}
    Let each $Z_i$ and $\varepsilon_i$ obey Assumption~\ref{asmp:momentZ} and Assumption~\ref{asmp:momentvare}, respectively.
    Consider a collection of indices $(i_1,\ldots,i_{m+1})~(m\geq 0)$, which need not be distinct.  Let $(Z_i,\varepsilon_i)$ and $(Z_{i'},\varepsilon_{i'})$ be independent if $i\neq i'$. Let  $r=|\set{i_1,\ldots,i_{m+1}}|$, $s_j$ be a non-negative integer, and $S = \sum_{j=1}^{m+1}s_j$. Without loss of generality, let $\set{i_1,\ldots,i_{m+1}}=\{1,\ldots,r\}$ and $i_{m+1}=r$. If
    \begin{equation*}
        \max_{k=1,\ldots,r-1}\sum_{j:i_j=k}(2+s_j)/q_z\leq 1\quad\mbox{and}\quad\sum_{\substack{j:i_j=r\\ j\neq m+1}}(2+s_j)/q_z +(1+s_{m+1})/q_z+1/q_\varepsilon\leq 1,
    \end{equation*}
    then
\begin{align}\label{eq:thm:B4a}
    \sup_{\substack{u,\theta_1,\ldots,\theta_m\in\Sb^{d-1}\\
    A_1,\ldots,A_m\in\Ac
    }}\Eb\left[u^\top Z_{i_1}Z_{i_1}^\top A_1 Z_{i_2}Z_{i_2}^\top\cdots A_{m}Z_{i_{m+1}}\varepsilon_{i_{m+1}}\prod_{j=1}^{m+1}(\theta_j^\top Z_{i_j})^{s_j}\right]\nonumber\\
    \leq \kappa_z^{2m+1+S}K_\varepsilon\,\sigma^{(1+S)/2}\,\tau^{m+1-r},
\end{align}
where $\sigma=\|\Sigma\|_{\rm op}$ and $\tau={\rm tr}(\Sigma)$ are associated with $\Sigma=\Eb[ZZ^\top]$ in Assumption~\ref{asmp:momentZ}.

Moreover, consider two collections of indices $(i_1,\ldots,i_{m_1+1})$ and $(j_1,\ldots,j_{m_2+1})~(m_1,m_2\geq 0)$, such that $r = |\set{i_1,\ldots,i_{m_1+1},j_1,\ldots,j_{m_2+1}}|$ and $|\set{i_1,\ldots,i_{m_1+1}}\cap \set{j_1,\ldots,j_{m_2+1}}|\geq 1$. Then
\begin{align}\label{eq:thm:B4b}
    &\sup_{\substack{u,v\in\Sb^{d-1}\\
    A_1,\ldots,A_{m_1},B_1,\ldots,B_{m_2}\in\Ac
    }}\Eb\bigg[u^\top Z_{i_1}Z_{i_1}^\top A_1 Z_{i_2}Z_{i_2}^\top\cdots A_{m_1}Z_{i_{m_1+1}}\varepsilon_{i_{m_1+1}}\nonumber\\
    &\quad\times~ v^\top Z_{j_1}Z_{j_1}^\top B_1 Z_{j_2}Z_{j_2}^\top\cdots B_{m_2}Z_{j_{m_2+1}}\varepsilon_{j_{m_2+1}} \bigg]\leq \kappa_z^{2(m_1+m_2)+2}K_\varepsilon^{2}\,\sigma\,\tau^{(m_1+m_2)+1-r}.
\end{align}
\end{theorem}

\begin{lemma}\label{lem:C5}
Let $Z$ and $\varepsilon$ be not necessarily independent and obey Assumption~\ref{asmp:momentZ} and \ref{asmp:momentvare}, respectively. Then we have
\begin{align*}
    &\sup_{a_1,\ldots,a_M\in\Sb^{d-1}}\sup_{A_1,\ldots,A_N\in\Ac}\Eb\left[|\varepsilon|^t\prod_{i=1}^M|a_i^\top Z|^{r_i}\prod_{j=1}^N|Z^\top A_j Z|^{s_j}\right]\\
    &\qquad\leq K_\varepsilon^t\,\kappa_z^{\sum_{i=1}^M r_i + 2\sum_{j=1}^Ns_j}\,\sigma^{\frac12\sum_{i=1}^M r_i}\,\tau^{\sum_{j=1}^Ns_j},
\end{align*}
for any $t,r_i,s_j\geq 0$ and $M,N\in\Nb$ as long as $(\sum_{i=1}^M r_i + 2\sum_{j=1}^Ns_j)/q_z+t/q_\varepsilon\leq 1$.
\end{lemma}

\begin{proof}[proof of Lemma~\ref{lem:C5}.]
Let $R = \sum_{i=1}^M r_i$ and $S = \sum_{j=1}^Ns_j$. Then
\begin{equation*}
    \bar q: =\bigg(\frac{R}{q_z} + \frac{S}{q_z/2}+\frac{t}{q_\varepsilon}\bigg)^{-1}\geq 1.
\end{equation*}
Therefore, it follows from H\"older's inequality that
\begin{align*}
    &\Eb\left[|\varepsilon|^t\prod_{i=1}^M|a_i^\top Z|^{r_i}\prod_{j=1}^N|Z^\top A_j Z|^{s_j}\right]\\
    &\quad \leq \bigg(\Eb\big[|\varepsilon|^{q_\varepsilon/\bar q}\big]\bigg)^{t\bar q/q_\varepsilon}\prod_{i=1}^M\bigg(\Eb\big[|a_i^\top Z|^{q_z/\bar q}\big]\bigg)^{r_i\bar q/q_z}\prod_{j=1}^N\bigg(\Eb\big[|Z^\top A_j Z|^{q_z/(2\bar q)}\big]\bigg)^{2s_j\bar q/q_z}\\
    &\quad \leq \bigg(\Eb\big[|\varepsilon|^{q_\varepsilon}\big]\bigg)^{t/q_\varepsilon}\prod_{i=1}^M\bigg(\Eb\big[|a_i^\top Z|^{q_z}\big]\bigg)^{r_i/q_z}\prod_{j=1}^N\bigg(\Eb\big[|Z^\top A_j Z|^{q_z/2}\big]\bigg)^{2s_j/q_z}.
\end{align*}
By Assumptions~\ref{asmp:momentZ}--\ref{asmp:momentvare} and $a_i^\top\Sigma a_i\le \sigma$, we have
\begin{equation*}
    \big(\Eb|\varepsilon|^{q_\varepsilon}\big)^{1/q_\varepsilon}\le K_\varepsilon,\qquad
    \big(\Eb|a_i^\top Z|^{q_z}\big)^{1/q_z}\le \kappa_z(a_i^\top\Sigma a_i)^{1/2}\le \kappa_z\sigma^{1/2}.
\end{equation*}
Moreover, since $\|A_j\|_{\rm op}\le 1$, we have $|Z^\top A_j Z|\le \|Z\|_2^2$, and hence
\begin{equation*}
    \left(\Eb\big[|Z^\top A_j Z|^{q_z/2}\big]\right)^{2/q_z}\le \left(\Eb\big[\|Z\|_2^{q_z}\big]\right)^{2/q_z}\le \kappa_z^2\,\tau,
\end{equation*}
where the last inequality follows from the same Jensen argument as in the proof of Lemma~\ref{lem:C1}. Combining the above displays yields the claim.
\end{proof}

\subsection{Auxiliary Results}
Suppose that each $Z_i$ obeys Assumption~\ref{asmp:momentZ} and let $Z_i$ and $Z_{i'}$ be independent if $i\neq i'$. For a collection of not necessarily distinct indices $(i_1,\ldots,i_{2m})~(m\geq 1)$ such that $|\{i_1,\ldots,i_{2m}\}|=r\geq 1$, consider the product
\begin{equation*}
    |Z_{i_1}^\top A_1Z_{i_2}|^{w_1}|Z_{i_3}^\top A_2Z_{i_4}|^{w_2}\cdots|Z_{i_{2m-1}}^\top A_m Z_{i_{2m}}|^{w_m},
\end{equation*}
for some non-negative reals $w_1,\ldots,w_m$. We construct the \emph{weighted graph} with $r$ nodes where $j$th node corresponds to index set $N_j:=\set{k:i_k=j}$ for $j\in \{i_1,\ldots,i_{2m}\}$. We connect two nodes $N_j$ and $N_{j'}$ whenever $k\in N_j$ and $k+1\in N_{j'}$ for some $1\leq k \leq m$, i.e. $|Z_j^\top A_* Z_{j'}|^{w_*}$ appears in the product. Precisely, we define the followings.
\begin{itemize}
    \item If $j=j'$, we call this connection a \emph{self loop} of weight $w$, corresponding to the term $|Z_j^\top A_* Z_j|^w$ in the product.
    \item If $j\neq j'$, we call this connection an \emph{edge} of weight $w$, corresponding to the term $|Z_j^\top A_* Z_{j'}|^w$ in the product.
\end{itemize}
We let $l_j$ and $(l_w)_j$ be the number of self loops at $N_j$ and the sum of weights of these self loops, respectively. Let $e_j$ be the number of edges whose one end is $N_j$ and let $(e_w)_j$ be the sum of weights of these edges. Moreover, for $j\neq j'$, we denote $e_{jj'}$ as the number of edges connecting $N_j$ and $N_{j'}$, and $(e_w)_{jj'}$ be the sum of weights of these edges. Then it follows from the construction that
\begin{equation*}
    m=\sum_{j=1}^r l_j + \frac{1}{2}\sum_{j=1}^r e_j
    \quad\mbox{and}\quad
    \sum_{i=1}^{m} w_i = \sum_{j=1}^r (l_w)_j + \frac{1}{2}\sum_{j=1}^r (e_w)_j.
\end{equation*}
For some $\theta_i\in\Sb^{d-1}$ and non-negative $s_i$ $(1\leq i\leq 2m)$, define
\begin{equation*}
    h(\Gc) = |Z_{i_1}^\top A_1Z_{i_2}|^{w_1}|Z_{i_3}^\top A_2Z_{i_4}|^{w_2}\cdots|Z_{i_{2m-1}}^\top A_m Z_{i_{2m}}|^{w_m}\prod_{r=1}^{2m}|\theta_i^\top Z_{i_r}|^{s_r}.
\end{equation*}

\begin{theorem}\label{thm:C.1.4}
Let $W = \sum_{i=1}^m w_i$, $S = \sum_{r=1}^{2m}s_r$ and $E_w = \frac{1}{2}\sum_{j=1}^r (e_w)_j$. If $\Eb|h(\Gc)|<\infty$, then
\begin{equation*}
    \sup_{\substack{\theta_1,\ldots,\theta_{2m}\in\Sb^{d-1}\\A_1,\ldots,A_m\in\Ac}}\Eb h(\Gc)\leq \kappa_z^{2W+S}\,\sigma^{S/2}\,\tau^{W-E_w/2}.
\end{equation*}
Therefore, if $w_1=\ldots=w_m=1$, we have
\begin{equation*}
\sup_{\substack{\theta_1,\ldots,\theta_{2m}\in\Sb^{d-1}\\A_1,\ldots,A_m\in\Ac}}\Eb h(\Gc)\leq \kappa_z^{2m+S}\,\sigma^{S/2}\,\tau^{m-E/2},
\end{equation*}
where $E = \frac{1}{2}\sum_{j=1}^r e_j$ is the number of edges in $\Gc$.
\end{theorem}

\begin{proof}
Proof proceeds via induction on $r$. If $r=1$, the result is immediate from Lemma~\ref{lem:C1}. Now consider the case $r\geq2$ and we let $\set{i_1,\ldots,i_{2m}}=\set{1,\ldots,r}$, without loss of generality. From the law of conditional expectation, we have
\begin{equation*}
\Eb h(\Gc)=\Eb\bigg[\Eb\big[|Z_{i_1}^\top A_1Z_{i_2}|^{w_1}|Z_{i_3}^\top A_2Z_{i_4}|^{w_2}\cdots|Z_{i_{2m-1}}^\top A_m Z_{i_{2m}}|^{w_m}\prod_{r=1}^{2m}|\theta_i^\top Z_{i_r}|^{s_r}\mid Z_2,\ldots,Z_r\big]\bigg].
\end{equation*}
If $e_1=0$, i.e., $N_1$ has no edges, then
\begin{equation}
    \Eb h(\Gc) = \Eb\big[|Z_1^\top A_{\iota_1}Z_1|^{w_{\iota_1}}\cdots |Z_1^\top A_{{\iota_{l_1}}}Z_1|^{w_{\iota_{l_1}}}\prod_{j:i_j=1}|\theta_j^\top Z_1|^{s_j}\big]\Eb h(\Gc^{-N_1}),
\end{equation}
where $\Gc^{-N_1}$ denotes the weighted graph after removing the node $N_1$ from $\Gc$. Note that the sum of weights of edges in $\Gc^{-N_1}$ is still $(e_w)$. From Lemma~\ref{lem:C1} and the induction hypothesis, we have
\begin{align*}
    \Eb\big|h(\Gc)\big|
    &\leq \kappa_z^{2(l_w)_1+\sum_{j:i_j=1}s_j}\,\sigma^{\frac12\sum_{j:i_j=1}s_j}\,\tau^{(l_w)_1}\times \kappa_z^{2(W-(l_w)_1)+\sum_{r:i_r\neq 1}s_r}\,\sigma^{\frac12\sum_{r:i_r\neq 1}s_r}\,\tau^{(W-(l_w)_1)-(e_w)/2}\\
    &=\kappa_z^{2W+S}\,\sigma^{S/2}\,\tau^{W-(e_w)/2}.
\end{align*}
Suppose $e_1\geq 1$ and let $N_1$ be connected to $N_2,\ldots,N_{q}$, without loss of generality. Then $e_{12}+\ldots+e_{1q}=e_1$ and
\begin{align}\label{eq:recursive_formula1}
\Eb h(\Gc) = \Eb\bigg[\Eb\big[|Z_1^\top A_{\iota_1}Z_1|^{w_{\iota_1}}\cdots |Z_1^\top A_{{\iota_{l1}}}Z_1|^{w_{\iota_{l1}}}\prod_{j:i_j=1}|\theta_j^\top Z_1|^{s_j}\nonumber\\
\times\prod_{j=2}^q\prod_{k=1}^{e_{1j}}|Z_1^\top A_{1j_{k}} Z_j|^{w_{1j_k}}~\big|~ Z_2,\ldots,Z_q\big]\Eb h(\Gc^{-N_1})\bigg].
\end{align}
Here, $\Gc^{-N_1}$ denotes the graph after removing node $N_1$ from $\Gc$ (removing all self loops and edges involving node $N_1$). Since $|ZA_*Z'|\leq \norm{Z}_2\norm{Z'}_2$ for $A_*\in\Ac$, we have
\begin{align}\label{eq:recursive_formula2}
&\Eb\bigg[|Z_1^\top A_{\iota_1}Z_1|^{w_{\iota_1}}\cdots |Z_1^\top A_{{\iota_{l_1}}}Z_1|^{w_{\iota_{l_1}}}\prod_{j:i_j=1}|\theta_j^\top Z_1|^{s_j}\nonumber\\
&\qquad\qquad\times \prod_{j=2}^q\prod_{k=1}^{e_{1j}}|Z_1^\top A_{1j_{k}} Z_j|^{w_{1j_k}}~\big|~ Z_2,\ldots,Z_q\bigg]\nonumber\\
&\leq \norm{Z_2}_2^{(e_w)_{12}}\cdots \norm{Z_q}_2^{(e_w)_{1q}}\Eb\big[\norm{Z_1}_2^{2(l_w)_1}\prod_{j:i_j=1}|\theta_j^\top Z_1|^{s_j}\big]\nonumber\\
&\leq \norm{Z_2}_2^{(e_w)_{12}}\cdots \norm{Z_q}_2^{(e_w)_{1q}}\kappa_z^{2(l_w)_1+\sum_{j:i_j=1}s_j}\,\sigma^{\frac12\sum_{j:i_j=1}s_j}\,\tau^{(l_w)_1}.
\end{align}
We consider a new weighted graph $\Gc'$ from $\Gc^{-N_1}$ by adding a self loop of weight $(e_w)_{1j}/2$ to $N_j$ for each $j=2,\ldots,q$. From \eqref{eq:recursive_formula1} and \eqref{eq:recursive_formula2}, we get
\begin{equation}\label{eq:recursive_formula3}
\Eb h(\Gc)\leq \kappa_z^{2(l_w)_1+\sum_{j:i_j=1}s_j}\,\sigma^{\frac12\sum_{j:i_j=1}s_j}\,\tau^{(l_w)_1} \Eb h(\Gc').
\end{equation}
Furthermore, it follows from the construction of $\Gc'$ that $\Gc'$ has $(r-1)$ nodes $N_2,\ldots,N_r$, and
\begin{align*}
    &\sum_{j=2}^r(l_w')_j=\frac{1}{2}\sum_{j=2}^q(e_w)_{1j}+\sum_{j=2}^r(l_w)_j=\frac{(e_w)_1}{2}+\sum_{j=2}^r(l_w)_j,\\
    &\sum_{j=2}^r(e_w')_j=\Big(\sum_{j=2}^q(e_w)_j-(e_w)_{1j}\Big)+\sum_{j=q+1}^{r}(e_w)_j=-(e_w)_1+\sum_{j=2}^r(e_w)_j.
\end{align*}
Therefore, the induction hypothesis applies to $\Gc'$, and combining this with \eqref{eq:recursive_formula3} gives
\begin{equation*}
\Eb h(\Gc)\leq \kappa_z^{2\sum_{j=1}^r(l_w)_j+\sum_{j=1}^r(e_w)_j+S}\,\sigma^{S/2}\,\tau^{\sum_{j=1}^r(l_w)_j+\frac{1}{4}\sum_{j=1}^r(e_w)_j}
=\kappa_z^{2W+S}\,\sigma^{S/2}\,\tau^{W-E_w/2}.
\end{equation*}
This completes the proof.
\end{proof}

\section{Useful Lemmas and Propositions}

\begin{lemma}[Moment Bounds for Bounded Multi-linear Form]\label{lem:varbound} For any $k\geq 1$ and $T_k\in\Lc^k(\Bb)$, let $\xv_1,\ldots,\xv_k$ be independent random variables in $\Bb$. For $p\geq 1$,
\begin{equation*}
    \Eb^{1/p}\big| T_k[\xv_1,\ldots,\xv_k]\big|^p\leq \norm{T_k}_{\rm S}\sup_{\norm{u}_*\leq 1}\Eb^{1/p}[|\ip{\xv_r,u}|^p]\prod_{j\neq r}\Eb^{1/p}[\norm{\xv_j}^p].
\end{equation*} Therefore, ${\rm Var}\left( T_k[W_1-\theta,\ldots,W_k-\theta]\right)\leq \norm{T_k}_{\rm S}^2\nu V^{k-1}.$
\end{lemma}
\begin{proof}[proof of Lemma~\ref{lem:varbound}] See Lemma~6.1. of \cite{KoltchinskiiLi2026FunctionalEstimation}.
    
\end{proof}

\begin{lemma}[Distribution Approximation under Conditional Perturbation]\label{lem:BElemma}
    Consider random sequences $\Av_n,\Bv_n,\Cv_n$ such that $\Av_n = \Bv_n +\Cv_n+\Dv_n$ where $\Bv_n$ admits the distribution approximation, that is, for some c.d.f $F:(-\infty,\infty)\to[0,1]$,
    \begin{equation*}
        \sup_{t\in\Real}\big|\Pb(\Bv_n\leq t)-F(t)\big|\leq {\rm BE}_n.
    \end{equation*} Define the anti-concentration constant of $F$ as
    \begin{equation*}
        \Xi_{F}:=\sup_{t\in\Real}\sup_{\varepsilon>0}\frac{|F(t)-F(t-\varepsilon)|}{\varepsilon}
    \end{equation*}Furthermore, suppose that there exist sigma algebras $\Fc_n,\Fc_n'$, $\Fc_n$-measurable event $\Ec_n$, and $\Fc_n'$-measurable event $\Ec_n'$ such that
    \begin{align*}
        |\Eb[\Cv_n|\Fc_n]\mathbbm{1}(\Ec_n)|\leq E_n,\quad |{\rm Var}(\Cv_n|\Fc_n)\mathbbm{1}(\Ec_n)|\leq V_n,\\
        |\Eb[\Dv_n|\Fc_n']\mathbbm{1}(\Ec_n')|\leq E_n',\quad |{\rm Var}(\Dv_n|\Fc_n')\mathbbm{1}(\Ec_n')|\leq V_n',
    \end{align*} for some non-stochastic sequence $E_n,E_n',V_n,V_n'>0$. Then
    \begin{equation*}
        \sup_{t\in\Real}|\Pb(\Av_n\leq t)-F(t)|\leq {\rm BE}_n + \sqrt[3]{\frac{\Xi_F^2V_n}{4}}+\Xi_FE_n+\Pb(\Ec_n^\complement) + \sqrt[3]{\frac{\Xi_F^2V_n'}{4}}+\Xi_FE_n'+\Pb((\Ec_n')^\complement).
    \end{equation*} Therefore, if $F$ is a cdf of standard normal distribution where $\Xi_F \leq 1/\sqrt{2\pi}$, then we have
    \begin{equation*}
        \sup_{t\in\Real}\big|\Pb(\Av_n\leq t)-\Phi(t)\big|\leq {\rm BE}_n + \sqrt[3]{\frac{V_n}{8\pi}}+\frac{E_n}{\sqrt{2\pi}} + \sqrt[3]{\frac{V_n'}{8\pi}}+\frac{E_n'}{\sqrt{2\pi}}+\Pb(\Ec_n^\complement)+\Pb((\Ec_n')^\complement).
    \end{equation*}
\end{lemma}

\begin{proof}
    Since $\Av_n=\Bv_n+\Cv_n+\Dv_n$, we have for any $t\in\Real$ and $\varepsilon,\varepsilon'>0$ that
\begin{align}\label{eq:ImplyingBE:1}
    \Pb(\Av_n> t)&\leq \Pb(\Bv_n> t-\varepsilon-\varepsilon')+\Pb(\Cv_n> \varepsilon)+\Pb(\Dv_n> \varepsilon')\nonumber
    \\
    &\leq 1-F(t-\varepsilon-\varepsilon') + {\rm BE}_n + \Pb(\Cv_n> \varepsilon)+\Pb(\Dv_n> \varepsilon')\nonumber\\
    &\leq 1-F(t) + \Xi_F(\varepsilon+\varepsilon')+ {\rm BE}_n+  \Pb(\Cv_n> \varepsilon)+\Pb(\Dv_n> \varepsilon').
\end{align} Similarly, we can deduce that
\begin{align*}
    \Pb(\Av_n\leq t)&\leq \Pb(\Bv_n\leq t+\varepsilon+\varepsilon')+\Pb(\Cv_n\leq-\varepsilon) + \Pb(\Dv_n\leq-\varepsilon') \\
    &\leq F(t+\varepsilon+\varepsilon') + {\rm BE}_n + \Pb(\Cv_n\leq-\varepsilon)+ \Pb(\Dv_n\leq-\varepsilon')\\
    &\leq F(t) +\Xi_F(\varepsilon +\varepsilon' )+ {\rm BE}_n + \Pb(\Cv_n\leq-\varepsilon)+ \Pb(\Dv_n\leq-\varepsilon').
\end{align*} Combining this with \eqref{eq:ImplyingBE:1} yields
\begin{equation}\label{eq:ImplyingBE:2}
    \sup_{t\in\Real}|\Pb(\Av_n\leq t)-F(t)|\leq {\rm BE}_n + \Set{\Pb(|\Cv_n|\geq \varepsilon)+\Xi_F\varepsilon}+\Set{\Pb(|\Dv_n|\geq \varepsilon')+\Xi_F\varepsilon'},\quad\forall\varepsilon,\varepsilon'>0.
\end{equation} Meanwhile, from the law of total probability and the conditional Chebyshev inequality, we get that for any $\varepsilon>0$,
\begin{align*}
    &\Pb(|\Cv_n|\geq \varepsilon)=\Eb[\Pb(|\Cv_n|\geq \varepsilon\mid\Ec_n)]\\
    &\quad=\Eb[\Pb(|\Cv_n|\geq \varepsilon\mid\Ec_n)\mathbbm{1}(\Ec_n)]+\Eb[\Pb(|\Cv_n|\geq \varepsilon\mid\Ec_n)\mathbbm{1}(\Ec_n^\complement)]\\
    &\qquad\leq \Eb\left[\frac{{\rm Var}(\Cv_n\mid\Ec_n)}{(\varepsilon-|\Eb[\Cv_n\mid\Ec_n]|)_+^2}\mathbbm{1}(\Ec_n)\right] + \Pb(\Ec_n^\complement)\leq \frac{V_n}{(\varepsilon-E_n)_+^2}+\Pb(\Ec_n^\complement).
\end{align*} Similarly,
\begin{equation*}
    \Pb(|\Dv_n|\geq \varepsilon')\leq \frac{V_n'}{(\varepsilon'-E_n')_+^2}+\Pb((\Ec_n')^\complement)
\end{equation*}
Combining this with \eqref{eq:ImplyingBE:2} implies that
\begin{equation*}
    \sup_{t\in\Real}|\Pb(\Av_n\leq t)-F(t)|\leq {\rm BE}_n + \frac{V_n}{(\varepsilon-E_n)_+^2}+\Xi_F\varepsilon+\Pb(\Ec_n^\complement)+\frac{V_n'}{(\varepsilon'-E_n')_+^2}+\Xi_F\varepsilon'+\Pb((\Ec_n')^\complement)\quad\forall\varepsilon,\varepsilon'>0.
\end{equation*}Note that
\begin{align*}
    \inf_{\varepsilon>0}\Set{\frac{V_n}{(\varepsilon-E_n)_+^2}+\Xi_F\varepsilon}=\inf_{\delta>0}\Set{\frac{V_n}{\delta^2}+\Xi_F\delta+\Xi_FE_n}\\
    =\inf_{\delta>0}\Set{\frac{V_n}{\delta^2}+\frac{\Xi_F\delta}{2}+\frac{\Xi_F\delta}{2}+\Xi_FE_n}= \sqrt[3]{\frac{\Xi_F^2V_n}{4}}+\Xi_FE_n.
\end{align*} The last equality follows from AM-GM inequality where the equality can be attained at $\delta = (2V_n/\Xi_F)^{1/3}$, i.e., $\varepsilon = E_n + (2V_n/\Xi_F)^{1/3}$. Also, taking $\varepsilon' = E_n'+(2V_n'/\Xi_F)^{1/3}$ leads to the desired results.
\end{proof}

\begin{proposition}[Concentration Inequality for Sample Gram matrix]\label{prop:samcov}
    Consider i.i.d random vectors $Z_1,\ldots,Z_n\in\Real^d$ such that $\Eb[Z_1Z_1^\top]=\Sigma_Z$ and
    \begin{equation*}
        \kappa_p:=\sup_{v\in\Real^d}\frac{\Eb^{1/p}[|v^\top Z_1|^p]}{\Eb^{1/2}[|v^\top Z_1|^2]}<\infty,\quad p\in[2,q],
    \end{equation*} for some $q> 4.$ Define $\hat\Sigma_Z = n^{-1}\sum_{i=1}^nZ_iZ_i^\top$, then there exist a constant $c>0$ and a constant $\Ck = \Ck(\kappa_q,q)>0$ such that if $r(\Sigma_Z)/n\leq c$ then
    \begin{equation*}
        \Pb\bigg(\norm{\hat\Sigma_Z-\Sigma_Z}_{\rm op}\geq \Ck\norm{\Sigma_Z}_{\rm op}\bigg[\sqrt{\frac{r(\Sigma_Z)+\log(2/\delta)}{n}}+\frac{r(\Sigma_Z)}{n^{1-2/q}\delta^{2/q}}\bigg]\bigg)\leq\delta,
    \end{equation*} for all $\delta>0$, where $r(\Sigma_Z) = {\rm tr}(\Sigma_Z)/\norm{\Sigma_Z}_{\rm op}$ denotes the effective rank of $\Sigma_Z$.
\end{proposition}
\begin{proof} Note that $\norm{\hat\Sigma_Z-\Sigma_Z}_{\rm op} = \max\{\lambda_{\max}(\hat\Sigma_Z-\Sigma_Z), \lambda_{\max}(\Sigma_Z-\hat\Sigma_Z)\}$. First, we control $\lambda_{\max}(\hat\Sigma_Z-\Sigma_Z)$ using Theorem~2.3 of \cite{JirakMinskerShenWahl2025HeavyTailedMatrices}. In order to apply their results, we first introduce some notation. Define $g_u(z) = u^\top (zz^\top-\Sigma_Z)u/n$, then
    \begin{equation*}
        \lambda_{\rm max}(\hat\Sigma_Z-\Sigma_Z)=\sup_{u\in\Sb^{d-1}}\sum_{i=1}^n g_u(Z_i). 
    \end{equation*} Let $M = \max_{i\in[n]} \sup_{u\in\Sb^{d-1}}|g_u(z)|$, then we have
    \begin{equation*}
        M= \frac{1}{n}\max_{i\in[n]}\sup_{u\in\Sb^{d-1}}|u^\top(Z_iZ_i^\top-\Sigma_Z)u|\leq \frac{1}{n}\big(\max_{i\in[n]}\norm{Z_i}^2_2+\norm{\Sigma_Z}_{\rm op}\big).
    \end{equation*} This implies that $\Eb^{1/p} M^p \leq \{\Eb^{1/p}[\max_{i\in[n]}\norm{Z_i}_2^{2p}]+\norm{\Sigma_Z}_{\rm op}\}/n$ for $p\in[1,q/2]$, and we further note from Minkowski's inequality that
    \begin{align*}
        \Eb^{1/p}\norm{Z_1}_2^{2p}\leq \sum_{j=1}^d \Eb^{1/p}|\ev_j^\top Z_1|^{2p}\leq \kappa_{2p}^2 \sum_{j=1}^d\Eb[|\ev_j^\top Z_1|^2]=\kappa_{2p}^2{\rm tr}(\Sigma_Z).
    \end{align*} This implies that
    \begin{equation*}
        \Eb^{1/p} M^p \leq \{n^{1/p}\Eb^{1/p}[\norm{Z_1}_2^{2p}]+\norm{\Sigma_Z}_{\rm op}\}/n\leq \frac{2\kappa_{2p}^2{\rm tr}(\Sigma_Z)}{n^{1-1/p}},
    \end{equation*} for $p\in[1,q/2]$. Also, denote $\sigma_*^2=n\sup_{u\in\Sb^{d-1}}\Eb[g_u(Z_1)^2]$, then
    \begin{align*}
        \sigma_*^2=n^{-1}\sup_{u\in\Sb^{d-1}}{\rm Var}((u^\top Z_1)^2)\leq n^{-1}\sup_{u\in\Sb^{d-1}}\Eb[(u^\top Z_1)^4]\\
        \leq n^{-1}\kappa_4^4\sup_{u\in\Sb^{d-1}}\big(\Eb[(u^\top Z_1)^2]\big)^2=n^{-1}\kappa_4^4\norm{\Sigma_Z}_{\rm op}^2.
    \end{align*} Now, an application of Theorem 2.3 of \cite{JirakMinskerShenWahl2025HeavyTailedMatrices} yields that there exists a universal constant $\Ck>0$ such that
    \begin{align*}
        \Pb\bigg(\lambda_{\rm max}(\hat\Sigma_Z-\Sigma_Z)\geq 24 \Eb\lambda_{\rm max}(\hat\Sigma_Z-\Sigma_Z) + 24t\bigg)\leq \\\Ck\Bigg(\exp\bigg(-\frac{t^2}{2\sigma_*^2+64(\Eb M)t}\bigg)+\Pb(M\geq t)+\bigg(\frac{q/2}{\log(eq/2)}\bigg)^q\frac{(\Eb M^{q/2})^2}{t^q}\Bigg),
    \end{align*} for any $t\geq \sigma_*\sqrt{2}$. For any $\delta\in(0,1)$, we set
    \begin{equation*}
        t_\delta = 64\Eb M \log(2/\delta) + \sqrt{4\sigma_*^2\log(2/\delta)}+2^{1/q}\left(\frac{q/2}{\log(eq/2)}\right)\frac{(\Eb M^{q/2})^{2/q}}{\delta^{2/q}}.
    \end{equation*} This choice implies $\Pb(\lambda_{\rm max}(\hat\Sigma_Z-\Sigma_Z)\geq 24 \Eb\lambda_{\rm max}(\hat\Sigma_Z-\Sigma_Z) + 24t_\delta)\leq\delta$, and we further note that 
    \begin{align*}
        t_\delta &\leq \frac{128 \kappa_q^2{\rm tr}(\Sigma_Z)\log(2/\delta)}{n^{1-2/q}}+2\kappa_4^2\norm{\Sigma_Z}_{\rm op}\sqrt{\frac{\log(2/\delta)}{n}} + 2^{1/q}\left(\frac{q/2}{\log(eq/2)}\right)\frac{2\kappa_q^2{\rm tr}(\Sigma_Z)}{n^{1-2/q}\delta^{2/q}}\\
        &\leq 2\kappa_4^2\norm{\Sigma_Z}_{\rm op}\sqrt{\frac{\log(2/\delta)}{n}} + c_q\kappa_q^2\frac{{\rm tr}(\Sigma_Z)}{n^{1-2/q}\delta^{2/q}}.
    \end{align*} By noting that $\log(2/\delta)\leq q 2^{2/q}/(2e\delta^{2/q})$ for all $\delta\in(0,1)$, a constant $c_q$ can be taken as
    \begin{equation*}
        c_q = 128\times \frac{q}{2e}2^{2/q}+ \frac{q}{\log(eq/2)}2^{1/q}< 50q,
    \end{equation*} Next, in order to control $\Eb\lambda_{\rm max}(\hat\Sigma_Z-\Sigma_Z)$, we use Thereom~3.1 of \cite{JirakMinskerShenWahl2025HeavyTailedMatrices} which shows that as long as $r(\Sigma)/n\leq c$ for some universal constant $c>0$,
    \begin{align*}
        \Eb\lambda_{\rm max}(\hat\Sigma_Z-\Sigma_Z)\leq \Ck_{\kappa_q,q} \bigg(\norm{\Sigma_Z}_{\rm op}\sqrt{\frac{r(\Sigma_Z)}{n}} + \frac{\Eb^{2/q}\max_{i\in[n]}\norm{Z_i}_2^q}{n}\bigg)\\
        \leq\Ck'_{\kappa_q,q} \bigg(\norm{\Sigma_Z}_{\rm op}\sqrt{\frac{r(\Sigma_Z)}{n}} + \frac{{\rm tr}(\Sigma_Z)}{n^{1-q/2}}\bigg)=\Ck'_{\kappa_q,q} \norm{\Sigma_Z}_{\rm op}\bigg(\sqrt{\frac{r(\Sigma_Z)}{n}} + \frac{r(\Sigma_Z)}{n^{1-q/2}}\bigg).
    \end{align*} This proves the probabilistic bound for $\lambda_{\max}(\hat\Sigma_Z-\Sigma_Z)$. The bound for $\lambda_{\max}(\Sigma_Z-\hat\Sigma_Z)$ similarly follows by setting $h_u(z) =-g_u(z)=-u^\top (zz^\top-\Sigma_Z)u/n$, so that
    $$
    \lambda_{\rm max}(\Sigma_Z -\hat\Sigma_Z) = \sup_{u\in\Sb^{d-1}}\sum_{i=1}^n h_u(Z_i).
    $$ We omit the further details.
\end{proof}
\begin{proposition}[Lower Bound on the Concentration of Sample Gram matrix]\label{prop:sampcov.opt}
Fix $n\ge2$, $d\ge2$, $q\ge 4$, and a matrix $0\preceq\Sigma\in\Real^{d\times d}$.
Let $r(\Sigma):=\tr(\Sigma)/\|\Sigma\|_{\rm op}$.
If $\frac{r(\Sigma)}{n^{1-2/q}} \ \ge\ \frac{10}{\sqrt n},$ then there exists a distribution $P=P_{n,d,q,\Sigma}$ on $\mathbb R^d$ such that for $X\sim P$:
\begin{enumerate}
\item $\mathbb E[XX^\top]=\Sigma$.
\item There exists $K(q)\in(0,\infty)$ depending only on $q$ such that
\[
\sup_{u\in\mathbb S^{d-1}}\ \frac{\Eb^{1/q}[|u^\top X|^q]}{\Eb^{1/2}[|u^\top X|^2]}\ \le\ K(q).
\]
\item If $X_1,\dots,X_n\overset{iid}{\sim}P$ and $\hat\Sigma=n^{-1}\sum_{i=1}^n X_iX_i^\top$, then
\[
\mathbb P\!\left(\|\hat\Sigma-\Sigma\|_{\rm op}\ge \frac{1}{10}\,\|\Sigma\|_{\rm op}\,\frac{r(\Sigma)}{n^{1-2/q}}\right)\ \ge\ \frac{1}{12}.
\]
\end{enumerate}
\end{proposition}
\begin{remark} For $X\sim\Nc(0,\Sigma)$, \cite{KoltchinskiiLounici2017Concentration} shows that
\begin{equation*}
    \Eb\norm{\hat\Sigma-\Sigma}_{\rm op}\asymp \norm{\Sigma}_{\rm op}\bigg(\sqrt{\frac{r(\Sigma)}{n}}+\frac{r(\Sigma)}{n}\bigg).
\end{equation*} Therefore, the rate $\norm{\Sigma}_{\rm op}\sqrt{r(\Sigma)/n}$ is also genuinely unavoidable as long as $n\gtrsim r(\Sigma)$, and thus the rate in Proposition~\eqref{prop:samcov} is exact up to constants and cannot be significantly improved.
\end{remark}
\begin{proof}
Write the spectral decomposition $\Sigma=U\Lambda U^\top$ where
$\Lambda=\mathrm{diag}(\lambda_1,\dots,\lambda_d)$ and $\lambda_1=\|\Sigma\|_{\rm op}$.
Let $s=(s_1,\dots,s_d)^\top$ have i.i.d.\ Rademacher coordinates, $\mathbb P(s_j=\pm1)=1/2$.
Define
\[
Z:=U\Lambda^{1/2}s.
\]
Then $\mathbb E[ZZ^\top]=U\Lambda^{1/2}\mathbb E[ss^\top]\Lambda^{1/2}U^\top=\Sigma$.
Moreover, we have deterministic identity
\begin{equation}\label{eq:1}
    \|Z\|_2^2 = s^\top\Lambda s = \sum_{j=1}^d \lambda_j s_j^2 = \sum_{j=1}^d \lambda_j = \tr(\Sigma)
\end{equation} Define
\[
p:=\frac{1}{8n},\qquad B:=n^{1/q},\qquad A^2:=\frac{1-pB^2}{1-p},
\]
and let $R$ be independent of $Z$ with $\mathbb P(R=B)=p$ and $\mathbb P(R=A)=1-p$.
Define $X:=RZ$ and let $P$ be the law of $X$.
Since $q\ge4$, $B^2=n^{2/q}\le n^{1/2}$ and thus $pB^2\le (8n)^{-1}n^{1/2}\le 1/8$,
so $A^2>0$ is well-defined.

\smallskip
\noindent\textbf{Covariance}
Because $R\perp\!\!\!\perp Z$, we have $\mathbb E[XX^\top]=\mathbb E[R^2]\ \mathbb E[ZZ^\top]=\mathbb E[R^2]\ \Sigma.$ From the definition of $A^2$,
\[
\mathbb E[R^2]=pB^2+(1-p)A^2=pB^2+(1-p)\frac{1-pB^2}{1-p}=1,
\]
hence $\mathbb E[XX^\top]=\Sigma$.

\smallskip
\noindent\textbf{Moment equivalence}
Fix $u\in\mathbb S^{d-1}$ and set $a:=\Lambda^{1/2}U^\top u\in\mathbb R^d$.
Then $u^\top Z = a^\top s$ and $\|a\|_2^2=u^\top\Sigma u$.
Hoeffding's inequality gives $\mathbb P(|a^\top s|\ge t)\le 2\exp(-t^2/(2\|a\|_2^2))$ for all $t\ge0$.
Using $\mathbb E|W|^q=q\int_0^\infty t^{q-1}\mathbb P(|W|\ge t)\,dt$ and the change of variables
$t=\|a\|_2\sqrt{2x}$, we obtain
\[
\|u^\top Z\|_{L_q}=\|a^\top s\|_{L_q}
\le \Big(q\,2^{q/2}\Gamma(q/2)\Big)^{1/q}\ \|a\|_2
=\Big(q\,2^{q/2}\Gamma(q/2)\Big)^{1/q}\ (u^\top\Sigma u)^{1/2}.
\]
Since $X=RZ$ and $R\perp\!\!\!\perp Z$,
\[
\|u^\top X\|_{L_q} = \|R\|_{L_q}\ \|u^\top Z\|_{L_q}.
\]
Also $B^q=n$ so $\mathbb E[R^q]=pB^q+(1-p)A^q\le pn+(1-p)^{-q/2}\le 1/8+2^{q/2}$,
since $n\ge2$ implies $p\le 1/16$ and thus $(1-p)^{-1}\le 2$.
Therefore
\[
\sup_{u\in\mathbb S^{d-1}} \frac{\|u^\top X\|_{L_q}}{(u^\top\Sigma u)^{1/2}}
\le \Big(q\,2^{q/2}\Gamma(q/2)\Big)^{1/q}\Big(\frac18+2^{q/2}\Big)^{1/q}
=:K(q).
\]

\smallskip
\noindent\textbf{Operator-norm deviation lower bound}
Let $(R_i,Z_i)_{i=1}^n$ be i.i.d.\ copies of $(R,Z)$, and define $X_i=R_iZ_i$ and
$\hat\Sigma=n^{-1}\sum_{i=1}^n X_iX_i^\top$.
Write $\zeta_i=\mathbf 1\{R_i=B\}$ and define disjoint events
\[
E_k:=\{\zeta_1=0,\dots,\zeta_{k-1}=0,\ \zeta_k=1\},\qquad k=1,\dots,n.
\]
Then $\bigcup_{k=1}^n E_k=\{\exists i:\zeta_i=1\}$ and
\begin{equation}\label{eq:2}
    \sum_{k=1}^n\mathbb P(E_k)=\Pb\Big(\bigcup_{k=1}^nE_k\Big)=1-(1-p)^n\ge 1-e^{-np}=1-e^{-1/8}\ge \frac19,
\end{equation}
where the last inequality uses $1-e^{-x}\ge x/(1+x)$ with $x=1/8$. We fix $k$, and on the event $E_k$ define the unit vector
\[
v_k:=\frac{Z_k}{\|Z_k\|_2}=\frac{Z_k}{\sqrt{\tr(\Sigma)}},
\]
where we used \eqref{eq:1}. On $E_k$, $R_k=B$, hence
\begin{equation}\label{eq:3}
    (v_k^\top X_k)^2 = (v_k^\top (BZ_k))^2 = B^2\|Z_k\|_2^2 = B^2\tr(\Sigma)= n^{2/q}\tr(\Sigma).
\end{equation}Let $Y_i:=(v_k^\top X_i)^2$ for $i\ne k$ and $S_{-k}:=\sum_{i\ne k}Y_i$, so
\begin{equation}\label{eq:4}
    v_k^\top \hat\Sigma v_k = \frac1n (v_k^\top X_k)^2 + \frac1n S_{-k}.
\end{equation}

Let $\mathcal G_k:=\sigma(R_1,\dots,R_k,Z_k)$ and $\mathcal H_k:=\sigma((R_i,Z_i):i>k)$.
By independence of i.i.d.\ blocks across $i$, $\mathcal H_k$ is independent of $\mathcal G_k$.
Since $E_k\in\sigma(R_1,\dots,R_k)\subset\mathcal G_k$ and $v_k$ is $\sigma(Z_k)\subset\mathcal G_k$-measurable,
$\mathcal H_k$ is independent of $\sigma(E_k,v_k)$.
Consequently, conditional on $(E_k,v_k)$, the variables $(X_i)_{i>k}$ remain independent with the same law as $X$.
Also, on $E_k$ we have $R_i=A$ deterministically for all $i<k$.

\smallskip
Using $Z_k=U\Lambda^{1/2}s^{(k)}$ and $s^{(k)}_j=\pm1$,
\begin{equation}\label{eq:5}
    v_k^\top \Sigma v_k
= \frac{Z_k^\top \Sigma Z_k}{Z_k^\top Z_k}
= \frac{s^{(k)\top}\Lambda^2 s^{(k)}}{s^{(k)\top}\Lambda s^{(k)}}
= \frac{\sum_{j=1}^d \lambda_j^2}{\sum_{j=1}^d \lambda_j}
=:m_\Sigma.
\end{equation}In particular, $m_\Sigma\le \|\Sigma\|_{\rm op}$.

\smallskip
For $i>k$, by $\mathbb E[R^2]=1$ and $\mathbb E[Z_iZ_i^\top]=\Sigma$,
\begin{equation}\label{eq:6}
    \mathbb E[Y_i\mid E_k,v_k]=\mathbb E[(v_k^\top X)^2\mid v_k]=v_k^\top\Sigma v_k = m_\Sigma.
\end{equation}For $i<k$, on $E_k$ we have $X_i=AZ_i$, hence $\mathbb E[Y_i\mid E_k,v_k]=A^2 m_\Sigma$.
Therefore,
\begin{equation}\label{eq:7}
    \mathbb E[S_{-k}\mid E_k,v_k] = (k-1)A^2 m_\Sigma + (n-k)m_\Sigma \ge (n-1)A^2 m_\Sigma.
\end{equation}
Next, we bound $\Var(Y_i\mid E_k,v_k)$ uniformly.
Condition on $v_k$ (equivalently on $Z_k$).
Then $v_k^\top Z_i$ is a Rademacher sum with coefficients of $\ell_2$-norm $\sqrt{m_\Sigma}$,
hence $\mathbb E[(v_k^\top Z_i)^4\mid v_k]\le 3 m_\Sigma^2$ (the standard Rademacher fourth-moment bound).
Also, since $q\ge4$, $B^4=n^{4/q}\le n$ and thus $pB^4\le 1/8$; and $A^2\le (1-p)^{-1}\le 2$ implies $A^4\le 4$.
Therefore $\mathbb E[R^4]=pB^4+(1-p)A^4\le 1/8+4\le 5$.
Hence for $i>k$,
\[
\mathbb E[Y_i^2\mid E_k,v_k]
=\mathbb E[R^4]\ \mathbb E[(v_k^\top Z_i)^4\mid v_k]
\le 5\cdot 3\, m_\Sigma^2 = 15 m_\Sigma^2,
\]
so $\Var(Y_i\mid E_k,v_k)\le 15m_\Sigma^2$.
For $i<k$, $Y_i=A^2(v_k^\top Z_i)^2$ and thus
$\Var(Y_i\mid E_k,v_k)\le \mathbb E[Y_i^2\mid E_k,v_k]\le 3A^4m_\Sigma^2\le 12m_\Sigma^2\le 15m_\Sigma^2$.
By conditional independence of $(Y_i)_{i\ne k}$ given $(E_k,v_k)$,
\begin{equation}\label{eq:8}
    \Var(S_{-k}\mid E_k,v_k)\le 15(n-1)m_\Sigma^2\le 15n\,m_\Sigma^2.
\end{equation} Define
\[
F_k:=\Big\{S_{-k}\ge (n-1)A^2m_\Sigma - 2m_\Sigma\sqrt{15n}\Big\}.
\] The conditional Chebyshev inequality and \eqref{eq:8} imply that
\begin{equation}\label{eq:9}
    \mathbb P(F_k\mid E_k,v_k)\ge \frac34,
\qquad\mbox{and hence}\qquad
\mathbb P(F_k\mid E_k)\ge \frac34.
\end{equation}On $E_k\cap F_k$, combining \eqref{eq:3}--\eqref{eq:5} and $v_k^\top \Sigma v_k=m_\Sigma$,
\begin{align*}
v_k^\top(\hat\Sigma-\Sigma)v_k
&= \frac1n(v_k^\top X_k)^2 + \frac1n S_{-k} - m_\Sigma\ge \frac{\tr(\Sigma)}{n^{1-2/q}} - \frac{m_\Sigma}{n} - (1-A^2)m_\Sigma - \frac{2\sqrt{15}\,m_\Sigma}{\sqrt n}.
\end{align*}
Using $m_\Sigma\le \|\Sigma\|_{\rm op}$ and $1/n\le 1/\sqrt n$, it remains to bound $1-A^2$:
\[
1-A^2 = \frac{p(B^2-1)}{1-p}\le \frac{pB^2}{1-p}\le 2pB^2=\frac14\,n^{2/q-1}\le \frac{1}{4\sqrt n},
\]
where the last inequality uses $q\ge4$.
Therefore, on $E_k\cap F_k$,
\begin{align*}\label{eq:10}
    \|\hat\Sigma-\Sigma\|_{\rm op}\ge v_k^\top(\hat\Sigma-\Sigma)v_k\ge \frac{\tr(\Sigma)}{n^{1-2/q}} - \frac{1+1/4+2\sqrt{15}}{\sqrt n}\,\|\Sigma\|_{\rm op}
\nonumber\\
\ge \frac{\tr(\Sigma)}{n^{1-2/q}} - \frac{9}{\sqrt n}\,\|\Sigma\|_{\rm op}\ge \frac{1}{10}\,\frac{\tr(\Sigma)}{n^{1-2/q}}
= \frac{1}{10}\,\|\Sigma\|_{\rm op}\,\frac{r(\Sigma)}{n^{1-2/q}}
\end{align*}
Hence
\[
\left\{\|\hat\Sigma-\Sigma\|_{\rm op}\ge \frac{1}{10}\,\|\Sigma\|_{\rm op}\,\frac{r(\Sigma)}{n^{1-2/q}}\right\}
\supset \bigcup_{k=1}^n (E_k\cap F_k).
\]
Since the events $E_k$ are disjoint, the events $E_k\cap F_k$ are disjoint, and thus by \eqref{eq:2} and \eqref{eq:9},
\[
\mathbb P\!\left(\|\hat\Sigma-\Sigma\|_{\rm op}\ge \frac{1}{10}\,\|\Sigma\|_{\rm op}\,\frac{r(\Sigma)}{n^{1-2/q}}\right)
\ge \sum_{k=1}^n \mathbb P(F_k\mid E_k)\mathbb P(E_k)
\ge \frac34\sum_{k=1}^n \mathbb P(E_k)
\ge\frac{1}{12}.
\]
This proves the claim.
\end{proof}

\begin{proposition}[Fuk--Nagaev inequality for the Euclidean sample mean]\label{prop:euclid_fuk_nagaev_mean}\label{prop:D5}
Let $X_1,\dots,X_n\sim X\in\mathbb{R}^d$ be i.i.d.\ with $\mathbb{E}X=0$ and covariance
$\Sigma:=\mathbb{E}[XX^\top]$. Fix $q>2$ and assume the directional $L_q$--$L_2$
equivalence: there exists $K\ge 1$ such that
\begin{equation}\label{eq:dir_Lq_L2_equiv}
  \Eb^{1/q}[|u^\top X|^q]\leq K\,(u^\top \Sigma u)^{1/2}
  \qquad \forall u\in\mathbb{S}^{d-1}.
\end{equation}
Then there exists a constant $C_q>0$ depending only on $q$ such that the following holds.
\begin{equation}\label{eq:prob_le_delta}
\mathbb{P}\!\left(
\left\|\frac1n\sum_{i=1}^n X_i\right\|_2
\ge \norm{\Sigma}_{\rm op}^{1/2}\bigg[\sqrt{\frac{4r(\Sigma)+3\log(2/\delta)}{n}}+C_q\frac{K r(\Sigma)^{1/2}}{\delta^{1/q}n^{1-1/q}}\bigg]
\right)\le \delta.
\end{equation}
\end{proposition}

\begin{proof}
Let $\Xc=(\mathbb{R}^d,\|\cdot\|_2)$ and set $S_n:=\sum_{i=1}^n X_i$. Apply
Theorem~3.1 of \cite{EinmahlLi2008LILBanach} with $s=q$, and choose the
parameters $\eta=1$ and $\delta=1$. The theorem yields that for every $t>0$,
\begin{equation}\label{eq:Sn_tail_from_EL_app}
\mathbb{P}\!\left(
\|S_n\|_2 \ge 2\,\mathbb{E}\|S_n\|_2 + t
\right)
\le
\exp\!\left(-\frac{t^2}{3\,\Lambda_n}\right)
+
C_q\,\frac{n\, \mathbb{E}\|W_1\|_2^q}{t^q}.
\end{equation}
where
\begin{equation*}\label{eq:Lambda_n_def}
\Lambda_n := \sup\left\{\sum_{j=1}^n \mathbb{E}\big[f(W_j)^2\big] : f\in \Xc_1^*\right\},
\end{equation*}
$\Xc_1^*$ denotes the unit ball of the dual space $\Xc^*$, and $C_q$ depends only on $q$. Every $f\in \Xc^*$ is of the form
$f_u(x)=u^\top x$ for some $u\in\mathbb{R}^d$, and $\|f_u\|_{\Xc^*}=\|u\|_2$.
Thus $\Xc_1^*=\{f_u:\|u\|_2\le 1\}$ and therefore
\[
\Lambda_n
=
\sup_{\|u\|_2\le 1}\sum_{j=1}^n \mathbb{E}\,(u^\top W_j)^2
=
\sup_{\|u\|_2\le 1} n\,u^\top \Sigma u
=
n\,\|\Sigma\|_{\rm op}.
\] 
By Cauchy--Schwarz, we have $\mathbb{E}\|S_n\|_2 \le \big(\mathbb{E}\|S_n\|_2^2\big)^{1/2}.$ Using independence and $\mathbb{E}X_i=0$,
\[
\mathbb{E}\|S_n\|_2^2
=
\mathbb{E}\,\mathrm{tr}(S_nS_n^\top)
=
\mathrm{tr}\big(\mathbb{E}[S_nS_n^\top]\big)
=
\mathrm{tr}\Big(\sum_{i=1}^n \mathbb{E}[X_iX_i^\top]\Big)
=
n\,\mathrm{tr}(\Sigma),
\]
hence $\mathbb{E}\|S_n\|_2 \le \sqrt{n\,\mathrm{tr}(\Sigma)}.$ Moreover, from Lemma~\ref{lem:C1}, we get $\mathbb{E}\|W_1\|_2^q \le K^q\,\tr(\Sigma)^{q/2}.$ Set $t=nx$ in \eqref{eq:Sn_tail_from_EL_app} to deduce
\[
\mathbb{P}\!\left(
\left\|\frac1n\sum_{i=1}^n X_i\right\|_2
\ge
2\sqrt{\frac{\mathrm{tr}(\Sigma)}{n}} + x
\right)
\le
\exp\!\left(-\frac{n x^2}{3\|\Sigma\|_{\rm op}}\right)
+
C_q\,\frac{K^q\,\mathrm{tr}(\Sigma)^{q/2}}{n^{q-1}x^q}.
\] Now, an elementary substitution completes the proof.
\end{proof}

\end{document}